\documentclass[a4paper,11pt]{article}
\usepackage[utf8x]{inputenc}

\usepackage{parskip}
\usepackage{a4wide}
\usepackage{amssymb}
\usepackage{amsmath}
\usepackage{amsthm}
\usepackage[mathscr]{euscript}
\usepackage[normalem]{ulem}
\usepackage[colorlinks=true]{hyperref}
\usepackage[usenames,dvipsnames]{xcolor}
\usepackage{appendix}
\usepackage{tikz}
\usepackage{cleveref}
\usetikzlibrary{arrows,angles}
\usetikzlibrary{matrix}
\usetikzlibrary{decorations}
\usetikzlibrary{arrows,calc,shapes,decorations.pathreplacing}
\usepackage{tikz-cd}
\usepackage{todonotes}

\hypersetup{hidelinks}

\numberwithin{equation}{section}

\theoremstyle{definition}

\newtheorem{Definition}{Definition}[section]

\newtheorem{Remark}[Definition]{Remark}
\newtheorem{Example}[Definition]{Example}

\theoremstyle{plain}

\newtheorem{Theorem}[Definition]{Theorem}
\newtheorem{Proposition}[Definition]{Proposition}
\newtheorem{Lemma}[Definition]{Lemma}
\newtheorem{Corollary}[Definition]{Corollary}

\newcommand{\R}{\mathbb R}
\newcommand{\N}{\mathbb N}

\usepackage[xcolor]{changebar}
\cbcolor{blue}
\newcommand{\Ric}{\mathrm{Ric}}

\usepackage[inline]{enumitem}
\newcommand{\enumlabelformat}{\roman}
\newcommand{\enumlabelfont}[1]{#1}
\newlength{\thelabelsep}
\setlist{labelsep=\thelabelsep}
\setlist[enumerate,1]{font=\enumlabelfont,label=(\enumlabelformat*),leftmargin=2.5em}
\setlist[itemize]{leftmargin=2.5em,label=$-$}

\newcounter{inlineenum}
\renewcommand{\theinlineenum}{\enumlabelformat{inlineenum}}

\let\epsilon\varepsilon

\makeatletter
\let\save@mathaccent\mathaccent
\newcommand*\if@single[3]{%
  \setbox0\hbox{${\mathaccent"0362{#1}}^H$}%
  \setbox2\hbox{${\mathaccent"0362{\kern0pt#1}}^H$}%
  \ifdim\ht0=\ht2 #3\else #2\fi
  }
\newcommand*\rel@kern[1]{\kern#1\dimexpr\macc@kerna}
\newcommand*\widebar[1]{\@ifnextchar^{{\wide@bar{#1}{0}}}{\wide@bar{#1}{1}}}
\newcommand*\wide@bar[2]{\if@single{#1}{\wide@bar@{#1}{#2}{1}}{\wide@bar@{#1}{#2}{2}}}
\newcommand*\wide@bar@[3]{%
  \begingroup
  \def\mathaccent##1##2{%
    \let\mathaccent\save@mathaccent
    \if#32 \let\macc@nucleus\first@char \fi
    \setbox\z@\hbox{$\macc@style{\macc@nucleus}_{}$}%
    \setbox\tw@\hbox{$\macc@style{\macc@nucleus}{}_{}$}%
    \dimen@\wd\tw@
    \advance\dimen@-\wd\z@
    \divide\dimen@ 3
    \@tempdima\wd\tw@
    \advance\@tempdima-\scriptspace
    \divide\@tempdima 10
    \advance\dimen@-\@tempdima
    \ifdim\dimen@>\z@ \dimen@0pt\fi
    \rel@kern{0.6}\kern-\dimen@
    \if#31
      \overline{\rel@kern{-0.6}\kern\dimen@\macc@nucleus\rel@kern{0.4}\kern\dimen@}%
      \advance\dimen@0.4\dimexpr\macc@kerna
      \let\final@kern#2%
      \ifdim\dimen@<\z@ \let\final@kern1\fi
      \if\final@kern1 \kern-\dimen@\fi
    \else
      \overline{\rel@kern{-0.6}\kern\dimen@#1}%
    \fi
  }%
  \macc@depth\@ne
  \let\math@bgroup\@empty \let\math@egroup\macc@set@skewchar
  \mathsurround\z@ \frozen@everymath{\mathgroup\macc@group\relax}%
  \macc@set@skewchar\relax
  \let\mathaccentV\macc@nested@a
  \if#31
    \macc@nested@a\relax111{#1}%
  \else
    \def\gobble@till@marker##1\endmarker{}%
    \futurelet\first@char\gobble@till@marker#1\endmarker
    \ifcat\noexpand\first@char A\else
      \def\first@char{}%
    \fi
    \macc@nested@a\relax111{\first@char}%
  \fi
  \endgroup
}
\makeatother

\newcommand\restr[2]{{
  \left.\kern-\nulldelimiterspace 
  #1 
  \vphantom{\big|} 
  \right|_{#2} 
  }}

\title{Timelike Ricci curvature lower bounds via optimal transport for Orlicz-type Lorentzian costs}

\author{Argam Ohanyan\thanks{{\tt argam.ohanyan@utoronto.ca}, Department of Mathematics, University of Toronto, 45 St.\ George Street, M5S 2E5 Toronto, Ontario, Canada.}\\Marta Sálamo Candal \thanks{{\tt marta.salamo.candal@univie.ac.at}, Faculty of Mathematics, University of Vienna, Oskar-Morgenstern-Platz 1, 1090 Vienna, Austria.}
}
\begin{document}

\date{\today}


\maketitle

\begin{abstract}

We study the optimal transport problem on globally hyperbolic spacetimes associated with Orlicz-type Lorentzian cost functions of the form $u \circ \ell$, where $u$ is a suitable monotonically increasing and concave function, and $\ell$ is the time separation. Our work encompasses and generalises the case $u(x) = u_p(x) = p^{-1}x^p$ for $p \in (0,1)$, as well as the more recent $p < 0$, which have been the only examples considered so far in the literature. A fundamental notion for our purposes is the property of $u$-separation for a pair of measures, which generalises McCann's $p$-separation and for which we are able to obtain strong duality to the full Orlicz-type optimisation problem. In our main results, we characterise timelike Ricci curvature lower bounds via the convexity of the relative entropy along geodesics arising from the Orlicz-type optimal transport with cost $u \circ \ell$, which is a far-reaching generalisation of McCann's seminal work in the case $u = u_p$, $p \in (0,1)$.

\vspace{1em}

\medskip

\end{abstract}
\tableofcontents

\section{Introduction}
\label{section: intro}

Orlicz spaces are a generalisation of $L^p$-spaces, allowing for more general convex functions $u$ (usually called \emph{Young functions} in the literature) in the definition of the norm, with $u(x):=x^p$ reproducing the well-known $L^p$-norm. We refer to the textbooks of Rao--Ren \cite{RaoRenTheory, RaoRenApplications} for a detailed treatment of this classical topic in analysis.

Just like in the case of $L^p$-spaces, one can study Wasserstein spaces where the distance arises from an Orlicz-type optimal transport problem, generalising the usual $L^p$-based problem in classical Wasserstein geometry. The definition is as follows: Given a compact metric space $X$ (for simplicity) and a suitable convex function $u:[0,\infty) \to [0,\infty)$, the Orlicz--Wasserstein distance $W_u:\mathcal{P}(X)^2 \to [0,\infty)$ is defined via the double infimisation problem
\begin{equation*}
    W_u(\mu,\nu):=\inf\left\{ \lambda > 0 : \inf_{\pi \in \Pi(\mu,\nu)}\int_{X^2} u\bigg( \frac{d(x,y)}{\lambda}\bigg) \, d\pi \leq 1 \right\},
\end{equation*}
see e.g.\ Kuwada \cite{KuwadaRIMS}. For $u = u_p = |\cdot|^p$, $p \geq 1$, this definition reproduces the usual $p$-Wasserstein distance $W_p$. An incomplete list of works on Orlicz--Wasserstein spaces includes works of Kuwada on gradient estimates \cite{KuwadaJfuncan, KuwadaRIMS}, Sturm \cite{SturmOrlicz} on more general classes of functions $u$, Lisini \cite{LisiniOrliczcurves} on absolutely continuous curves and their speed, as well as Kell \cite{kellphdthesis, kell2017interpolation} who studied the implication of entropic convexity along Orlicz--Wasserstein geodesics from Ricci curvature lower bounds in the spirit of Lott--Villani \cite{LottVillani:2009} and Sturm \cite{Sturm:2006a, Sturm:2006b}. Let us remark that \cite[Ch.\ 7]{kellphdthesis} provides a comprehensive and reader-friendly introduction to Orlicz--Wasserstein spaces.

As far as Lorentzian geometry and signature are concerned, methods of optimal transport have seen a significant rise in popularity over the last decade (see for example the works of Eckstein--Miller \cite{EM:17}, Suhr \cite{SuhrLorentziancost}, Kell--Suhr \cite{KellSuhrLorentziancost}). The seminal works of McCann \cite{McCann:2020} and Mondino--Suhr \cite{MS:22} have established the equivalent characterisation of timelike Ricci curvature lower (and, in the case of \cite{MS:22}, also upper) bounds via convexity properties of the Boltzmann--Shannon entropy, extended later to the Rényi entropy by Braun \cite{Braun:2023Renyi} and to Finsler spacetimes by Braun--Ohta \cite{BO:23}. This has led to a fruitful study of non-smooth Lorentzian spaces satisfying timelike Ricci curvature bounds, where powerful theorems from classical smooth spacetime geometry have been established in extremely low regularity (see e.g.\ Cavalletti--Mondino \cite{CM:20}, Braun \cite{Braun:2023Renyi, braun2024dAlembertian}, Braun--McCann \cite{BraunMcCann:2023}, Beran--Braun--Calisti--Gigli--McCann--Ohanyan--Rott--Sämann \cite{Octet}). So far, the majority of the literature involving optimal transport in the Lorentzian setting has concerned itself with studies of the $p$-Lorentz-Wasserstein distance
\begin{equation*}
    \ell_p(\mu,\nu):= \left( \sup_{\pi \in \Pi_\leq(\mu,\nu)} \int_{M^2} \ell^p \, d\pi \right)^{1/p} = u_p^{-1}\left( \sup_{\pi \in \Pi_\leq(\mu,\nu)} \int u_p \circ \ell \,  d \pi \right),
\end{equation*}
where $M$ is either a smooth or a nonsmooth spacetime, $\ell:M^2 \to \{-\infty\} \cup [0,\infty]$ is the (extended) time separation function, and $u_p(x):=x^p/p$ for $p \in (0,1]$ (the more recent works among those cited above have also considered $p < 0$, starting with \cite{Octet}). The reason for these choices of $u_p$ has to do with the concavity properties of the time separation function, much like the choices $u_p(x) = |x|^p$ for $p \geq 1$ mirror the convexity properties of the positive definite distance in metric and Riemannian geometry. One may thus conclude that the natural range of $p$ for Lorentzian $L^p$-theory is $p \in [-\infty,1]$, which has been formalised by Gigli \cite{gigli2025hyperbolic} in the context of hyperbolic Banach spaces. One may wonder whether $p = 0$ behaves like $p = 2$ in the positive definite case, i.e.\ a natural ``best" self-dual choice. However, as pointed out in \cite[Rem.\ 5.9]{gigli2025hyperbolic}, there are many equally valid self-dual concave functions $u$ besides $u_0 = \frac{1}{2} + \log$.

These observations are motivational for an Orlicz-type approach to Lorentzian optimal transport theory, both to remove the need for an (arbitrary) choice of $p$ and for a deeper understanding as to how far known results extend for general (suitable) concave functions $u$. The introduction and study of Orlicz-type optimal transportation problems on smooth globally hyperbolic spacetimes is one of our principal aims in the current work. Our main objective in this context is to extend McCann's \cite{McCann:2020} characterisation of timelike Ricci curvature lower bounds via entropic convexity to hold along $\ell_u$-geodesics of probability measures for a large class of $u$, which includes in particular all $u_p$, $p \in (-\infty,1)$. Let us give an outline of our program in the following, along with the main results.

\begin{Definition}[Admissible function]
A function $u:(0,\infty) \to \R$ which is $C^2$, satisfies $u' > 0$, $u'' < 0$, and $u':(0,\infty) \to (0,\infty)$ is surjective, is called \emph{admissible}.
\end{Definition}

For many purposes (such as the reverse triangle inequality for $\ell_u$), fewer assumptions on $u$ are necessary. However, in the technical arguments necessary to study the relationship between Ricci curvature lower bounds and entropic convexity, we found the set of assumptions we have given for an admissible function to be indispensable. Note that, in particular, all $u_p$ for $p \in (-\infty,1)$, $u_p(x) = x^p/p$ for $p \neq 0$ and $u_0(x) = \frac{1}{2} + \log(x)$, are admissible. A crucial observation is that if $u$ is admissible, so is its \emph{concave conjugate}
\begin{equation*}
    u^*(x):=\inf_{y > 0} (xy - u(y)).
\end{equation*}
This allows us to associate to the dual pair $(u,u^*)$ a dual convex Lagrangian--Hamiltonian pair $L(\cdot,x;u):T_xM \to \R \cup \{+\infty\}$, $H(\cdot, x; u^*): T_x^*M \to \R \cup \{+\infty\}$ defined by
\begin{equation*}
    L(v,x;u):=\begin{cases}
        - u(|v|_g), & v \text{ future,}\\+\infty & \text{otherwise,}
    \end{cases}
    \quad H(p,x;u^*):=\begin{cases}
        -u^*(|p|_g), & p \text{ past,}\\+\infty & \text{otherwise,}
    \end{cases}
\end{equation*}
Given $u$ admissible and $\lambda > 0$, we set $u_\lambda(x):=u(x/\lambda)$, which is also admissible. In the context of Lorentzian Orlicz theory, we need to allow our admissible functions $u$ to be unbounded above and below in general ($u_0 = \frac{1}{2} + \log$ provides a natural unbounded example), which creates some complications. For $\mu,\nu \in \mathcal P(M)$, let $\Pi^{u_\lambda}_\leq(\mu,\nu):=\{\pi \in \Pi(\mu,\nu) : (u_\lambda \circ \ell)_- \in L^1(\pi)\} \subseteq \Pi_\leq(\mu,\nu)$, where we tacitly set $u(-\infty):=-\infty$. The \emph{$u$-Lorentz--Orlicz--Wasserstein time separation} $\ell_u:\mathcal P(M)^2 \to \{-\infty\} \cup [0,\infty]$ is now defined as follows: $\ell_u(\mu,\nu):=-\infty$ if $\Pi_\leq(\mu,\nu) = \emptyset$, otherwise
\begin{equation*}
    \ell_u(\mu,\nu):=\sup \bigg\{\lambda > 0 : \sup_{\pi \in \Pi^{u_\lambda}_{\leq}(\mu,\nu)} \int_{M^2} u_{\lambda} \circ \ell \, d\pi \geq u(1)\bigg\} \cup \{0\}.
\end{equation*}
In the cases where $u$ is bounded above or bounded below, the integration requirement can be ignored, i.e.\ the supremum can be taken over $\Pi_\leq(\mu,\nu)$. The normalisation $u(1)$ is chosen precisely so that the compatibility $\ell_{u_p}(\mu,\nu) = \ell_p(\mu,\nu)$ holds. Under reasonable conditions on $(\mu,\nu)$, we can always find an \emph{optimal} coupling $\pi$, i.e.\ a coupling that satisfies (with $\lambda:= \ell_u(\mu,\nu) \in (0,\infty)$)
\begin{equation*}
    \int u_{\lambda} \circ \ell \, d\pi = \sup_{\tilde \pi \in \Pi^{u_\lambda}_\leq(\mu,\nu)} \int u_\lambda \circ \ell \, d \tilde \pi \geq u(1),
\end{equation*}
see Proposition \ref{prop: existence of optimal coupling}, and this inequality is saturated in many interesting examples. The corresponding dual problem to the one defining $\ell_u$ is given by 
\begin{equation*}
    \tilde C_u(\mu,\nu):=\inf\Bigg\{\eta>0\ :\ \inf_{\varphi\in L^1(\mu): \varphi^{(u_\eta \circ \ell)} \in L^1(\nu)}\Big\{\int\varphi\, d\mu+\int\varphi^{(u_{\eta}\circ\ell)}\, d\nu\Big\}\leq u(1)\Bigg\},
\end{equation*}
where $\varphi^{(u_\eta \circ \ell)}(x):=\sup_{y \in \mathrm{supp}(\mu)} u_\eta(\ell(x,y)) - \varphi(y).$ In this generality, the problem of duality is difficult to handle, however we identify a condition we term \emph{$u$-separation} (see Definition \ref{Definition: u-separeted new}) from which we are able to deduce strong duality (see Theorem \ref{Theorem: Duality by u-separation}). Our notion of $u$-separation agrees with McCann's $p$-separation \cite[Def.\ 4.1]{McCann:2020} in the case $u = u_p$, $p \in (0,1)$. In the case where $(\mu,\nu)$ are $u$-separated and $\mu \in \mathcal P^{ac}_c(M)$, there exists a unique optimal coupling that is induced by a map of the form $F(x):= \exp_x DH(D\varphi(x),x;u_\lambda^*)$, where $\lambda:=\ell_u(\mu,\nu) \in (0,\infty)$, see Theorem \ref{thm: characterising opt maps}. In this case, $\mu_s:=(F_s)_{\#}\mu$, with $F_s(x):=\exp_x s\,DH(D\varphi(x), x; u_\lambda^*)$, is the unique $u$-geodesic connecting $\mu$ to $\nu$, and moreover $\mu_s \in \mathcal P^{ac}_c(M)$ for all $s < 1$ (see Corollary \ref{cor: lag charact og geodesics}). 

Having established these properties of the Orlicz-type transport in Section \ref{section: OT}, which are in analogy to the usual transport for $u_p$, we can proceed with the usual Jacobi field computations in order to relate entropic convexity with timelike Ricci curvature bounds (see Section \ref{Section: Ricci and entropic convexity}). Before continuing, let $V \in C^2(M)$, recall that the Bakry-Émery Ricci curvature $\Ric^{(N,V)}$ with $N \in [n,\infty]$ (where the case $N = n$ means $V = 0$ by convention; and $N = \infty$ is understood in the limit sense) is defined by
\begin{equation*}
    \Ric^{(N,V)} :=\Ric + \mathrm{Hess}_g(V) + \frac{dV \otimes dV}{N-n},
\end{equation*}
and the Boltzmann--Shannon entropy $E_V:\mathcal P(M) \to [-\infty,\infty]$ is defined by
\begin{equation*}
        E_V(\mu):= \int \frac{d\mu}{d\mathfrak m} \log \frac{d\mu}{d\mathfrak m} \, d\mathfrak m
    \end{equation*}
    if this expression is well-defined in $[-\infty,\infty]$, and set to $-\infty$ otherwise, where $\mathfrak m:=e^{-V} \mathrm{vol}_g$.
The result of our efforts in Subsection \ref{subsection: entropic convexity from Ricci} and Subsection \ref{Subsection: Ricci from entropic convexity} is collected in the following theorem.

\begin{Theorem}[Equivalence of timelike Ricci bounds and entropic convexity along $u$-geodesics]
Let $(M,g)$ be a globally hyperbolic spacetime, $u:(0,\infty) \to \R$ admissible, $V \in C^2(M)$, $K \in \R$, $N \in [n,\infty]$, and set $\mathfrak m:=e^{-V} \mathrm{vol}_g$.
\begin{enumerate}
    \item If $\Ric^{(N,V)}(v,v) \geq K\,g(v,v)$ for every timelike $v \in TM$, then for every $u$-separated and absolutely continuous pair $(\mu_0,\mu_1) \in \mathcal{P}^{ac}_c(M)^2$, letting $\mu_s$ be the unique $u$-geodesic from $\mu_0$ to $\mu_1$ and $\pi$ the unique optimal coupling, $e(s):=E_V(\mu_s)$ is semiconvex and satisfies
    \begin{equation*}
        e'' - \frac{1}{N} (e')^2 \geq  K \|\ell\|_{L^2(\pi)}
    \end{equation*}
    in the distributional sense. If moreover $K \geq 0$, then in fact
    \begin{equation*}
        e'' - \frac{1}{N} (e')^2 \geq K \,\ell_u(\mu_0,\mu_1)^2.
    \end{equation*}
    \item If $\Ric^{(N,V)}(v,v) < K\, g(v,v)$ for some $v \in T_xM$, then, supported in any arbitrarily small neighborhood of $x$, one may find a pair of measures $(\mu_0,\mu_1) \in \mathcal P^{ac}_c(M)^2$ with $\mathrm{supp}(\mu_0) \times \mathrm{supp}(\mu_1) \subseteq \{\ell > 0\}$ such that $e(s) = E_V(\mu_s)$ is $C^2$ and satisfies
    \begin{equation*}
        e'' - \frac{1}{N} (e')^2 < K\, \ell_u(\mu_0,\mu_1)^2.
    \end{equation*}
\end{enumerate}
\end{Theorem}

In Subsection \ref{Subsection: relaxation}, we relax the assumption of $u$-separation to obtain a weak entropic convexity inequality for a larger class of measures, in analogy with McCann \cite[Sec.\ 7]{McCann:2020}. The main difference between our setting and McCann's is that we make several technical assumptions which are implied by McCann's assumption of attainment and finiteness of the infimum in the Kantorovich duality problem, cf.\ Remark \ref{Remark: comparing relaxation with McCann}, in the case $u = u_p$, $p \in (0,1)$. The precise result we obtain is the following:

\begin{Theorem}[Weak convexity from timelike lower Ricci curvature bounds for more general measures]
        Let $(M,g)$ be a globally hyperbolic spacetime. Fix $V\in C^2(M)$ bounded, $N \in [n,\infty]$, and $u:(0,\infty)\to\R$ admissible, by convention $V = 0$ if $N = n$. If $\operatorname{Ric}^{(N,V)}(v,v)\geq Kg(v,v)\geq0$ holds in every timelike direction $v\in TM$, then for the set $Q\subseteq\mathcal{P}^{ac}(M)^2$ of measures $(\mu,\nu)$ such that 
        \begin{enumerate}
            \item $\lambda:=\ell_u(\mu,\nu)\in(0,\infty)$, \item there exist lower semicontinuous functions $a,b:M \to \R$ with $a \in L^1(\mu)$, $b \in L^1(\nu)$ such that $u_{\lambda} \circ \ell \leq a \oplus b$ on $\mathrm{supp}(\mu \times \nu)$, and \item  there exists an optimal coupling $\pi\in\Pi^{u_{\lambda}}_{\leq}(\mu,\nu)$ with $\ell>0$ holding $\pi$-a.e.\ which satisfies $\int u_{\lambda} \circ \ell \, d\pi = u(1)$, 
        \end{enumerate}
            weak entropic convexity holds in the following sense: For each $(\mu,\nu) \in Q$ there exists a $u$-geodesic $(\mu_s)_{0 \leq s \leq 1}$ with $\mu_0 = \mu$, $\mu_1 = \nu$, such that the relative entropy $e(s):=E_V(\mu_s)$ is semiconvex and satisfies
        \begin{equation*}
            e'' - \frac{1}{N}(e')^2 \geq K\, \ell_u(\mu,\nu)^2
        \end{equation*}
        in the distributional sense.
    \end{Theorem}

This concludes the main body of our work. In Appendix \ref{Appendix}, we clarify the relationship between order-theoretic notions of completeness considered recently in the literature \cite{Octet, braun2024dAlembertian, braun2026spacetime, gigli2025hyperbolic, mondinosaemann2025lorentzian} and global hyperbolicity on a (finite-dimensional) smooth spacetime. Under the assumption of closedness of the causal relation $\leq$, it turns out that all of these notions are equivalent to global hyperbolicity (see Theorem \ref{Theorem: equivalence of order completeness and glob hyp}), mirroring the well-known equivalence between completeness and properness of a finite-dimensional Riemannian manifold as a consequence of the Hopf--Rinow theorem. Given the indispensability of the assumption that the causal relation be closed in the synthetic literature, our compatibility result justifies the restriction to the globally hyperbolic setting whenever one deals with problems of optimal transport on smooth spacetimes.

\subsection{Notation \& conventions}
\label{subsection: Notation conventions}

Throughout, let $(M,g)$ be a globally hyperbolic smooth spacetime, where we assume that $g$ has signature $(+,-,\dots,-)$. $\mathcal P(M)$ denotes the Borel probability measures on $M$, while $\mathcal P_c(M) \subseteq \mathcal P(M)$ are those measures with compact support, and $\mathcal P^{ac}(M)$ are the ones which are absolutely continuous with respect to the volume measure $\mathrm{vol}_g$. We also set $\mathcal P^{ac}_c(M):=\mathcal P^{ac}(M) \cap \mathcal P_c(M)$. We also fix a complete background Riemannian metric $\tilde g$ on $M$. While we write the timelike and causal relations as $\ll$ and $\leq$, respectively, we instead use the notation $M^2_\ll$ and $M^2_\leq$ whenever we want to emphasise the corresponding subsets of $M^2$, i.e., $M^2_\ll = \{(x,y) \in M^2 : x \ll y\}$, similarly $M^2_\leq$. We often also write $M^2_\ll = \{\ell > 0\}$ and $M^2_\leq = \{\ell \geq 0\}$, where $\ell:M^2 \to \{-\infty\} \cup [0,\infty)$ is the extended Lorentz distance or time separation. We denote by $pr_i$ the projection map onto the $i$-th factor. The symbol $\tilde D$ will be used to denote the approximate derivative of a function (cf.\ \cite[Def.\ 3.8]{McCann:2020}).

\section{Orlicz-type optimal transport in smooth spacetimes}\label{section: OT}

\subsection{The Lorentz distance and admissible functions}

We start by mentioning some well-known properties of the Lorentz distance, cf.\ McCann \cite{McCann:2020}.

    \begin{Definition}[Singularities of the Lorentz distance]
        We say a pair $(x,y)\in M^2$ is \emph{not singular} is $\ell(x,y)>0$ and $x$ and $y$ both lie in the relative interior of some affinely parametrised proper-time maximising geodesic segment. In any other case, we say $(x,y)$ are \emph{singular}, and we write $(x,y)\in\operatorname{sing}(\ell)$.
    \end{Definition}

    \begin{Lemma}[Midpoint continuity away from cut locus, {\cite[Lem.\ 2.4]{McCann:2020}}]\label{lemma: midpoint continuity away from cut locus}
      Let $(x,y)\in M^2\setminus\operatorname{sing}(\ell)$ (which is an open subset of $M^2$), and fix $s\in[0,1]$. Then there is a unique point $z=z_s(x,y)\in M$ such that
      \begin{equation*}
          \ell(x,z)=s\,\ell(x,y)\quad\text{and}\quad\ell(z,y)=(1-s)\,\ell(x,y).
      \end{equation*}
      Moreover, $z$ depends smoothly on $(s,x,y)\in[0,1]\times (M^2\setminus\operatorname{sing}(\ell))$.
    \end{Lemma}

    \begin{Lemma}[Midpoint sets inherit compactness, {\cite[Lem.\ 2.5]{McCann:2020}}]\label{lemma: midpoints inherit compactness}
        Given $S\subset M^2$ and $s\in[0,1]$, let
        \begin{equation*}
            Z_s(S):=\bigcup_{(x,y)\in S}Z_s(x,y),
        \end{equation*}
        where
        \begin{equation*}
            Z_s(x,y):=\Big\{z\in M:\ \ell(x,z)=s\,\ell(x,y)\quad\text{and}\quad\ell(z,y)=(1-s)\,\ell(x,y)\Big\}
        \end{equation*}
        if $\ell(x,y)\geq0$, and $Z_s(x,y)=\emptyset$ otherwise. If $S$ is precompact, then $Z_s(S):=\cup_{s\in[0,1]}Z_s(S)$ is precompact. If, in addition, $S$ is compact, then $Z(S)$ and $Z_s(S)$ are compact.
    \end{Lemma}

    \begin{Lemma}[Selecting midpoints on the timelike cut locus, {\cite[Lem.\ 2.8]{McCann:2020}}]\label{lemma: selecting midpoints away from cut locus}
        In the conditions of Lemma \ref{lemma: midpoint continuity away from cut locus}, the maps $z_s$ can be measurably extended to $\{\ell>0\}$ by $\bar{z}_s$, satisfying that if $(x,y)\in\{\ell>0\}$, then $s\in[0,1]\mapsto\bar{z}_s(x,y)$ is a proper-time maximising geodesic segment joining $x$ to $y$.
    \end{Lemma}

    For the next result, recall that the set of subgradients of a function $u: M \to [-\infty,+\infty]$ at $x \in M$, denoted $\partial_{\cdot} u(x)$ is the set of $p \in T_x^*M$ such that $u(x)\in\R$ and
    \begin{equation*}
        u(\exp^{\tilde g}_x v) \geq u(x) + p[v] + o(|v|_{\tilde g})
    \end{equation*}
    for small enough $v \in T_xM$. This definition is easily seen to be independent of the background Riemannian metric $\tilde g$. The reversed inequality is characteristic of supergradients, denoted $\partial^{\cdot}u$.

    \begin{Theorem}[Smoothness of time separation, {\cite[Thm.~3.6]{McCann:2020}}]
    \label{Theorem: SmoothnessofLorentzdistance}
        The time separation $\ell:M^2\to[0,\infty)\cup\{-\infty\}$ is
        \begin{enumerate}[label=(\alph*)]
            \item upper semicontinuous,
            \item continuous on $\ell^{-1}([0,\infty))$,
            \item smooth precisely on the complement of the closed set $\operatorname{sing}(\ell)$,
            \item \begin{enumerate}[label=\Roman*.]
                \item locally Lipschitz and locally semiconvex on the open set $\{\ell>0\}$.
                \item Moreover, if $y=\exp_xv$ and $x=\exp_yw$ for $(x,y)\in\ell^{-1}((0,\infty))$, then $-\frac{v_*}{|v_*|_g}\in\partial_{\cdot}f(x)$ and $-\frac{w_*}{|w_*|_g}\in\partial_{\cdot}\overline{f}(y)$, where $f(\cdot):=\ell(\cdot,y)$, $\overline{f}(\cdot):=\ell(x,\cdot)$ and $v_*=g(v,\cdot)$.
            \end{enumerate}  
            \item However, the superdifferential of $\ell(\cdot,y)$ is empty at $x$ if $\ell(x,y)=0$ unless $x=y$, in which case the supergradients lie in the solid hyperboloid $\{\omega\in T^*_xM:\ \omega \text{ past, }g^*(\omega,\omega)\geq 1\}$.
        \end{enumerate}
        \label{thm: smoothness of l}
    \end{Theorem}

    \begin{Remark}
        Note that, as a consequence of Theorem \ref{Theorem: SmoothnessofLorentzdistance}(b), $\ell_+:=\max\{\ell,0\}$ is continuous.
    \end{Remark}

    We now introduce the class of functions for which we will study the Orlicz-type Lorentzian optimal transport problem.

    \begin{Definition}[Admissible function]
    A function $u:(0,\infty) \to \R$ that is $C^2$, satisfies $u' > 0$, $u'' < 0$, and $u':(0,\infty) \to (0,\infty)$ is surjective, will be called \emph{admissible}.
    \end{Definition}

    Note that an admissible function has well defined values $u(0)$ and $u(\infty)$ (possibly $-\infty$ and $+\infty$, respectively). The convention $u(-\infty):=-\infty$ will be used in some places.

    \begin{Lemma}[Concave conjugate]
    \label{Lemma: Concave conjugate}
    Let $u:(0,\infty) \to \R$ be admissible. Then its \emph{concave conjugate} $u^*:(0,\infty) \to \R$ defined by
    \begin{equation*}
        u^*(x):=\inf_{y > 0} (xy - u(y))
    \end{equation*}
    is also admissible and satisfies $(u^*)' = (u')^{-1}$, i.e., $(u^*)' \circ u' = u' \circ (u^*)' = \mathrm{id}$ on $(0,\infty)$. Moreover, $(u^*)^* = u$.
    \end{Lemma}
    \begin{proof}
        Clearly, the definition of admissibility gives that $u':(0,\infty) \to (0,\infty)$ is a $C^1$-diffeomorphism. For fixed $x > 0$, the function $(0,\infty) \ni y \mapsto xy - u(y)$ is strictly convex and $C^2$, its unique minimum is given by $y = (u')^{-1}(x)$. Thus,
        \begin{equation*}
            u^*(x) = x \,(u')^{-1}(x) -  u((u')^{-1}(x))
        \end{equation*}
        is clearly $C^1$. An elementary calculation, using the above formula, yields $(u^*)' = (u')^{-1}$, showing that $u^*$ is $C^2$. The remaining properties of admissible functions also easily follow from this relationship between the first derivatives. Finally, to see that $(u^*)^* = u$, observe that $(u^*)^* \geq u$, and, for given $y > 0$ and $x:=u'(y)$, we have
        \begin{equation*}
            u(y) = xy - u^*(x) \geq (u^*)^*(y),
        \end{equation*}
        from which we can conclude $u = (u^*)^*$.
    \end{proof}

    \begin{Remark}[On the definition of $u$]
    \begin{enumerate}
    \item[]
        \item Without the assumption of surjectivity of $u'$, the concave conjugate $u^*$ need not be admissible: For example, consider $u(x) = -e^{-x}$ for $x > 0$. Then $u'(x) = e^{-x} > 0$, $u''(x) = u(x) < 0$, and $u'$ is bounded above by $1$. Thus, if $x > 1$, then $y \mapsto xy - u(y)$ is strictly increasing, so that
    \begin{equation*}
        u^*(x) = \inf_{y > 0} (xy - u(y)) = u(0) = 1.
    \end{equation*}
    In particular, $u^*$ is neither strictly increasing nor strictly concave. For an example with $u'$ bounded below, consider for given $c > 0$ the function $u(x) = cx + \log(x)$ which satisfies $u'(x) = c + x^{-1} \geq c > 0$ and $u''(x) = -x^{-2} < 0$. Then for any $0 < x < c$, $y \mapsto xy - u(y)$ is strictly decreasing, so that
    \begin{equation*}
        u^*(x) = \inf_{y > 0} (xy - u(y)) = \lim_{y \to \infty} -(c-x)y - \log(y) = -\infty.
    \end{equation*}
    Thus, in this case, $u^*$ is not even finite-valued.
    \item For many ``synthetic-geometric" purposes, such as the triangle inequality of the time separation $\ell_u$ on $\mathcal P(M)^2$ (see Proposition \ref{prop: reverse triangle inequality} below), it suffices to assume $u'' \geq 0$ and even $u' \geq 0$, one may even weaken the differentiability assumption on $u$ or drop it entirely ($u$, under any of these reduced assumptions, may reasonably be called \emph{weakly admissible}). However, since we will be interested in the connection of entropic convexity and timelike Ricci curvature lower bounds, (strictly) convex Hamiltonian--Lagrangian duality (see Proposition \ref{lem: lagrange-hamilton duality} below) will be an essential technical tool, which on the other hand crucially relies on strict concavity of $u$ and the fact that $u^*$ is admissible.
    \end{enumerate}
    \end{Remark}

    \begin{Example}[Admissible functions]
    The main examples of admissible functions considered so far in the literature are
    \begin{align*}
        &u_p(x):=\frac{x^p}{p}, \quad 0 \neq p < 1,\\
        &u_0(x):=\frac{1}{2} + \log(x),
    \end{align*}
    for $p \in (0,1)$ cf.\ \cite{McCann:2020, MS:22, CM:20, Braun:2023Renyi, BraunMcCann:2023}, for the more recent $p < 0$ see \cite{Octet, braun2024dAlembertian, gigli2025hyperbolic}.
    They satisfy $u_p^* = u_q$, where $p^{-1} + q^{-1} = 1$ ($q = 0$ if $p = 0$). It is sometimes useful to consider the following shifted versions:
    \begin{align*}
        &\bar{u}_p(x):=\frac{x^p -1}{p}, \quad 0 \neq p < 1,\\
        &\bar{u}_0(x):= \log(x).
    \end{align*}
    In this case, the joint assignment $(-\infty,1) \times (0,\infty) \to \R$, $(p,x) \mapsto \bar{u}_p(x)$ is smooth.
    \end{Example}

    The following is an analogue of the Lagrangian--Hamiltonian duality obtained by McCann in the case $u = u_p$ for $p \in (0,1)$, cf.\ \cite[Lem.\ 3.1]{McCann:2020}.

    \begin{Proposition}[General convex Lagrangian--Hamiltonian duality]\label{lem: lagrange-hamilton duality}
    Let $u:(0,\infty) \to \R$ be admissible, denote by $u^*$ its admissible concave conjugate. Fix a point $x \in M$. Define the $u$-Lagrangian $L(\cdot,x;u):T_xM \to \R \cup \{+\infty\}$ via
    \begin{equation*}
        L(v, x; u):=\begin{cases}
            -u(|v|_g), & v \text{ future causal,}\\
            +\infty & \text{ otherwise.}
        \end{cases}
    \end{equation*}
    Also, define the corresponding $u^*$-Hamiltonian $H(\cdot,x;u^*):T_x^*M \to \R \cup \{+\infty\}$ via
    \begin{equation*}
        H(p,x;u^*):=\begin{cases}
            -u^*(|p|_{g^*}), & p\text{ past causal},\\
            +\infty & \text{otherwise}.
        \end{cases}
    \end{equation*}
    Then $L(\cdot,x; u)$ and $H(\cdot,x; u^*)$ are convex on $T_xM$ resp.\ $T^*_xM$, $C^2$ and strictly convex in the interior of the future lightcone in $T_xM$ resp.\ the interior of the past lightcone in $T_x^*M$. Moreover, $H$ is the Legendre--Fenchel convex conjugate of $L$ (and vice versa) in the sense that
    \begin{align*}
        H(p,x;u^*) &= \sup_{v \in T_xM} (p(v) - L(v,x;u)),\\
        L(v,x;u) &= \sup_{p \in T_x^*M} (p(v) - H(p,x;u^*)).
    \end{align*}
    Consequently, $DH = (DL)^{-1}$ is a $C^1$-diffeomorphism from the past timelike cone in $T_x^*M$ into the future timelike cone in $T_xM$.
    \end{Proposition}
    \begin{proof}
        This follows from the same arguments used by McCann in the case of $u = u_p$ and $u^* = u_q$ \cite[Lem.\ 3.1]{McCann:2020}, using the fact that $(u^*)' = (u')^{-1}$ holds generally (cf.\ Lemma \ref{Lemma: Concave conjugate}) in conjunction with the following exact forms of $DL,DH,D^2L,D^2H$ at future timelike vectors $v \in T_xM$ resp.\ past timelike covectors $p \in T^*_xM$:
        \begin{align*}
            &DL(v,x;u) = - u'(|v|_g) \frac{v_*}{|v|_g}, \quad DH(p,x;u^*) = -(u^*)'(|p|_g) \frac{p_*}{|p|_g},\\
            &D^2L(v,x;u) = - \frac{u'(|v|_g)}{|v|_g} g + (u'(|v|_g) - |v|_g u''(|v|_g)) \frac{v_* \otimes v_*}{|v|_g^3},\\
            &D^2 H(p,x;u^*) = -\frac{(u^*)'(|p|_g)}{|p|_g} g + ((u^*)'(|p|_g) - |p|_g (u^*)''(|p|_g)) \frac{p_* \otimes p_*}{|p|_g^3}.
        \end{align*}
        Here, $v_* = g(v,\cdot) \in T_x^*M$, $p_* = g(p,\cdot) \in T_xM$, and we use $g$ to denote both the scalar product on $T_xM$ and on $T_x^*M$. 
    \end{proof}

\begin{Corollary}[Twist and non-degeneracy]\label{cor: twist and non degeneracy}
        For an admissible function $u:(0,\infty)\to\R$,
        \begin{enumerate}
            \item $u\circ\ell: M^2 \to [-\infty,+\infty)$ inherits properties (a), (b) and (d)(I) of Theorem \ref{thm: smoothness of l} from $\ell$, and is moreover $C^2$ precisely on $M^2 \setminus \operatorname{sing}(\ell)$;
            \item If there exists a supergradient $w \in T_x^*M$ of $u\circ\ell(\cdot,y)$ at the point $x$ and $\ell(x,y) > 0$, then $v = DH(w,x;u^*)$ is the unique vector such that $y = \exp_x(v)$;
            \item If $(x,y)\in\{\ell>0\}\cap\operatorname{sing}(\ell)$, then 
            \begin{equation}
                \sup_{0<|w|_{\tilde{g}}<1}\frac{u\circ\ell\left(\exp^{\tilde{g}}_xw,y\right)+u\circ\ell\left(\exp^{\tilde{g}}_x(-w),y\right)-2u\circ\ell(x,y)}{2|w|^2_{\tilde{g}}}=+\infty; \label{eq: twist&nondegeneracy2}
            \end{equation}
            \item If $(x,y)\notin\operatorname{sing}(\ell)$, then $\det \frac{\partial^2 (u \circ \ell)}{\partial x^j \partial y^i}\neq0$.
        \end{enumerate}
    \end{Corollary}
    \begin{proof}
        \begin{enumerate}
            \item The properties (a), (b) and as well as the claimed regularity of $u \circ \ell$ on $M^2 \setminus \operatorname{sing}(\ell)$ are immediate given that $u$ is $C^2$ and extendible to a continuous function $u:[0,\infty) \to [-\infty,+\infty)$. 
                
                To prove the inheritance of local semiconvexity from $\ell$, we work with the following characterisation of the property: a function $f:\R^n\to[0,\infty)$ has semiconvexity constant $C$ near $\bar{x}$ if there exists $p\in\partial_\cdot f(\bar{x})$ such that 
                \begin{equation*}
                    f(x)\geq f(\bar{x})+p(x-\bar{x})-\frac{1}{2}C|x-\bar{x}|^2+O(|x-\bar{x}|^2)
                \end{equation*}
                holds for all $x$ in a small enough neighbourhood of $\bar{x}$. The idea now is to evaluate $u$ with the elements of this inequality, which will be preserved since $u'>0$. To obtain a similar expression for $u\circ f$, we use the Taylor expansion of $u$ around a point:
                \begin{equation*}
                    u(\bar{y}+y)=u(\bar{y})+u'(\bar{y})\,(y-\bar{y})+\frac{1}{2}u''(\bar{y})\,|y-\bar{y}|^2.
                \end{equation*}
                Bringing everything together, we conclude that $u\circ f$ has at $\bar{x}$ semiconvexity constant $Cu'(f(\bar{x}))-2u''(f(\bar{x}))|p|^2$. Applying this argument in Riemannian normal coordinates gives the local semiconvexity of the function $u\circ\ell(\cdot,y)$ at each point $\bar{x}$ with $\ell(\bar{x},y)>0$.

                From this, the fact that $u\circ\ell$ is locally Lipschitz follows immediately, since locally bounded convex functions are locally Lipschitz.

                \item Suppose $u\circ\ell(\cdot,y)$ admits $w\in T^*_xM$ as a supergradient. Using the one-sided chain rule \cite[Lem.~5]{McCann:01}, one can obtain that $\ell(\cdot,y)$ admits
            \begin{equation*}
                \frac{1}{u'\left(u^{-1}(u(\ell(x,y)))\right)}w=\frac{1}{u'\left(\ell(x,y)\right)}w
            \end{equation*}
            as a supergradient at $x\in M$. Since the $\ell(\cdot,y)$ is always subdifferentiable on $I^-(y)$ by Theorem \ref{Theorem: SmoothnessofLorentzdistance}, we conclude that it is thus differentiable at $x$, whence $w$ is unique. Using the chain rule and noting that $\ell(x,y) = |\exp^{-1}_xy|_g$, we obtain
            \begin{equation}
            \label{eq: w = DL}
                w = d(u \circ \ell(\cdot,y))(x) = u'( |\exp^{-1}_x y|_g) \,\frac{- \exp_x^{-1}y}{|\exp_x^{-1}y|_g} = DL(\exp_x^{-1}y,x; u).
            \end{equation}
            An application of $DH(\cdot,x; u^*)$ to this equation gives the claim.

             \item If \eqref{eq: twist&nondegeneracy2} does not hold, this would imply that the function $f(\cdot):=u\circ\ell(\cdot,y)$ has semiconcavity constant $C>0$ at some $\bar{x}$ with $(\bar{x},y)\in\{\ell>0\}\cap\operatorname{sing}(\ell)$. This would now imply that $u^{-1}\circ f$ has semiconcavity constant $\tilde{C}$ at $\bar{x}$, contradicting \cite[Thm.~3.5]{McCann:2020}. Therefore \eqref{eq: twist&nondegeneracy2} must hold.

             \item This is simply a consequence of convexity of $L$ on the future timelike cone, upon differentiating \eqref{eq: w = DL} with respect to $y$.
                \end{enumerate}
                \end{proof}

        \begin{Remark}[Solutions for the functions $u_p$]\label{rem: solutions for uq}
        Note that in the case where $u(x)=\frac{1}{p}x^p$, $0<p<1$, in equation \eqref{eq: w = DL} we can write an explicit expression for $v=DH(w,x;u^*)$. This is the result obtained in {\cite[Cor.\ 3.7]{McCann:2020}}, which gives that the geodesic joining $x$ and $y$ in the hypotheses of Corollary \ref{cor: twist and non degeneracy} is the exponential map with tangent vector $DH(w,x;q)=-\vert w\vert_{g}^{q-2}\,w$, where $H$ is the Hamiltonian defined by McCann \cite[Lem.\ 3.1]{McCann:2020}, and $q$ is the conjugate exponent to $p$,\ i.e. $q^{-1}+p^{-1}=1$. Note that this explicit solution is also valid for the more general case in which $u(x)=\frac{1}{p}x^p$, $0\neq p<1$. In the case where $p=0$, note that \eqref{eq: w = DL} can be explicitly written as $$w=-\frac{v_*}{\vert v_*\vert^2_{g}}\Big\vert_{v=\exp_x^{-1}y},$$ and it is easily obtained that $y=\exp_x -\vert w\vert^{-2}_{g}w$.
    \end{Remark}

    \subsection{The $\ell_u$-optimal transport problem for admissible functions $u$}

    Let $u:(0,\infty) \to \R$ be admissible. Given $\mu,\nu \in \mathcal{P}(M)$, we write $\Pi(\mu,\nu):=\{\pi \in \mathcal{P}(M^2) : (pr_1)_{\#} \pi = \mu, \, (pr_2)_{\#}\pi = \nu \}$ for the set of couplings between $\mu$ and $\nu$, $\Pi_\leq(\mu,\nu) = \{\pi \in \Pi(\mu,\nu) : \pi(M^2_\leq) = 1\}$ for the set of causal couplings, as well as
    \begin{equation*}
        \Pi_{\leq}^{u}(\mu,\nu):= \{ \pi \in \Pi(\mu,\nu) :  (u \circ \ell)_- \in L^1(\pi)\} \subseteq \Pi_\leq(\mu,\nu)
    \end{equation*}
    for the set of \emph{$u$-compatible causal couplings} between $\mu$ and $\nu$. While the integrability condition $(u \circ \ell)_- \in L^1(\pi)$ already forces $\pi(M^2_\leq) = 1$, we leave the subscript ``$\leq$" in the notation for clarity. Note that if $u$ is bounded below, the integrability condition is trivially satisfied. Also, note that for any $\lambda  > 0$ also $u_\lambda(x):=u(x/\lambda)$ is admissible.

    \begin{Definition}[$u$-Lorentz--Orlicz--Wasserstein time separation]
    We define the \emph{$u$-Lorentz--Orlicz--Wasserstein time separation} $\ell_u:\mathcal{P}(M)^2 \to [0,\infty) \cup \{-\infty\}$ as follows: set $\ell_u(\mu,\nu):=-\infty$ if $\Pi_{\leq}(\mu,\nu) = \emptyset$, otherwise
    \begin{equation*}
        \ell_u(\mu,\nu):=\sup \bigg\{\lambda > 0 : \sup_{\pi \in \Pi^{u_\lambda}_{\leq}(\mu,\nu)} \int_{M^2} u_\lambda \circ\ell \, d\pi \geq u(1)\bigg\} \cup \{0\}.
    \end{equation*}

    For $\mu,\nu\in\mathcal{P}(M)$ such that $\lambda:=\ell_u(\mu,\nu) \in (0,\infty)$, a coupling $\pi\in\Pi^{u_\lambda}_{\leq}(\mu,\nu)$ is said to be \emph{optimal} if 
    \begin{equation}
    \label{eq: optimal coupling inequality}
        \sup_{\pi' \in \Pi^{u_\lambda}_\leq(\mu,\nu)} \int_{M^2} u_\lambda \circ \ell \, d\pi' =\int_{M^2} u_\lambda \circ \ell \, d\pi(x,y) \geq u(1).
    \end{equation}
    \end{Definition}
    
\begin{Remark}[On $\ell_u$]
    Observe that the set of eligible $\lambda > 0$ in the definition of $\ell_u(\mu,\nu)$ is the following: It is the set of those $\lambda > 0$ for which $\Pi_{\leq}^{u_\lambda}(\mu,\nu) \neq \emptyset$, and the given integral inequality holds. If there is no eligible $\lambda > 0$ of this type while $\Pi_\leq(\mu,\nu) \neq \emptyset$, then $\ell_u(\mu,\nu) = 0$ by definition. Also note that since $u(\ell(x,y)/\lambda) \leq u(\ell(x,y)/\lambda')$ whenever $\ell(x,y) \geq 0$ and $\lambda \geq \lambda'$, it follows that whenever $(u_\lambda \circ \ell)_- \in L^1(\pi)$ for some $\pi \in \Pi_\leq(\mu,\nu)$, then also $(u_{\lambda'} \circ \ell)_- \in L^1(\pi)$. We conclude that $\Pi^{u_\lambda}_\leq(\mu,\nu) \subseteq \Pi^{u_{\lambda'}}_\leq(\mu,\nu)$ for all $(\mu,\nu) \in \mathcal{P}(M)^2$ whenever $\lambda' \leq \lambda$.
    \end{Remark}

    \begin{Lemma}[Simplifications for the sets $\Pi^{u_\lambda}_\leq$]\label{Lem: simplifications for Pi}
Let $u$ be admissible.
\begin{enumerate}
    \item If $u$ is bounded below, then for every $\mu,\nu \in \mathcal P(M)$ and every $\lambda > 0$,
    \begin{equation*}
        \Pi^{u_\lambda}_\leq(\mu,\nu) = \Pi_\leq(\mu,\nu).
        \end{equation*}
    \item If $u$ is bounded above, then for every $\mu,\nu \in \mathcal{P}(M)$, if $\ell_u(\mu,\nu) > -\infty$, we have
    \begin{equation*}
        \ell_u(\mu,\nu) =\sup \bigg\{\lambda > 0 : \sup_{\pi \in \Pi_{\leq}(\mu,\nu)} \int_{M^2} u(\lambda^{-1}\ell(x,y)) \, d\pi(x,y) \geq u(1)\bigg\} \cup \{0\},
    \end{equation*}
    i.e., the supremum over $\Pi^{u_\lambda}_\leq(\mu,\nu)$ can be replaced by the supremum over $\Pi_\leq(\mu,\nu)$. This claim also holds for arbitrary admissible $u$ if $\mu,\nu$ have compact support.
    \item If $\mu,\nu \in \mathcal P(M)$ are such that $\mathrm{supp}(\mu) \times \mathrm{supp}(\nu)$ is compact and contained in $\{\ell > 0\}$, then for every $\lambda > 0$
    \begin{equation*}
        \Pi^{u_\lambda}_\leq(\mu,\nu) = \Pi_\leq(\mu,\nu).
    \end{equation*}
\end{enumerate}
\end{Lemma}
\begin{proof}
    The claims in (i) and (iii) are elementary to check. For (ii), note that if $\pi \in \Pi_\leq(\mu,\nu) \setminus \Pi^{u_\lambda}_\leq(\mu,\nu)$, by the boundedness of $u$ from above we have that
    \begin{equation*}
        \int_{M^2} u_\lambda \circ \ell \, d\pi = -\infty < u(1),
    \end{equation*}
    so taking the supremum over $\Pi_\leq(\mu,\nu)$ does not change $\ell_u(\mu,\nu)$. The same argument holds if $\mu,\nu$ have compact support, since in this case $u_\lambda \circ \ell$ is bounded above on $\mathrm{supp}(\mu) \times \mathrm{supp}(\nu)$ for every $\lambda > 0$.
\end{proof}

    The following is easily checked.

\begin{Lemma}[Constant shifts]
For any constant $c \in \R$ and admissible function $u:(0,\infty) \to \R$, also $u+c$ is admissible and
\begin{equation*}
    \ell_{u+c} = \ell_u.
\end{equation*}
\end{Lemma}

\begin{Example}[$\ell_u = \ell_p$ for $u = u_p$]
    In the Lorentzian optimal transport literature (see e.g.\ \cite{McCann:2020, MS:22,CM:20,Braun:2023Renyi}), the most common function $u$ which is considered is $u = u_p$ for $p \in (0,1)$ (more recently, also $p < 0$, see e.g.\ \cite{Octet, braun2024dAlembertian}). For $0 \neq p < 1$, classically, the $p$-Lorentz-Wasserstein time separation $\ell_p:\mathcal{P}(M)^2 \to \{-\infty\} \cup [0,\infty]$ is defined as $\ell_p(\mu,\nu):=-\infty$ if $\Pi_{\leq}(\mu,\nu) = \emptyset$, otherwise
    \begin{equation*}
        \ell_p(\mu,\nu):=u_p^{-1}\left( \sup_{\pi \in \Pi_\leq(\mu,\nu)} \int_{M^2} u_p \circ \ell \, d\pi\right),
    \end{equation*}
    see \cite[Def.\ 2.12]{Octet}, where $u_p$ is understood to be extended to $u_p(0)$ and $u_p(\infty)$ (and, accordingly, also $u_p^{-1}$ is extended). We wish to show that $\ell_{u_p} = \ell_p$ for $0 \neq p < 1$.

\begin{enumerate}
    \item First, suppose $p \in (0,1)$. In this case $u_p$ is bounded below, so that $\Pi^{(u_p)_\lambda}_\leq(\mu,\nu) = \Pi_\leq(\mu,\nu)$ for every $(\mu,\nu) \in \mathcal P(M)^2$. For simplicity, let $\ell_p(\mu,\nu) < + \infty$, the case $\ell_p(\mu,\nu) = + \infty$ can easily be handled by taking limits. We suppose first that $\ell_{u_p}(\mu,\nu) > 0$. Then for every $0 < \lambda < \ell_{u_p}(\mu,\nu)$, we have that (using the $p$-homogeneity of $u_p$)
    \begin{align*}
        &\sup_{\pi \in \Pi_\leq(\mu,\nu)}\int_{M^2} u_p\left( \frac{\ell(x,y)}{\lambda}\right) \, d\pi(x,y) \geq u_p(1)\\
        \Leftrightarrow \quad &\sup_{\pi \in \Pi_\leq(\mu,\nu)}\int_{M^2} u_p(\ell(x,y)) \, d\pi(x,y) \geq u_p(\lambda).
    \end{align*}
    Inverting $u_p$, we conclude that
    \begin{align*}
        \lambda \leq u_p^{-1}\left( \sup_{\pi \in \Pi_\leq(\mu,\nu)} \int_{M^2} u_p \circ \ell \, d\pi\right) = \ell_p(\mu,\nu).
    \end{align*}
    Taking a supremum over $\lambda < \ell_{u_p}(\mu,\nu)$, we thus conclude that $\ell_{u_p}(\mu,\nu) \leq \ell_p(\mu,\nu)$. In particular, also $\ell_p(\mu,\nu) > 0$ since we assumed $\ell_{u_p}(\mu,\nu) > 0$. To prove the reverse inequality, observe that $\ell_p(\mu,\nu)$ is itself among the eligible $\lambda$ in the definition of $\ell_{u_p}$. Indeed, for any $\pi \in \Pi_\leq(\mu,\nu)$,
    \begin{align*}
        \int_{M^2} u_p\left(\frac{\ell(x,y)}{\ell_p(\mu,\nu)} \right) \, d\pi(x,y) &= \frac{1}{\ell_p(\mu,\nu)^p} \int_{M^2} u_p(\ell(x,y)) \, d\pi(x,y) \\
        &= \frac{1}{p} \frac{1}{u_p(\ell_p(\mu,\nu))} \int_{M^2} u_p(\ell(x,y)) \, d\pi(x,y).
    \end{align*}
    Taking the supremum over $\pi \in \Pi_\leq(\mu,\nu)$, the right hand side becomes $1/p = u(1)$, as claimed. Hence, $\ell_p(\mu,\nu) \leq \ell_{u_p}(\mu,\nu)$ and thus equality must hold.

    Now, suppose that $\ell_{u_p}(\mu,\nu) = 0$. Then also $\ell_p(\mu,\nu) = 0$, since otherwise the calculation above shows that $\ell_p(\mu,\nu) > 0$ is among the eligible $\lambda$ in the definition of $\ell_{u_p}$, from which $\ell_{u_p}(\mu,\nu) > 0$ would follow.

    \item Now we consider the case $p < 0$. Note that due to $u_p(\lambda^{-1} \ell(x,y)) = \lambda^{-p}\, u_p(\ell(x,y))$ for every $\lambda > 0$, we have that $\Pi_\leq^{(u_p)_\lambda}(\mu,\nu) = \Pi^{u_p}_\leq(\mu,\nu)$ for every $\lambda > 0$. Moreover, for $\pi \in \Pi_\leq(\mu,\nu) \setminus \Pi^{u_p}_\leq(\mu,\nu)$, given that $u_p < 0$, we have that $\int_{M^2} u_p \circ \ell \, d\pi = -\infty$, so that these couplings may be omitted in the definition of $u_p$:
    \begin{equation*}
        \ell_p(\mu,\nu) = u_p^{-1} \left( \sup_{\pi \in \Pi^{u_p}_\leq(\mu,\nu)}\int_{M^2} u_p \circ \ell \, d\pi \right).
    \end{equation*}
    Having made this observation, the compatibility $\ell_{u_p}(\mu,\nu) = \ell_p(\mu,\nu)$ is checked precisely as in the case $p \in (0,1)$ in (i).

    \item Let us remark on the case $p = 0$. Here, $u_0 = \log$, and $\log(\ell(x,y)/\lambda) = \log(\ell(x,y)) - \log(\lambda)$ for every $\lambda > 0$, so that for given $\pi \in \Pi_\leq(\mu,\nu)$ we have
    \begin{equation*}
        \int_{M^2} \log(\ell(x,y)/\lambda) \, d\pi = \int_{M^2} \log(\ell(x,y)) \, d\pi - \log(\lambda).
    \end{equation*}
    It follows that $\Pi^{(u_0)_\lambda}_\leq(\mu,\nu) = \Pi^{u_0}_\leq(\mu,\nu)$ for every $\lambda > 0$ and $(\mu,\nu) \in \mathcal{P}(M)^2$. While a definition of $\ell_0$ (analogous to $\ell_p$ for $0 \neq p < 1$) has not been given in the literature so far, it is reasonable to define it as $\ell_0(\mu,\nu):=-\infty$ if $\Pi_\leq(\mu,\nu) = \emptyset$, and otherwise 
    \begin{equation*}
        \ell_0(\mu,\nu):=u_0^{-1} \left( \sup_{\pi \in \Pi^{u_0}_\leq(\mu,\nu)} \int_{M^2} u_0 \circ \ell \, d \pi \right) = \exp \left( \sup_{\pi \in \Pi^{\log}_\leq(\mu,\nu)} \int_{M^2} \log \circ \ell \, d \pi\right).
    \end{equation*}
    As in the cases $0 \neq p < 1$, one readily checks that $\ell_0 = \ell_{u_0}$. In fact, the normalisation $u(1)$ in the definition of $\ell_u$ was chosen precisely in order for these compatibilities to be true (and is fairly arbitrary otherwise).
    \item The case $p = 1$ can also be covered by $\ell_u$ if one weakens the definition of admissible function to only require $u'' \leq 0$ instead of $< 0$. However, since we our primary interest in this paper is to characterise timelike Ricci curvature bounds via entropic convexity, strict concavity of $u$ will be essential, so we require it throughout for simplicity.
\end{enumerate}
    \end{Example}

The following gives a basic result on the existence of optimal couplings and their properties.

\begin{Proposition}[Existence of optimal couplings]
\label{prop: existence of optimal coupling}
Let $\mu,\nu \in \mathcal P(M)$ such that $\lambda:=\ell_u(\mu,\nu) \in (0,\infty)$, and suppose there are lower semicontinuous functions $a,b: M \to \R$ such that $a \in L^1(\mu), \, b \in L^1(\nu)$, and for some $\lambda' < \lambda$, we have $u_{\lambda'}(\ell(x,y)) \leq a\oplus b(x,y):=a(x) + b(y)$ on $\mathrm{supp}(\mu) \times \mathrm{supp}(\nu)$. Moreover, assume that $\eta \mapsto S_{\mu,\nu}(\eta)$ is right-continuous at $\eta = \lambda$, where
\begin{equation*}
    S_{\mu,\nu}(\eta):=\sup_{\pi' \in \Pi^{u_\eta}_\leq(\mu,\nu)} \int_{M^2} u_\eta \circ \ell \, d\pi'.
\end{equation*}
Then for any $\pi \in \Pi_\leq(\mu,\nu)$ and any $\lambda'' \geq \lambda'$, $(u_{\lambda''}\circ \ell)_+ \in L^1(\pi)$. In particular, in this case $\ell_u(\mu,\nu)$ can be written as
\begin{equation}
\label{eq: rewriting of ellu}
    \lambda = \ell_u(\mu,\nu) = \sup\left\{ \lambda'' \geq \lambda' : \sup_{\pi \in \Pi_\leq(\mu,\nu)} \int_{M^2} u_{\lambda''}\circ \ell \, d\pi \geq u(1)\right\}
\end{equation}
and there exists an optimal coupling $\pi \in \Pi^{u_\lambda}_\leq(\mu,\nu)$. Moreover, equality holds in the sense that
\begin{equation*}
     \int_{M^2} u_{\lambda} \circ \ell \, d\pi = \sup_{\tilde \pi \in \Pi_\leq(\mu,\nu)} \int_{M^2} u_{\lambda} \circ \ell \, d\tilde \pi = u(1).
\end{equation*}
Finally, if $u(0) = -\infty$, then any optimal coupling $\pi$ is concentrated on $\{\ell > 0\}$.
\end{Proposition}
\begin{proof}
    The claim about the integrability of $(u_{\lambda''} \circ \ell)$ is trivial, since $u_{\lambda''}\circ \ell \leq u_{\lambda'} \circ \ell$ whenever $\lambda' \leq \lambda''$. From this, it is easily seen that one can rewrite $\ell_u$ in the form \eqref{eq: rewriting of ellu}, since any $\pi \in \Pi_\leq(\mu,\nu) \setminus \Pi^{u_{\lambda''}}_\leq(\mu,\nu)$ integrates $u_{\lambda''} \circ \ell$ to $-\infty$ and thus contributes nothing to the supremum.

    It remains to show that there exists an optimal coupling. To this end, pick $\lambda_k \uparrow \lambda$, and couplings $\pi_k \in \Pi_\leq(\mu,\nu)$ such that
    \begin{equation*}
        \int_{M^2} u_{\lambda_k} \circ \ell \, d\pi_k \geq u(1) - \frac{1}{k}.
    \end{equation*}
    Since $\Pi_\leq(\mu,\nu)$ is compact, up to subsequences, $(\pi_k)_{k\in\N}$ narrowly converges to some coupling $\pi \in \Pi_\leq(\mu,\nu)$. Fix now $\lambda'' \in (\lambda',\lambda)$. Then, since $u_{\lambda''} \circ \ell - a\oplus b$ is upper semicontinuous and bounded above (by $0$), using the inverse monotonicity of $\lambda'' \mapsto u_{\lambda''} \circ \ell$, we have (using the narrow convergence $\pi_k \to \pi$)
    \begin{align*}
        \int u_{\lambda''} \circ \ell - a \oplus b \, d\pi &\geq \limsup_{k \to \infty} \int u_{\lambda''} \circ \ell - a \oplus b\, d\pi_k\\
        &\geq \limsup_{k \to \infty} \int u_{\lambda_k} \circ \ell - a \oplus b \, d\pi_k\\
        &\geq u(1) - \int_M a \, d\mu - \int_M b \, d\nu.
    \end{align*}
    Adding $\int a\, d\mu + \int b \, d\nu$ to both sides and using the monotone convergence theorem for the limit $\lambda'' \uparrow \lambda$, we are able to conclude
    \begin{equation*}
        \int u_\lambda \circ \ell \, d\pi \geq u(1).
    \end{equation*}
    This implies that $\pi \in \Pi^{u_\lambda}_\leq(\mu,\nu)$. By right-continuity of $S(\eta)$ at $\eta = \lambda$, we can see that $S(\lambda) \leq u(1)$, and thus $=u(1)$ since $\pi$ saturates the supremum. Indeed, if $S(\lambda) > u(1)$, then by right-continuity $S(\lambda + \delta) > u(1)$ for some $\delta > 0$, hence $\lambda + \delta \leq \ell_u(\mu,\nu) = \lambda$, a contradiction. We conclude that $\pi$ is optimal and 
    \begin{equation*}
        \int_{M^2} u_\lambda \circ \ell \, d\pi = S(\lambda) = u(1).
    \end{equation*}
    Finally, observe that if $\pi$ is optimal and $u(0) = -\infty$, then necessarily $\pi(\{\ell > 0\}) = 1$ since otherwise $\int u_\lambda \circ \ell \, d \pi = -\infty = u(0) < u(1)$, contradicting optimality.
    \end{proof}

    \begin{Remark}[On the right-continuity of $S_{\mu,\nu}(\eta)$ at $\eta = \ell_u(\mu,\nu)$]
    \label{Remark: Right continuity}
    The assumption of right-continuity of $S_{\mu,\nu}(\eta)$ at $\eta = \lambda$ in the previous result, where
    \begin{equation*}
        S_{\mu,\nu}(\eta):=\sup_{\pi' \in \Pi^{u_\eta}_\leq(\mu,\nu)} \int_{M^2} u_\eta \circ \ell \, d\pi',
    \end{equation*}
    is trivial in many cases. For example, if $u=u_p$ for $0 \neq p < 1$, then
    \begin{equation*}
        (S_p)_{\mu,\nu}(\eta) = \eta^{-p} \ell_p(\mu,\nu)^p, 
    \end{equation*}
    and if $u = u_0$, then
    \begin{equation*}
        (S_0)_{\mu,\nu}(\eta) = \sup_{\pi \in \Pi_\leq^{\log}} \int_{M^2} \log \circ \ell \, d\pi - \log(\eta).
    \end{equation*}
    Two further instances where the right-continuity is easily seen to hold are if $u \circ \ell$ is bounded below on $\mathrm{supp}(\mu \times \nu)$ (which is for example the case if $u$ itself is bounded below), or if there exists an optimal coupling $\pi$ for which $u_\lambda \circ \ell$ is bounded below on $\mathrm{supp}(\pi)$.
    \end{Remark}

    \begin{Remark}[Simple situation where optimal couplings exist]
    \label{Remark: simple situation for opt coupl}
    Note that if $\mathrm{supp}(\mu) \times \mathrm{supp}(\nu)$ is compact and contained in $\{\ell > 0\}$, then clearly $u_\lambda \circ \ell$ is bounded above and below on $\mathrm{supp}(\mu) \times \mathrm{supp}(\nu)$ for every $\lambda > 0$, and it is easy to see that $\ell_u(\mu,\nu) \in (0,\infty)$ in this case. By the previous remark, $S_{\mu,\nu}$ is right-continuous at $\ell_u(\mu,\nu)$.  Thus, for such $\mu$ and $\nu$, we always have an optimal coupling $\pi \in \Pi_\leq(\mu,\nu)$ due to Proposition \ref{prop: existence of optimal coupling}.
    \end{Remark}

The following gluing construction is standard, see e.g.\ Cavalletti--Mondino \cite[Lem.\ 2.4]{CM:20}.

\begin{Lemma}[Gluing of causal couplings]
Let $\mu_1,\mu_2,\mu_3 \in \mathcal{P}(M)$. If $\pi_{12} \in \Pi_{\leq}(\mu_1,\mu_2)$ and $\pi_{23} \in \Pi_{\leq}(\mu_2,\mu_3)$, then there exists $\pi_{123} \in \mathcal{P}(M^3)$ such that $(Pr_{12})_{\#} \pi_{123} = \pi_{12}$, $(Pr_{23})_{\#} \pi_{123} = \pi_{23}$, and the coupling $(Pr_{13})_{\#} \pi_{123} =:\pi_{13}$ is causal between $\mu_1$ and $\mu_3$, i.e.\ $\pi_{13} \in \Pi_{\leq}(\mu_1,\mu_3)$.
\label{lemma:gluing}
\end{Lemma}

We now establish the reverse triangle inequality for $\ell_u$, which first appeared for $u = u_p$ ($p \in (0,1)$) in Eckstein--Miller \cite{EM:17}.

\begin{Proposition}[Reverse triangle inequality]\label{prop: reverse triangle inequality}
Let $u$ be admissible. Then $\ell_u: \mathcal{P}(M)^2 \to [0,\infty] \cup \{-\infty\}$ satisfies the reverse triangle inequality:
\begin{equation}\label{eq: reverse triangle inequality}
    \ell_u(\mu_1,\mu_3) \geq \ell_u(\mu_1,\mu_2) + \ell_u(\mu_2,\mu_3)
\end{equation}
for all $\mu_1,\mu_2,\mu_3 \in \mathcal{P}(M)$, where the right hand side is understood to be $-\infty$ if either of the terms is $-\infty$.
\end{Proposition}
\begin{proof}
    If either term on the right hand side of \eqref{eq: reverse triangle inequality} is $-\infty$, there is nothing to prove.

    Assume first that $\ell_u(\mu_1,\mu_2),\, \ell_u(\mu_2,\mu_3)>0$. Consider $\lambda, \eta > 0$ such that $\ell_u(\mu_1,\mu_2) > \lambda$ and $\ell_u(\mu_2,\mu_3) > \eta$, as well as
    \begin{align*}
        &\sup_{\pi \in \Pi_{\leq}^{u_{\lambda}}(\mu_1,\mu_2)} \int_{M^2}u_\lambda \circ \ell \, d\pi \geq u(1),\\
        &\sup_{\pi \in \Pi_{\leq}^{u_{\eta}}(\mu_2,\mu_3)} \int_{M^2}u_\eta \circ \ell \, d\pi \geq u(1).
    \end{align*}
    Set $\tau:=\lambda + \eta$. Fix $\varepsilon > 0$ and let $\pi_{12} \in \Pi_{\leq}^{u_\lambda}(\mu_1,\mu_2)$ and $\pi_{23} \in \Pi^{u_\eta}_{\leq}(\mu_2,\mu_3)$ such that
    \begin{align*}
        &\int_{M^2}u_\lambda \circ \ell \, d\pi_{12} \geq u(1) - \varepsilon,\\
        &\int_{M^2}u_\eta \circ \ell \, d\pi_{23} \geq u(1) - \varepsilon.
    \end{align*}
    Perform a causal gluing of $\pi_{12}$ and $\pi_{23}$ to obtain a probability measure $\pi_{123} \in \mathcal{P}(M^3)$ concentrated on triples $(x,y,z)$ such that $(x,y) \in \mathrm{supp}(\pi_{12})$, $(y,z) \in \mathrm{supp}(\pi_{23})$ and $x \leq y \leq z$. Then for such triples $(x,y,z)$, using monotonicity and concavity of $u$, as well as the reverse triangle inequality for $\ell$, we get
    \begin{align}\label{eq: proof reverse triangle ineq}
        u\left(\frac{\ell(x,z)}{\tau}\right) \geq u\left(\frac{\lambda}{\tau} \frac{\ell(x,y)}{\lambda} + \frac{\eta}{\tau} \frac{\ell(y,z)}{\eta}\right) \geq \frac{\lambda}{\tau}\, u\left(\frac{\ell(x,y)}{\lambda}\right) + \frac{\eta}{\tau}\,u\left(\frac{\ell(y,z)}{\eta}\right).
    \end{align}
    Integrating this inequality with respect to $\pi_{123}$ (which is possible since the negative part of the right hand side is in $L^1(\pi_{123})$, hence also the negative part of the left hand side), we see that $\pi_{13}:=(pr_1,pr_3)_{\#} \pi_{123} \in \Pi^{u_{\tau}}_{\leq}(\mu_1,\mu_3)$ and
    \begin{equation}\label{eq: reverse triang pi13}
        \int_{M^2} u_\tau \circ \ell \, d\pi_{13} \geq \frac{\lambda }{\tau}(u(1) - \varepsilon) + \frac{\eta}{\tau} (u(1) - \varepsilon) = u(1) - \varepsilon.
    \end{equation}
    Taking the supremum over couplings in $\Pi^{u_{\tau}}_{\leq}(\mu_1,\mu_3)$ and then $\varepsilon \to 0$, we see that $\ell_u(\mu_1,\mu_3) \geq \tau$, and since $\lambda$ and $\eta$ can be chosen arbitrarily close to $\ell_u(\mu_1,\mu_2)$ and $\ell_u(\mu_2,\mu_3)$, respectively, this yields the reverse triangle inequality.

    If, for example, $\ell_u(\mu_1,\mu_2)=0$, consider $\eta > 0$ such that $\ell_u(\mu_2,\mu_3) > \eta$, and take $\lambda=0$ and $\tau:=\eta$. Then, in \eqref{eq: proof reverse triangle ineq},
    \begin{align*}
        u\left(\frac{\ell(x,z)}{\tau}\right) \geq u\left(\frac{\ell(y,z)}{\eta}\right),
    \end{align*}
    where we have used the reverse triangle inequality for $\ell$ and the monotonicity of $u$. In the same way as above, integrating this inequality with respect to $\pi_{123}$, and taking the supremum over couplings in $\Pi^{u_{\tau}}_{\leq}(\mu_1,\mu_3)$ yields the result. The case where $\ell_u(\mu_2,\mu_3)=0$ is analogous.

    If both $\ell_u(\mu_1,\mu_2)=\ell_u(\mu_2,\mu_3)=0$, then there exist causal couplings $\pi_{12}\in\Pi_{\leq}(\mu_1,\mu_2)$ and $\pi_{23}\in\Pi_{\leq}(\mu_2,\mu_3)$. Performing a causal gluing of $\pi_{12}$ and $\pi_{23}$ yields $\pi_{13}\in\Pi_{\leq}(\mu_1,\mu_3)$. Hence, $\ell_u(\mu_1,\mu_3)\geq0=\ell_u(\mu_1,\mu_2)+\ell_u(\mu_2,\mu_3)$, which yields the reverse triangle inequality in this case.

    Finally, if $\ell_u(\mu_1,\mu_3)=0$, it means that for every $\tau>0$,
    \begin{equation*}
        \sup_{\pi \in \Pi^{u_\tau}_{\leq}(\mu_1,\mu_3)} \int_{M^2} u(\tau^{-1}\ell(x,y)) \, d\pi(x,y) < u(1).
    \end{equation*}
    
    Then, arguing as above,
    \begin{align*}
        \sup_{\pi_{12} \in \Pi^{u_\tau}_{\leq}(\mu_1,\mu_2)} \int_{M^2} u(\tau^{-1}\ell(x,y)) \, d\pi_{12}(x,y)  \leq\sup_{\pi \in \Pi^{u_\tau}_{\leq}(\mu_1,\mu_3)} \int_{M^2} u(\tau^{-1}\ell(x,y)) \, d\pi(x,y) < u(1),
    \end{align*}
    so $\ell_u(\mu_1,\mu_2)\leq0$. Analogously, $\ell_u(\mu_2,\mu_3)\leq0$, and the reverse triangle inequality follows immediately.
\end{proof}

We will require the following fine properties of the reverse triangle inequality, in analogy with McCann \cite[Prop.\ 2.9]{McCann:2020}. Recall the notation $S(\eta)$ from Proposition \ref{prop: existence of optimal coupling}.

\begin{Proposition}[Properties of the reverse triangle inequality]\label{prop: propoerties of the reverse ineq}
Let $\mu_1,\mu_2,\mu_3\in\mathcal{P}(M)$. 
\begin{enumerate}
      \item Assume that $\mu_1,\mu_2,\mu_3$ are compactly supported. Let $X_1,X_3\subset M$ be such that $\mu_1(X_1)=1=\mu_3(X_3)$ and further assume that $\ell_u(\mu_1,\mu_2)$ and $\ell_u(\mu_2,\mu_3)$ are positive and finite. If $\ell_u(\mu_1,\mu_3) = +\infty$, or else $\mu_2[Z(X_1\times X_3)]<1$ and $S(\eta) = S_{\mu_1,\mu_3}(\eta)$ is right-continuous at $\eta = \ell_u(\mu_1,\mu_3)$, then the reverse triangle inequality for $\ell_u$ holds strictly.

    \item Conversely, if \begin{enumerate}
                \item $\ell_u(\mu_1,\mu_3)\in(0,\infty)$,
                \item $\ell_u(\mu_1,\mu_2)+\ell_u(\mu_2,\mu_3)=\ell_u(\mu_1,\mu_3)$,
                \item there are optimal couplings realising $\ell_u(\mu_1,\mu_2)$ and $\ell_u(\mu_2,\mu_3)$ (in particular, $\ell_u(\mu_1,\mu_2)$ and $\ell_u(\mu_2,\mu_3) \in (0,\infty)$),
            \end{enumerate}
        then there exists $\omega\in\mathcal{P}(M^3)$ for which $\pi_{ij}:=(P_{ij})_{\#}\omega\in\Pi(\mu_i,\mu_j)$ is $\ell_u$-optimal for each $i<j$ with $i,j\in\{1,2,3\}$, and each $(x,y,z)\in\operatorname{supp}\omega$ satisfies
            \begin{equation}\label{eq: reverse triangle inequality converse}
                \ell(x,y)=s\,\ell(x,z)\quad \text{and}\quad \ell(y,z)=(1-s)\,\ell(x,z)
            \end{equation}
        with $s:=\frac{\ell_u(\mu_1,\mu_2)}{\ell_u(\mu_1,\mu_3)}$. 
    \item Assume the hypotheses of (ii) hold and, in addition, assume $\pi_{13}(S)=1$ for some $S\subset M\times M$. Then $\mu_2$ vanishes outside $Z_s(S)$, where $s\in[0,1]$ is given by (ii). In particular, if $Z_s(x,y)=\{z_s(x,y)\}$ holds $\pi_{13}$ a.e. $(x,y)$, then $\omega=(z_0\times z_s\times z_1)_{\#}\pi_{13}$ and $\mu_2=(z_s)_{\#}\pi_{13}$.    
\end{enumerate}
\end{Proposition}

\begin{proof}
    \begin{enumerate}
        \item We suppose $\ell_u(\mu_1,\mu_3) < +\infty$, otherwise the claim is trivial. For any pair of couplings $\pi_{12}\in\Pi_{\leq}(\mu_1,\mu_2)$ and $\pi_{23}\in\Pi_{\leq}(\mu_2,\mu_3)$, consider the gluing measure denoted by $\pi_{123}\in\mathcal{P}(M^3)$. Since $\mu_1(X_1)=1=\mu_3(X_3)$, $\pi_{123}$ vanishes outside $X_1\times M\times X_3$ by construction of the gluing plan and is concentrated on triples $x \leq y \leq z$, $x \in X_1$, $z \in X_3$.

        Consider a triple $(x,y,z)\in X_1\times M\times X_3$ with $\ell(x,z)\geq0$. For such a triple, the reverse triangle inequality of $\ell$ holds strictly unless $y\in Z(x,z)\subset Z(X_1\times X_3)$. Hence, if $\mu_2[Z(X_1\times X_3)]<1$, there exists a $\mu_2$-positive measure set in $M$ such that it is disjoint with $Z(X_1\times X_3)$, and on which the reverse triangle inequality for $\ell$ will hold strictly for any choice of $y$ in this set. This set $K$ can be chosen to be compact, and let $0 < \varepsilon:=\mu_2(K)$.

        Take $0<\lambda<\ell_u(\mu_1,\mu_2)$ and $0<\eta<\ell_u(\mu_2,\mu_3)$, and write $\tau:=\lambda+\eta$, and let $\pi_{12} \in \Pi^{u_\lambda}_\leq(\mu_1,\mu_2)$, $\pi_{23} \in \Pi^{u_\eta}_\leq(\mu_2,\mu_3)$, $\pi_{123}$ the gluing measure discussed above, and $\pi_{13} \in \Pi^{u_\tau}_\leq(\mu_1,\mu_3)$ the projection to the first and third components. Using strict monotonicity and strict concavity of $u$, as well as compactness of $X_1,K,X_3$, it follows that there exists $\delta > 0$ (uniform in $\lambda, \eta$ up to $\ell_u(\mu_1,\mu_2)$ and $\ell_u(\mu_2,\mu_3)$) such that for all $(x,y,z) \in X_1 \times K \times X_3$,
        \begin{align*}
            u_{\tau}(\ell(x,z)) \geq \frac{\lambda}{\tau}\, u_\lambda(\ell(x,y)) + \frac{\eta}{\tau}\, u_\eta(\ell(y,z)) + \delta/\varepsilon.
        \end{align*}
        
        Integrating this inequality with respect to $\pi_{123}$ as in the proof of Proposition \ref{prop: reverse triangle inequality}, it follows that
        \begin{equation*}
            \int_{M^2}u_{\tau}(\ell(x,z))\, d\pi_{13}(x,z)\geq\frac{\lambda}{\tau}\int_{M^2} u_{\lambda}(\ell(x,y))\, d\pi_{12}+\frac{\eta}{\tau}\int_{M^2}u_{\eta}(\ell(y,z))\, d\pi_{23}+\delta.
        \end{equation*}

        Taking the supremum over all couplings $\pi_{12}$ and $\pi_{23}$, it follows that
        \begin{align*}
            &\sup_{\pi\in\Pi_{\leq}^{u_{\tau}}(\mu_1,\mu_3)}\int_{M^2}u_{\tau}(\ell(x,z))\, d\pi(x,z)  \\ &\geq\frac{\lambda}{\tau}\sup_{\pi_{12}\in\Pi_{\leq}^{u_{\lambda}}(\mu_1,\mu_2)}\int_{M^2} u_{\lambda}(\ell(x,y))\, d\pi_{12}+\frac{\eta}{\tau}\sup_{\pi_{23}\in\Pi_{\leq}^{u_{\eta}}(\mu_2,\mu_3)}\int_{M^2}u_{\eta}(\ell(y,z))\, d\pi_{23}+\delta\\&\geq u(1)+\delta.
        \end{align*}
        This shows that $\ell_u(\mu_1,\mu_3) \geq \tau$. By right-continuity of $S(\sigma)$ at $\sigma = \ell_u(\mu_1,\mu_3)$, there exists $\varepsilon > 0$ such that $\ell_u(\mu_1,\mu_3) \geq \tau + \varepsilon$, and this $\varepsilon > 0$ is uniform in $\lambda, \eta$. Thus, taking the supremum over all such $\lambda, \eta$, we conclude that $\ell_u(\mu_1,\mu_3) > \ell_u(\mu_1,\mu_2) + \ell_u(\mu_2,\mu_3)$.

        \item Since, by assumption, the suprema defining $\ell_u(\mu_1,\mu_2)$ and $\ell_u(\mu_2,\mu_3)$ are attained, in the proof of Proposition \ref{prop: reverse triangle inequality}, we can choose $\lambda=\ell_u(\mu_1,\mu_2)$, $\eta=\ell_u(\mu_2,\mu_3)$, and optimal $\pi_{12}$ and $\pi_{23}$. In particular, in this proof we can take $\varepsilon=0$. Denote by $\omega:=\pi_{123}\in\mathcal{P}(M^3)$ the gluing measure of $\pi_{12}$ and $\pi_{23}$. With this, since we are also assuming that the reverse triangle inequality for $\ell_u$ is saturated, $\tau=\lambda+\eta=\ell_u(\mu_1,\mu_3)$, and every inequality in \eqref{eq: proof reverse triangle ineq} is saturated as well. Analysing what these saturations give, we will get the claim:

        Saturation of the first inequality in \eqref{eq: proof reverse triangle ineq} gives $\ell(x,z)=\ell(x,y)+\ell(y,z)$ holds $\pi_{123}$-a.e., where we are using the fact that $u$ is monotonically increasing and $\ell_u(\mu_1,\mu_3)<\infty$. This condition implies the existence of $s\in[0,1]$ such that $(1-s)\,\ell(x,y)=s\,\ell(y,z)$. These two conclusions combined give \eqref{eq: reverse triangle inequality converse} $\pi_{123}$-a.e.\ On the other hand, saturation of the second inequality in \eqref{eq: proof reverse triangle ineq} gives $\frac{\ell(x,y)}{\lambda}=\frac{\ell(y,z)}{\eta}$. Using this, we can determine the value of $s$ to be $\frac{\lambda}{\tau}$, as desired. Finally, it remains to check that $\pi_{13}$ is $\ell_u$-optimal. This follows immediately by noticing that, under these assumptions, \eqref{eq: reverse triang pi13} gives the definition of $\ell_u$-optimal coupling.

        \item Since $\pi_{13}(S)=1$, $\omega(\tilde{S})=1$, where $\tilde{S}:=\{(x,y,z)\in M^3:\ (x,z)\in S\}$. Indeed, this follows since $(P_{13})_{\#}\omega=\pi_{13}$. Consider $s\in[0,1]$ as obtained in (ii), and let $A\subset M$ be any Borel set disjoint from $Z_s(S)$. It follows that $M\times A\times M$ is disjoint from $\tilde S$. Indeed, this is a consequence of the conclusions of (ii), asserting that for $\omega$-a.e. $(x,y,z)$, $x\leq y\leq z$. Therefore, $\mu_2(A)=\omega(M\times A\times M)=0$, as desired. In particular, suppose $Z_s(x,y)=\{z_s(x,y)\}$ holds for $\pi_{13}$-a.e. $(x,y)$. Then, the previous paragraph gives that $\omega$ vanishes outside the graph of $z_0\times z_s\times z_1:\operatorname{supp}\pi_{13}\to M^3$. Applying {\cite[Lem.\ 3.1]{ahmad2011optimal}}, $\omega=(z_0\times z_s\times z_1)_{\#}\pi_{13}$, and $\mu_2=(z_s)_{\#}\pi_{13}$, as desired.
    \end{enumerate}
\end{proof}

\subsection{Duality for $u$-separated measures}

\begin{Definition}[$u \circ \ell$ - cyclical monotonicity]
Let $u:(0,\infty)\to \R$ be an admissible function. A subset $\Gamma \subseteq M^2_{\leq}$ is called $u \circ \ell$-\emph{cyclically monotone} if for any finite sequence of points $(x_1,y_1), \dots, (x_N, y_N) \in \Gamma$, the following inequality holds:
\begin{align*}
    \sum_{i=1}^N u(\ell(x_i,y_i)) \geq \sum_{i=1}^N u(\ell(x_{i+1},y_i)),
\end{align*}
with the convention $x_{N+1} = x_1$, $u(-\infty) = -\infty$ and $\infty - \infty =:-\infty$ on both sides. A coupling is $u\circ\ell$-\emph{cyclically monotone} if it is concentrated on a $u\circ\ell$-cyclically monotone set. Similarly, the notion of $u \circ \ell_+$-cyclical monotonicity is defined.
\end{Definition}

The following general relationship between cyclical monotonicity and optimality persists in our setting and is an analogue of \cite[Prop.\ 2.8]{CM:20}.

\begin{Proposition}\label{Prop: cyclmonotoneoptimal}
    Let $u:(0,\infty)\to\R$ be an admissible function, and let $\mu,\nu\in\mathcal{P}(M)$. Assume that $\lambda:=\ell_u(\mu,\nu)\in(0,\infty)$, and that there exist measurable functions $a,b:M\to\R$, with $a\oplus b\in L^1(\mu\times\nu)$ and such that $u_{\lambda}\circ\ell\leq a\oplus b$, $\mu\times\nu$-a.e. Then the following holds:
    \begin{enumerate}
        \item If $\pi$ is $\ell_u$-optimal and $\int u_\lambda \circ \ell \, d\pi = u(1)$, then $\pi$ is $u_{\lambda}\circ\ell$-cyclically monotone.
        \item If $\pi\left(M^2_{\ll}\right)=1$ and $\pi$ is $u_{\lambda}\circ\ell$-cyclically monotone and $\int u_\lambda \circ \ell \, d\pi = u(1)$, then $\pi$ is $\ell_u$-optimal.
    \end{enumerate}
    \label{cycmon}
\end{Proposition}

\begin{proof}

    Let $\pi\in\Pi_{\leq}^{u_{\lambda}}(\mu,\nu)$ be optimal. By assumption, $(u_\lambda \circ \ell)_+ \in L^1(\tilde \pi)$ for every $\tilde \pi \in \Pi(\mu,\nu)$, so that the supremum in the definition of $\ell_u(\mu,\nu)$ can be formulated with respect to $\Pi_\leq(\mu,\nu)$. It follows from the same arguments as the one given in the proof of Proposition \ref{prop: existence of optimal coupling} that an optimal $\pi \in \Pi_\leq(\mu,\nu)$ is necessarily in $\Pi^{u_\lambda}_\leq(\mu,\nu)$ and must satisfy (the first equation is by assumption)
    \begin{align*}
        u(1)&=\int_{M^2} u_{\lambda}(\ell(x,y))\,d\pi(x,y)\\
        &=\sup_{\pi\in\Pi_{\leq}(\mu,\nu)}\int_{M^2} u_{\lambda}(\ell(x,y))\,d\pi(x,y).
    \end{align*}
    Hence, $\pi$ is also an optimal coupling for the standard optimal transport problem associated to the cost $u_{\lambda}\circ\ell$. Furthermore, in this case the maximising optimal couplings for the cost $u_{\lambda}\circ\ell$ are the same as for the cost $u_{\lambda}\circ\ell - a \oplus b$, which is non-positive $\mu\times\nu$-a.e. We therefore enter the framework of \cite{BianchiniCaravenna}, which deals with optimal transport problems for arbitrary Borel (nonnegative) cost functions, and the proof uses the same arguments as \cite[Prop.~2.8]{CM:20}.
\end{proof}

\begin{Definition}[$u\circ\ell$-convex functions, $u\circ\ell$-transform and $u\circ\ell$-subdifferential]
     Let $U,V\subset M$. A measurable function $\varphi:U\to\R$ is $u\circ\ell$-\emph{convex relatively to $(U,V)$} if there exists a function $\psi:V\to\R$ such that 
    \begin{equation*}
        \varphi(x)=\sup_{y\in V}u(\ell(x,y))-\psi(y),\quad \forall x\in U.
    \end{equation*}
The function
\begin{equation*}
    \varphi^{(u\circ\ell)}:V\to\R\cup\{\pm \infty\},\quad
    \varphi^{(u\circ\ell)}(y):=\sup_{x\in U}u(\ell(x,y))-\varphi(x)
    \label{eq:ul transform}
\end{equation*}
is called $u\circ\ell$-\emph{transform of $\varphi$}. The $u\circ\ell$-\emph{subdifferential} $\partial_{(u\circ\ell)}\varphi\subset(U\times V)\cap M^2_{\leq}$ is defined by
\begin{equation*}
    \partial_{(u\circ\ell)}\varphi:=\big\{(x,y)\in(U\times V)\cap M^2_{\leq}\ :\ \varphi^{(u\circ\ell)}(y)+\varphi(x)=u\circ\ell(x,y) \big\}.
\end{equation*}
We will often require these notions for $u_\lambda$, where $u$ is admissible and $\lambda > 0$, and will speak of \emph{$u_{\lambda}\circ\ell$-convex functions}, the \emph{$u_{\lambda}\circ\ell$-transform} and the \emph{$u_{\lambda}\circ\ell$-subdifferential}.
\label{concavetransformsubdiff}
\end{Definition}

\begin{Lemma}
\label{Lemma: Concavity along geodesics}
            Let $u:(0,\infty)\to\R$ be an admissible function, and fix a timelike proper-time maximising segment $s\in[0,1]\mapsto x_s\in M$. For each $s\in(0,1)$ and $x\in M$, we have
            \begin{equation}
                u(\ell(x,x_1))\geq s\,u_s\left(\ell(x,x_s)\right)+(1-s)\,u\left(\ell(x_0,x_1)\right),
                \label{eq: variational characterization}
            \end{equation}
            with equality if and only if $x=x_0$.
            \label{lemma: variational characterization}
        \end{Lemma}
        \begin{proof}
            Given $x,x_s,x_1\in M$, we use the reverse triangle inequality of $\ell$ to obtain:
            \begin{equation*}
                \ell(x,x_1)\geq s\,\frac{\ell(x,x_s)}{s}+(1-s)\,\frac{\ell(x_s,x_1)}{1-s},
            \end{equation*}
            with equality only if $x_s$ lies on the maximising segment joining $x$ to $x_1$. Note that the segment from $x_s$ to $x_1$ is indeed, since $x_s$ being internal to the maximising segment joining $x_0$ and $x_1$ forces $(x_s,x_1)\notin\operatorname{sing}(\ell)$. Therefore, by Lemma \ref{lemma: midpoint continuity away from cut locus}, the geodesic joining $x_s$ and $x_1$ is unique. Now, monotonicity and strict concavity of $u$ yield \eqref{eq: variational characterization} (using that $\ell(x_s,x_1) = (1-s) \,\ell(x_0,x_1)$), with equality forcing $\frac{\ell(x,x_s)}{s}=\frac{\ell(x_s,x_1)}{1-s}$. The equation is uniquely solved on the geodesic in question by $x=x_0$. 
        \end{proof}

\begin{Lemma}[Star-shapedness of $u\circ\ell$-convex functions]\label{lemma: starshape of concave}
    Let $U,V\subset M$ be compact, and let $\varphi:U\to\R$ be a $u\circ\ell$-convex function relatively to $(U,V)$. Then, for $t\in(0,1]$, $t^{-1}\varphi$ is a $u_t\circ\ell$-convex function relatively to $(U,Z_t(U,V))$.
\end{Lemma}
\begin{proof}
    For $t=1$, there is nothing to prove. For the rest, we follow the strategy of {\cite[Lem.\ 5.1]{CorderoMcCannSchmuckenschlager01}} and {\cite[Lem.\ A.9]{kell2017interpolation}}.

    Fix $t\in(0,1)$ and $y\in V$. For $x \mapsto u(\ell(x,y))$, we claim that the following representation holds for $m \in U$:
    \begin{equation}\label{eq: claim starshap}
        t^{-1}\,u(\ell(m,y))=\sup_{z\in Z_t(U,y)}\bigg\{ u_t\left(\ell(m,z)\right)+\sup_{\{x\in U\,:\,z\in Z_t(x,y)\}}\Big\{t^{-1}\,(1-t)\,u\left(\ell(x,y)\right)\Big\}\bigg\}.
    \end{equation}
    Indeed, from Lemma \ref{lemma: variational characterization}, ``$\geq$" holds for any $z\in Z_t(U,y)$. Moreover, if $\ell(m,y)\geq0$, choosing $x=m$ saturates \eqref{eq: variational characterization}. Indeed, note that Lemma \ref{lemma: variational characterization} assumes the segment to be timelike, so the case $\ell(m,y)>0$ is covered. If $\ell(m,y)=0$, we distinguish two further cases. If $u(0)\in\R$, then \eqref{eq: variational characterization} becomes $u(0)\geq s\,u_s(0)+(1-s)\,u(0)$, which is trivially an equality. If $u(0)=-\infty$, then both sides of \eqref{eq: variational characterization} are $-\infty$, which is again a trivial equality. Finally, if $\ell(m,y)=-\infty$, by convention $u(\ell(m,y))=-\infty$, and the set on the right hand side is empty, and therefore this is $-\infty$. Thus, the representation holds.

    Note that \eqref{eq: claim starshap} implies, in particular, that the function $m\in U\mapsto t^{-1}\,u(\ell(m,y))$ is the $u_t\circ\ell$-transform of the function 
    \begin{equation*}
        \psi(z)=-\sup_{\{x\in U\,:\,z\in Z_t(x,y)\}}\Big\{t^{-1}\,(1-t)\,u\left(\ell(x,y)\right)\Big\},
    \end{equation*}
    and therefore it is $u_t\circ\ell$-convex relative to $(U,Z_t(U,y))$, and we conclude that each $m\in U\mapsto t^{-1}\,u(\ell(m,y))$ is $u_t\circ\ell$-convex relative to $(U,Z_t(U,V))$. The latter is true because, if we choose $z\in Z_t(U,V)$ such that $z\notin Z_t(U,y)$, the second term on the right hand side of \eqref{eq: claim starshap} is $-\infty$, and therefore can be neglected when computing the supremum over all $z\in Z_t(U,V)$.

    Now, let $\varphi$ be an arbitrary $u\circ\ell$-convex function. It holds that $\varphi=\left(\varphi^{(u\circ\ell)}\right)^{(u\circ\ell)}$, hence
    \begin{equation*}
        t^{-1}\,\varphi(x)=\sup_{y\in V} t^{-1}\,u(\ell(x,y))-t^{-1}\,\varphi^{(u\circ\ell)}(y).
    \end{equation*}
    But every function of the form
    \begin{equation*}
        t^{-1}\,u(\ell(x,y))-t^{-1}\,\varphi^{(u\circ\ell)}(y)
    \end{equation*}
    is $u_t\circ\ell$-convex relative to $(U,Z_t(U,V))$. Indeed, this follows by rewriting $t^{-1}\,u(\ell(x,y))$ using the claim \eqref{eq: claim starshap}. Therefore, since $\varphi(U)\subset\R$, the supremum of these $u\circ\ell$-convex functions relative to $(U,Z_t(U,V))$ is also $u_t\circ\ell$-convex relative to $(U,Z_t(U,V))$, which yields the result. 
\end{proof}

Let us now introduce the optimisation problem which is dual to that of determining $\ell_u(\mu,\nu)$:
\begin{equation}
    \tilde C_u(\mu,\nu):=\inf\Bigg\{\eta>0\ :\ \inf_{\varphi\in L^1(\mu):\varphi^{(u_\eta \circ \ell)} \in L^1(\nu)}\Big\{\int\varphi\, d\mu+\int\varphi^{(u_{\eta}\circ\ell)}\, d\nu\Big\}\leq u(1)\Bigg\}.
\end{equation}
Here, $\varphi^{(u_\eta \circ \ell)}(y):=\sup_{x \in \mathrm{supp}(\mu)} u_\eta(\ell(x,y)) - \varphi(x)$, $y \in M$. For $\mu,\nu\in\mathcal{P}(M)$ such that $c:=\tilde C_u(\mu,\nu)\in(0,\infty)$, a function $\varphi\in L^1(\mu)$ is a \emph{Kantorovich potential} if $\varphi^{(u_c \circ \ell)} \in L^1(\nu)$ and
\begin{equation}
\label{eq: Kantorovich potential inequality}
    \int\varphi\, d\mu+\int\varphi^{(u_{c}\circ\ell)}\, d\nu\leq u(1).
\end{equation}

While the Orlicz-type dual problem of determining $\tilde C_u(\mu,\nu)$ (and its relation to $\ell_u(\mu,\nu)$) is difficult to handle in general, we identify the following condition which helps us do that. It is analogous to and generalises McCann's notion of $p$-separation in the case $u = u_p$, $p \in (0,1)$, cf.\ \cite[Def.\ 4.1]{McCann:2020}. Note that by convention, our Kantorovich potential(s) arise by $u_\lambda \circ \ell$-transform, where $\lambda = \ell_u(\mu,\nu)$, so they are related to the usual Kantorovich potentials (e.g.\ \cite{McCann:2020, MS:22, CM:20}, etc.) by a factor of $\lambda^p$ if $u = u_p$.

\begin{Definition}[$u$-separated]
\label{Definition: u-separeted new}
Let $u:(0,\infty) \to \R$ be admissible. We say that a pair $(\mu,\nu) \in \mathcal{P}_c(M)^2$ is \emph{$u$-separated} if there exists $\pi \in \Pi(\mu,\nu)$, $\varphi:\mathrm{supp}(\mu) \to \R \cup \{+\infty\}$ and $\psi:\mathrm{supp}(\nu) \to \R \cup \{+\infty\}$ lower semicontinuous, as well as $\lambda > 0$, such that
\begin{enumerate}
    \item $u_\lambda \circ \ell \leq \varphi \oplus \psi$ on $\mathrm{supp}(\mu \times \nu)$,
    \item $\mathrm{supp}(\pi) \subseteq S:=\{(x,y) \in \mathrm{supp}(\mu \times \nu) : \varphi(x) + \psi(y) = u_\lambda(\ell(x,y))\} \subseteq \{ \ell > 0\}$,
    \item $\int_M \varphi \, d\mu + \int_M \psi \, d\nu = u(1)$.
\end{enumerate}
\end{Definition}

\begin{Theorem}[Duality by $u$-separation]
\label{Theorem: Duality by u-separation}
Let $(\mu,\nu) \in \mathcal P_c(M)^2$ be $u$-separated for some admissible $u$. If $\lambda > 0$ is as in the definition of $u$-separation, then
\begin{enumerate}
    \item $\lambda = \ell_u(\mu,\nu ) \in (0,\infty)$, $\pi$ is optimal and satisfies $\int u_\lambda \circ \ell \, d\pi = u(1)$;
    \item $(\varphi,\psi) = (\varphi,\varphi^{(u_\lambda \circ \ell)})$ and $\varphi = \psi^{(u_\lambda \circ \ell)}$ on $X \times Y:=\mathrm{supp}(\mu) \times \mathrm{supp}(\nu)$;
    \item the set $S$ is compact and $u_\lambda \circ \ell$-cyclically monotone;
    \item $\tilde C_u(\mu,\nu) = \ell_u(\mu,\nu) = \lambda$, and $\varphi$ is a Kantorovich potential;
    \item the extensions $\varphi:=\psi^{(u_\lambda\circ\ell)}$ and $\psi:=\varphi^{(u_\lambda\circ\ell)}$ are semiconvex Lipschitz functions on neighbourhoods of $X$ and $Y$, respectively, with Lipschitz and semiconvexity constants given by those of $u_\lambda\circ\ell$ on $S$.
\end{enumerate}
\end{Theorem}
\begin{proof}
    First, given that $\mu,\nu$ are compactly supported, we may consider all causal couplings in the definition of $\ell_u(\mu,\nu)$ without changing its value.
    \begin{enumerate}
        \item We first show that $\lambda \leq \ell_u(\mu,\nu)$: Indeed,
        \begin{equation*}
            \int u_\lambda \circ \ell \, d\pi = \int (\varphi \oplus \psi) d\pi = \int \varphi \, d\mu + \int \psi \, d\nu = u(1).
        \end{equation*}
        To see that $\ell_u(\mu,\nu) = \lambda$, suppose that $\lambda' > \lambda$, then for any $\pi' \in \Pi_\leq(\mu,\nu)$, using inverse monotonicity of $\lambda' \mapsto u_{\lambda'} \circ \ell$ we obtain
        \begin{equation*}
            \int u_{\lambda'}\circ \ell \, d\pi' \leq \int u_{\lambda } \circ \ell \, d\pi' \leq \int \varphi \, d\mu + \int \psi \, d\nu = u(1).
        \end{equation*}
        Now suppose that for some $\lambda' > \lambda$,
        \begin{equation*}
            \sup_{\pi' \in \Pi_\leq(\mu,\nu)} \int u_{\lambda'} \circ \ell \, d\pi' = u(1).
        \end{equation*}
        Just like in the proof of Proposition \ref{prop: existence of optimal coupling}, we can use the compactness of $\Pi_\leq(\mu,\nu)$ to find a causal coupling $\pi'$ of $(\mu,\nu)$ realising this supremum, i.e.,
        \begin{equation*}
            \int u_{\lambda'} \circ \ell \, d\pi' = u(1).
        \end{equation*}
        By monotonicity of $u$, it is easily seen that $\pi'$ cannot be concentrated on $\{\ell = 0\}$, thus there exists $\varepsilon > 0$ such that $m_\varepsilon:=\pi'(\{\ell \geq \varepsilon\}) > 0$. Using strict inverse monotonicity of $\lambda' \mapsto u_{\lambda'}(\ell(x,y))$ whenever $\ell(x,y) > 0$, as well as continuity, since the set $\{\ell \geq \varepsilon\} \cap \mathrm{supp}(\pi')$ is compact, there is some $C_\varepsilon > 0$ such that
        \begin{equation*}
            u_\lambda \circ \ell - u_{\lambda'}\circ \ell \geq C_\varepsilon > 0
        \end{equation*}
        on $\{\ell \geq \varepsilon\} \cap \mathrm{supp}(\pi')$. But then
        \begin{align*}
            u(1) &= \int u_{\lambda'} \circ \ell \, d\pi' = \int_{\{\ell \geq \varepsilon\}} u_{\lambda'} \circ \ell \, d\pi' + \int_{\{\ell < \varepsilon\}} u_{\lambda'} \circ \ell \, d\pi'\\
            &\leq -C_\varepsilon m_\varepsilon + \int_{\{\ell \geq \varepsilon\}} u_\lambda \circ \ell \, d\pi' + \int_{\{\ell < \varepsilon\}} u_\lambda \circ \ell \, d\pi' \\
            &= -C_\varepsilon m_\varepsilon + \int u_\lambda \circ \ell \, d \pi'\\
            &\leq - C_\varepsilon m_\varepsilon + \int \varphi \, d\mu + \int \psi \, d\nu < u(1),
        \end{align*}
        a contradiction. We conclude that $\lambda = \ell_u(\mu,\nu) \in (0,\infty)$ and that $\pi$ is optimal.
        \item Consider the $u_\lambda\circ\ell$-conjugate of $\psi$ in $M$,
                \begin{equation}
                    \psi^{(u_\lambda\circ\ell)}(x)=\sup_{y\in Y} u_\lambda(\ell(x,y))-\psi(y).
                    \label{eq: duality by u separation}
                \end{equation}
                The result will follow if we establish that $\psi^{(u_\lambda\circ\ell)}=\varphi$ in $X$. To this aim, notice that, since $\nu$ is compactly supported and both $-\psi$ and $u\circ\ell$ are upper semicontinuous, the supremum in the previous equation is attained.
                Fix $\bar{x}\in X$. The inequality $\varphi(\bar{x})\geq\psi^{(u_\lambda\circ\ell)}(\bar{x})$ follows from Definition \ref{Definition: u-separeted new}(i). On the other hand, using Definition \ref{Definition: u-separeted new}(ii) and compactness of $X$ and $Y$, and since $\pi\in\Pi(\mu,\nu)$, there exists $\bar{y}\in Y$ such that $(\bar{x},\bar{y})\in S$. Hence, this $\bar{y}$ must maximise \eqref{eq: duality by u separation}, and $\varphi$ and $\psi^{(u_\lambda\circ\ell)}$ agree on $X$. Since $\varphi$ is only defined on $X$, it can be extended to $M$ by imposing that it agrees with $\psi^{(u_\lambda\circ\ell)}$ outside $X$, which yields the first part of the proof.
                The second part of the proof requires to show $\varphi^{(u_\lambda\circ\ell)}=\psi$ in $Y$, which can be done analogously.
                \item Compactness of $S$ is clear, as it is a closed set contained in the compact set $X \times Y$. To see $u_\lambda\circ\ell$-cyclical monotonicity of $S$, choose $k\in\N$ and a permutation $\sigma$ of $\{1,\dots,k\}$. For $(x_1,y_1),\dots,(x_k,y_k)\in S$, we have
                \begin{equation*}
                    \sum_{i=1}^k u_\lambda(\ell(x_i,y_i))=\sum_{i=1}^k \varphi(x_i)+\psi(y_i)=\sum_{i=1}^k \varphi(x_i)+\psi(y_{\sigma(i)})\geq \sum_{i=1}^k u_\lambda(\ell(x_i,y_{\sigma(i)})),
                \end{equation*}
                where the last inequality is a consequence of the $u$-separation. This gives the result.
                \item First, since $\psi = \varphi^{(u_\lambda \circ \ell)}$ on $Y = \mathrm{supp}(\nu)$ and
                \begin{equation*}
                    \int \varphi \, d\mu + \int \psi \, d\nu = u(1),
                \end{equation*}
                it follows that $\tilde C_u(\mu,\nu) \leq \lambda$. Now suppose that $\eta < \lambda$, and take any $\tilde \varphi \in L^1(\mu)$ such that $\tilde \varphi^{(u_\eta \circ \ell)} \in L^1(\nu)$. Since $\tilde \varphi(x) + \tilde \varphi^{(u_\eta \circ \ell)}(y) \geq u_\eta(\ell(x,y))$, for any $\pi' \in \Pi_\leq(\mu,\nu)$ we have that 
                \begin{align*}
                    \int \tilde \varphi \, d\mu + \int \tilde \varphi^{(u_\eta \circ \ell)} \, d\nu \geq \int u_\eta \circ \ell \, d\pi'.
                \end{align*}
                Given that $\eta < \lambda$, taking a supremum over $\pi'$ shows that
                \begin{align*}
                    \int \tilde \varphi \, d\mu + \int \tilde \varphi^{(u_\eta \circ \ell)} \, d\nu \geq u(1).
                \end{align*}
                Take now $\pi' = \pi$, the optimal plan for $\ell_u(\mu,\nu) = \lambda$. Again using strict inverse monotonicity of $\eta \mapsto u_\eta \circ \ell$, given that $\pi$ has support compactly contained in $\{\ell > 0\}$, there is some $C > 0$ such that $u_\eta \circ \ell \geq u_\lambda \circ \ell + C$ on $\mathrm{supp}(\pi)$. Thus, for any $\tilde \varphi \in L^1(\mu)$,
                \begin{align*}
                    \int \tilde \varphi \, d\mu + \int \tilde \varphi^{(u_\eta \circ \ell)} \, d\nu &= \int \tilde \varphi \oplus \tilde \varphi^{(u_\eta \circ \ell)} \, d\pi \geq \int u_\eta \circ \ell \, d\pi\\
                    &\geq C + \int u_\lambda \circ \ell \, d\pi = C + u(1).
                \end{align*}
                Thus $\tilde C_u(\mu,\nu) \geq \lambda$, so altogether $\tilde C_u(\mu,\nu) = \lambda$ and we conclude that $\varphi$ is a Kantorovich potential.
                \item By (iii), $S$ is a compact set. Therefore, its Riemannian distance (using the auxiliary metric $\tilde{g}$) from $\{\ell\leq0\}$ is positive. Let this distance be denoted by $3R:=d_{\tilde{g}\oplus\tilde{g}}(\{\ell\leq0\},S)>0$. Therefore, $S_{2R}\subset\{\ell>0\}$, where $S_{2R}$ denotes the closed tubular neighbourhood of $S$ with radius $2R$. In particular, since by Corollary \ref{cor: twist and non degeneracy}(i), $u_\lambda\circ\ell$ is locally Lipschitz and semiconvex in $\{\ell>0\}$, the same properties hold in $S_{2R}$. We will use these to obtain sufficiently small (tubular) neighbourhoods of $X$ and $Y$ where the claim holds.
                Let $\bar{x}\in X$. By definition of $S$, there exists $\bar{y}\in Y$ such that $(\bar{x},\bar{y})\in S$. Then, for any $x\in X$,
                \begin{equation}\label{eq: lipschitz kantorovich potentials}
                    \begin{aligned}
                        \psi^{(u_\lambda\circ\ell)}(x)&\geq u_\lambda(\ell(x,\bar{y}))-\psi(\bar{y})\\
                    &\geq u_\lambda(\ell(\bar{x},\bar{y}))-\vert\vert u_\lambda\circ\ell\vert\vert_{C^{0,1}(B_R(\bar{x}))}d_{\tilde{g}}(x,\bar{x})-\psi(\bar{y}) \\
                    &\geq \psi^{(u_\lambda\circ\ell)}(\bar{x})-\vert\vert u_\lambda\circ\ell\vert\vert_{C^{0,1}(S_{2R})}d_{\tilde{g}}(x,\bar{x}),
                    \end{aligned}
                \end{equation}
                where in the second line we are using the local Lipschitzness of $u_\lambda\circ\ell$. Exchanging above the roles of $x$ and $\bar{x}$, it follows that $\psi^{(u_\lambda\circ\ell)}$ is Lipschitz continuous in $X$, and the Lipschitz constant is bounded by $\vert\vert u_\lambda\circ\ell\vert\vert_{C^{0,1}(S_{2R})}$. Analogously, one would obtain the same bound for $\varphi^{(u_\lambda\circ\ell)}$ on $Y$.
                Note how the argument derived in \eqref{eq: lipschitz kantorovich potentials} cannot be used if $x\notin X$, since it is no longer possible to exchange the role of $x$ and $\bar{x}$, as there is no duality \textit{a priori} for $x$. However, one can deduce lower semicontinuity of $\psi^{(u_\lambda\circ\ell)}$, since the bound in \eqref{eq: lipschitz kantorovich potentials} still holds, and taking the limit $x\to\bar{x}$,
                \begin{equation}\label{eq: lsc kantorovich}
        \lim_{x\to\bar{x}}\psi^{(u_\lambda\circ\ell)}(x)\geq\psi^{(u_\lambda\circ\ell)}(x).
                \end{equation}
                In order to conclude Lipschitzness of $\psi^{(u_\lambda\circ\ell)}$ and $\varphi^{(u_\lambda\circ\ell)}$ in neighbourhoods of $X$ and $Y$, respectively, let $r>0$ be sufficiently small so that $X_r$ inherits compactness from $X$, and $Y_r$ inherits compactness from $Y$ (this can be done since our manifold is Hausdorff and second countable). Moreover, $r>0$ can be chosen so that $S_{(r,0)}:=\{(x,y)\in X_r\times Y:\ \psi^{(u_\lambda\circ\ell)}(x)+\psi(y)=u_\lambda(\ell(x,y))\}$ is contained in $S_{R}$ (respectively $S_{(0,r)}$). Indeed, assume for contradiction that there exists a sequence $(x_k,y_k)\in S_{(\frac{1}{k},0)}\setminus S_R$, $k\in\N$. Since this is in particular contained in $S_{(\frac{1}{k},0)}$, for $k\in\N$ sufficiently large, this is contained in the compact set $X_{\frac{1}{k}}\times Y$, so there exists a converging subsequence (without loss of generality, let us not rename it) to $(x_{\infty},y_{\infty})$. Note that for every $k\in\N$, by construction it holds that
                \begin{equation*}
                    \psi^{(u_\lambda\circ\ell)}(x_k)+\psi(y_k)=u_\lambda(\ell(x_k,y_k)).
                \end{equation*}
                Now, by assumption, $\psi$ is lower semicontinuous. By \eqref{eq: lsc kantorovich}, $\psi^{(u_\lambda\circ\ell)}$ is also lower semicontinuous. Finally, by Corollary \ref{cor: twist and non degeneracy}(i), $u_\lambda\circ\ell$ is upper semicontinuous. Bringing these together, and taking the limit above, we obtain
                \begin{align*}
u_\lambda(\ell(x_\infty,y_\infty))\leq\psi^{(u_\lambda\circ\ell)}(x_\infty)+\psi(y_\infty)&\leq\lim_{k\to\infty}\psi^{(u_\lambda\circ\ell)}(x_k)+\psi(y_k)\\ &=\lim_{k\to\infty}u_\lambda(\ell(x_k,y_k))\leq u_\lambda(\ell(x_\infty,y_\infty)),
                \end{align*}
                which forces everything above to be an equality. This implies that $(x_{\infty},y_{\infty})\in S$, but in particular, this would give that for sufficiently big $k\in\N$, $(x_k,y_k)\in S_{R}$, which is a contradiction.
                To finally conclude Lipschitzness, let $\bar{x}\in X_r$ and $x\in X_r\cap B_R(\bar{x})$. As above, there exists $\bar{y}\in Y$ such that $(\bar{x},\bar{y})\in S_{(r,0)}\subset S_{R}$, hence
                \begin{align*}
                    \psi^{(u_\lambda\circ\ell)}(x)&\geq u_\lambda(\ell(x,\bar{y}))-\psi(\bar{y})\\
                    &\geq u_\lambda(\ell(\bar{x},\bar{y}))-\vert\vert u_\lambda\circ\ell\vert\vert_{C^{0,1}(B_R(\bar{x}))}d_{\tilde{g}}(x,\bar{x})-\psi(\bar{y}) \\
                    &\geq \psi^{(u_\lambda\circ\ell)}(\bar{x})-\vert\vert u_\lambda \circ\ell\vert\vert_{C^{0,1}(S_{2R})}d_{\tilde{g}}(x,\bar{x}),
                \end{align*}
            and in this case one can interchange the roles of $x$ and $\bar{x}$, yielding a Lipschitz bound for $\psi^{(u_\lambda\circ\ell)}$ on $X_r$. Analogously, one derives the same bound for $\varphi^{(u_\lambda \circ\ell)}$ on $Y_r$.
            Finally, let us show the claimed semiconvexity. Denote by $C_{2R}$ the semiconvexity constant of $u\circ\ell$ in $S_{2R}$ given by Corollary \ref{cor: twist and non degeneracy}(i). Choose $(\bar{x},\bar{y})\in S_{(r,0)}\subset S_R$ and $w\in T_{\bar{x}}M$ such that $\vert w\vert_{\tilde{g}}<R$. Then
            \begin{align*}
                \psi^{(u_\lambda\circ\ell)}(\exp_{\bar{x}}^{\tilde{g}}w)&+\psi^{(u_\lambda\circ\ell)}(\exp_{\bar{x}}^{\tilde{g}}-w)-2\,\psi^{(u_\lambda\circ\ell)}(\bar{x})\\
                &\geq u_\lambda(\ell(\exp_{\bar{x}}^{\tilde{g}}w,\bar{y}))+u_\lambda(\ell(\exp_{\bar{x}}^{\tilde{g}}-w,\bar{y}))-2\, u_\lambda(\ell(\bar{x},\bar{y}))\\
                &\geq C_{2R}\,\vert w\vert^2_{\tilde{g}}.
            \end{align*}
            This yields semiconvexity of $\psi^{(u_\lambda\circ\ell)}$ in $X_r$. Analogously, $\varphi^{(u_\lambda\circ\ell)}$ is semiconvex in $Y_r$, which concludes the proof.
    \end{enumerate}
\end{proof}

 \begin{Remark}[Simple situation where $u$-separation holds]
 \label{Remark: simple u separation}
 Let $(\mu,\nu) \in \mathcal P_c(M)^2$ such that $\mathrm{supp}(\mu) \times \mathrm{supp}(\nu) \subseteq \{\ell > 0\}$. By Proposition \ref{prop: existence of optimal coupling} and Remark \ref{Remark: simple situation for opt coupl}, $\lambda:=\ell_u(\mu,\nu) \in (0,\infty)$, and there exists an optimal coupling $\pi \in \Pi^{u_\lambda}_\leq(\mu,\nu) = \Pi_\leq(\mu,\nu)$ such that
 \begin{equation*}
     u(1) = \int u_\lambda\circ \ell \, d\pi = \sup_{\tilde \pi \in \Pi_{\leq}(\mu,\nu)} \int u_\lambda \circ \ell \, d\tilde \pi.
 \end{equation*}
 Given that $u \circ \ell$ is continuous and finite-valued on the compact set $\mathrm{supp}(\mu) \times \mathrm{supp}(\nu)$, standard Kantorovich duality results (cf.\ Villani \cite{Villani:2009}) produce a pair $(\varphi, \varphi^{(u_\lambda \circ \ell)})$, where $\varphi$ is $u_\lambda \circ \ell$-convex, such that
 \begin{align*}
     u(1) &= \int u_\lambda \circ \ell \, d\pi = \sup_{\tilde \pi \in \Pi_{\leq}(\mu,\nu)} \int u_\lambda \circ \ell \, d\tilde \pi\\ 
     &= \inf_{\tilde \varphi \in L^1(\mu)} \int \tilde \varphi \, d\mu + \int \tilde \varphi^{(u_\lambda \circ \ell)} \, d\nu \\
     &= \int \varphi \, d\mu + \int  \varphi^{(u_\lambda \circ \ell)}.
 \end{align*}
 Thus, the data $(\lambda,\pi,\varphi,\varphi^{(u_\lambda \circ \ell)})$ is easily seen to verify $u$-separation for $(\mu,\nu)$.
\end{Remark}

\subsection{Geodesics of probability measures}

\begin{Definition}[$u$-geodesic]
    Let $u:(0,\infty)\to\R$ be an admissible function. We say that a map $s\in[0,1]\mapsto\mu_s\in\mathcal{P}(M)$ is a \emph{$u$-geodesic} if
    \begin{equation*}
        \ell_u(\mu_s,\mu_t)=(t-s)\,\ell_u(\mu_0,\mu_1)\in(0,\infty)
    \end{equation*}
    for each $0\leq s<t\leq1$. It will be denoted throughout as $(\mu_s)_{s\in[0,1]}$. If $u=u_p$ for $p\in(-\infty,1)$, we will call such a map simply \emph{$p$-geodesic}.
\end{Definition}

\begin{Proposition}[Interpolants inherit compact support]\label{prop: interpolants inherit compactness}
            Let $u:(0,\infty)\to\R$ be an admissible function and let $(\mu_s)_{s\in[0,1]}\subset\mathcal{P}(M)$ be a $u$-geodesic. If $\mu_0$ and $\mu_1$ are compactly supported and $S_{\mu_0,\mu_1}(\eta)$ is right-continuous at $\eta =\ell_u(\mu_0,\mu_1)$, then $\operatorname{supp}\mu_s\subset Z_s\left(\operatorname{supp}(\mu_0\times\mu_1)\right)$, and the latter are compact for every $s\in[0,1]$.
\end{Proposition}
\begin{proof}
            For the $u$-geodesic $(\mu_s)_{s\in[0,1]}$, write $X_s:=\operatorname{supp}\mu_s$. By assumption, $X_0$ and $X_1$ are compact. Hence, by Lemma \ref{lemma: midpoints inherit compactness}, $Z(X_0\times X_1)$ is compact, and so is $Z_s(X_0\times X_1)$ for any $s\in[0,1]$. Fix $s\in[0,1]$, from the definition of $u$-geodesic we get $\ell_u(\mu_0,\mu_s)+ \ell_u(\mu_s,\mu_1)=\ell_u(\mu_0,\mu_1)$. Proposition \ref{prop: propoerties of the reverse ineq}(i) asserts that in this case $\mu_s[Z(X_0\times X_1)]=1$. Together with the compactness of $Z(X_0\times X_1)$, this gives that $\mu_s$ is compactly supported. Finally, $\sup_{Z(X_0\times X_1)}u\circ\ell<\infty$, given the compactness of $Z(X_0\times X_1)$ and the upper semicontinuity of $u\circ\ell$ established in Corollary \ref{cor: twist and non degeneracy}(i)(a). Thus, by Proposition \ref{prop: existence of optimal coupling}, we can find optimal couplings for both $\ell_u(\mu_0,\mu_s)$ and $\ell_u(\mu_s,\mu_1)$. Then, we can apply Proposition \ref{prop: propoerties of the reverse ineq}(iii) (taking $S=X_0\times X_1$), get $\mu_s(Z_s(X_0\times X_1))=1$, and hence $\operatorname{supp}\mu_s\subset Z_s\left(\operatorname{supp}(\mu_0\times\mu_1)\right)$, as desired.
\end{proof}

\begin{Theorem}[Existence of $u$-geodesics]\label{thm: existence of u geod}
Let $(\mu,\nu) \in \mathcal{P}(M)^2$ such that $\lambda:=\ell_u(\mu,\nu) \in (0,\infty)$, and suppose there exists an optimal $\pi \in \Pi^{u_\lambda}_\leq(\mu,\nu)$ which is concentrated on $\{\ell > 0\}$. Then $\mu_s:= (\bar z_s)_{\#} \pi$ is a $u$-geodesic from $\mu$ to $\nu$. Moreover, for each $s < t$ such that $S_{\mu_s,\mu_t}(\eta)$ is right continuous at $(t-s)\lambda$, $(\bar z_s \times \bar z_t)_{\#} \pi$ is optimal for $(\mu_s,\mu_t)$. This optimal plan saturates the inequality in \eqref{eq: optimal coupling inequality} if and only if $\pi$ does. Moreover, if $\mu,\nu$ are compactly supported, the optimal coupling $\pi$ is unique and satisfies $\pi(\mathrm{sing}(\ell)) = 0$, then the $u$-geodesic $(\mu_s)_{s\in[0,1]}$ is unique.
\end{Theorem}
\begin{proof}
    Fix $s,t \in [0,1]$, $s < t$. Since, by assumption, for $\pi$-a.e.\ $(x,y)$, $s \mapsto \bar{z}_s(x,y)$ is a timelike geodesic parametrised proportional to arclength (so that $\ell(\bar z_s(x,y), \bar z_t(x,y)) = (t-s)\,\ell(x,y)$ for $\pi$-a.e.\ $(x,y)$), it follows that $\pi_{s,t}:=(\bar z_s \times \bar z_t)_{\#} \pi \in \Pi^{u_{(t-s) \lambda}}_\leq(\mu_s,\mu_t)$, as well as
    \begin{align*}
        \int u_{(t-s)\lambda} \circ \ell \, d\pi_{s,t} = \int u_{(t-s)\lambda} \circ \ell(\bar z_s(x,y),\bar z_t(x,y)) \, d\pi =  \int u_\lambda \circ \ell (x,y) \, d\pi \geq u(1),
    \end{align*}
    and clearly this is an equality for $\pi_{s,t}$ if and only if it is for $\pi$. Thus, $\ell_u(\mu_s,\mu_t) \geq (t-s)\, \lambda = (t-s)\, \ell(\mu_0,\mu_1)$.
This implies that $(\mu_s)$ saturates the reverse triangle inequality for $\ell_u$ and is thus a $u$-geodesic. Moreover, the above calculation shows that $\pi_{s,t}$ is optimal for $(\mu_s,\mu_t)$. Under the right-continuity assumption on $S_{\mu_s,\mu_t}$, it would follow that $S_{\mu_s,\mu_t}((t-s)\lambda) \leq u(1)$, thus $\pi_{s,t}$ is optimal in this case. The additional claim follows as in \cite[Thm.\ 2.11]{McCann:2020}, using that, in this case, each $\mu_s$ has compact support by Proposition \ref{prop: interpolants inherit compactness}, as well as the fine properties of the reverse triangle inequality (cf.\ Proposition \ref{prop: propoerties of the reverse ineq}).
\end{proof}

\begin{Proposition}[Lagrangian trajectories do not cross]\label{prop: lagrangian trajectories dont cross}

Fix an admissible function $u:(0,\infty)\to\R$. Let $(x,y),(x',y')\in M^2$, and fix $s\in(0,1)$. If $Z_s(x,y)\cap Z_s(x',y')\neq\emptyset$, and
            \begin{equation}
                u(\ell(x,y'))+u(\ell(x',y))\leq u(\ell(x,y))+u(\ell(x',y'))
                \label{eq: lagrangian trajectories dont cross}
            \end{equation}
holds, then $(x,y)=(x',y')$.
\end{Proposition}
\begin{proof}
        Since $u$ is admissible, it is strictly increasing and strictly concave. Together with the reverse triangle inequality satisfied by $\ell$, we have for $m \in Z_s(x,y) \cap Z_s(x',y')$
            \begin{equation*}
                \begin{aligned}
                    u\circ\ell(x,y')&\geq u\left(s\,\frac{\ell(x,m)}{s}+(1-s)\,\frac{\ell(m,y')}{1-s}\right)\\ &\geq s\,u\left(\frac{\ell(x,m)}{s}\right)+(1-s)\,u\left(\frac{\ell(m,y')}{1-s}\right)\\&=s\,u\circ\ell(x,y)+(1-s)\,u\circ\ell(x',y').
                \end{aligned}
            \end{equation*}
By strict monotonicity, the first inequality is strict unless $m$ lies on a maximising timelike geodesic from $x$ to $y'$, while the second is strict unless $\ell(x,m)/s = \ell(m,y')/(1-s)$ and thus $\ell(x,y) = \ell(x',y')$, by strict concavity of $u$. The analogous inequality for $u \circ \ell(x',y)$ is strict unless $m$ lies on a maximising timelike geodesic from $x'$ to $y$. Summing these two estimates, any strictness in the inequalities would contradict our assumption, thus equalities hold throughout. But the affine parametrisation, combined with the previous observations on $x,x',m,y,y'$, forces $x = x'$ and $y = y'$.
\end{proof}

The above result will often be used for the admissible function $u_\lambda$, with $u$ a fixed admissible function and $\lambda > 0$.

\begin{Proposition}[(Lipschitz) continuity of inverse maps]
\label{Prop: Lipschitz continuous inverse map}
Let $u$ be admissible, $(\mu_0,\mu_1) \in \mathcal{P}_c(M)^2$ be $u$-separated and $s \in (0,1)$. Then there exists a Lipschitz continuous map $W:Z_s(S) \to \mathrm{supp}(\mu_0) \times \mathrm{supp}(\mu_1)$ (where $S$ is the set given in Definition \ref{Definition: u-separeted new} of $u$-separation) with the property that, whenever $(\mu_t)_{t\in[0,1]}$ is a $u$-geodesic from $\mu_0$ to $\mu_1$, then $W_{\#}\mu_s$ is optimal for $(\mu_0,\mu_1)$ and the restriction of $\bar z_s \circ W$ to $\bar z_s(S)$ is the identity map.
\end{Proposition}
\begin{proof}
    We sketch the argument, which is analogous to McCann \cite[Cor.\ 5.2]{McCann:2020} and refer there for more details.

    The duality obtained from the $u$-separation of $(\mu_0,\mu_1)$ (cf.\ Theorem \ref{Theorem: Duality by u-separation}) provides semiconvex optimisers $(\varphi,\varphi^{(u_\lambda \circ \ell)})$, where $\lambda := \ell_u(\mu_0,\mu_1) \in (0,\infty)$. Given that $\int \varphi \, d\mu + \int \varphi^{(u_\lambda \circ \ell)} \, d\nu = u(1)$, and since $(\mu_0,\mu_1) \in \mathcal{P}_c(M)^2$, any optimal coupling $\pi$ for $(\mu_0,\mu_1)$ is concentrated on $S:=\{(x,y) \in X_0 \times X_1 : \varphi(x) + \varphi^{(u_\lambda \circ \ell)}(y) = u_\lambda(\ell(x,y))\}$ (here $X_i:=\mathrm{supp}(\mu_i)$), the latter is compactly contained in $\{\ell > 0\}$ and is $u_\lambda \circ \ell$-cyclically monotone. Combining this with the noncrossing of trajectories established in Proposition \ref{prop: lagrangian trajectories dont cross} implies that setting $W(m):=(x,y)$ whenever $(x,y) \in S$ and $m \in Z_s(x,y)$ yields a continuous map which is right-inverse to $\bar z_s$ on $\bar z_s(S)$.

    Now, given $\mu_s$ on a $u$-geodesic from $\mu_0$ to $\mu_1$, since $\ell_u(\mu_0,\mu_1) > 0$, also $\ell_u(\mu_0,\mu_s),\ell_u(\mu_s,\mu_1) > 0$, and we are in the setting of Proposition \ref{prop: propoerties of the reverse ineq}(ii), which gives $\omega \in \mathcal{P}(M^3)$ whose projection to any pair of coordinates is optimal, as well as $z \in Z_s(x,y)$ for $\omega$-a.e.\ $(x,z,y) \in M^3$. Since $\pi:=(P_{13})_{\#}\omega$ is optimal for $(\mu_0,\mu_1)$, it is supported on the compact set $S$, hence $\mu_s$ must vanish outside of $Z_s(S)$. This is sufficient to conclude that $\pi = W_{\#} \mu_s$. 

    The Lipschitz continuity of $W$ is a consequence of McCann's adaptation of the Monge--Mather shortening estimate and goes through for general $u$ exactly as for $u = u_p$, $p \in (0,1)$. For details on the latter see \cite[Thm.\ A.1]{McCann:2020}, and note that our assumption $u \in C^2$ gives enough regularity for all of the computations in that reference.
\end{proof} 

\begin{Proposition}[Star-shapedness of $u$-separation]
\label{Prop: star shaped u separation}
Let $u$ be admissible, and $(\mu_s)_{s\in[0,1]}$ a $u$-geodesic with $u$-separated endpoints. Then also $(\mu_s,\mu_t)$ is $u$-separated for every $0 \leq s < t \leq 1$.
\end{Proposition}
\begin{proof}
    It suffices to prove $u$-separation for $(\mu_0,\mu_s)$. By Proposition \ref{prop: interpolants inherit compactness} (cf.\ also Remark \ref{Remark: Right continuity}), $\mu_s$ has compact support. Set $X_s:=\mathrm{supp}(\mu_s)$ and $\lambda:=\ell_u(\mu_0,\mu_1)$. By duality for $u$-separated measures (cf.\ Theorem \ref{Theorem: Duality by u-separation}), we find a pair $(\varphi,\varphi^{(u_\lambda \circ \ell)})$ which are semiconvex Lipschitz functions on neighborhoods of $X_0$ and $X_1$, respectively. Moreover, the set $S:=\{(x,y) \in X_0 \times X_1 : \varphi(x) + \varphi^{(u_\lambda \circ \ell)}(y) = u_\lambda(\ell(x,y))\}$ is compact, contained in $\{\ell > 0\}$, and points in $S$ realise the suprema in the conjugacy of $\varphi, \psi:=\varphi^{(u_\lambda \circ \ell)}$ in the sense that
    \begin{equation*}
        \varphi(m) = \max_{(x,y) \in S} u_\lambda (m,y) - \psi(y).
    \end{equation*}
    The maximum above is attained at some $(x,y) \in S$ with $x = m$.
    Moreover, the first and second projections of $S$ cover all of $X_0$ and $X_1$. Let $W:Z_s(S) \to S$ be the (Lipschitz) continuous inverse map from Proposition \ref{Prop: Lipschitz continuous inverse map} satisfying $W(z) = (x,y) =:(U_s(z),V_s(z))$, $z \in Z_s(x,y)$. By (the proof of) Lemma \ref{lemma: starshape of concave},
    \begin{equation*}
        s^{-1}\, \varphi(m) = \max_{z \in Z_s(S)}  u_{s\lambda}(\ell(x,z)) + s^{-1}(1-s) \,u_\lambda(\ell(U_s(z),V_s(z))) - s^{-1}\,\psi(V_s(z)).
    \end{equation*}
    The maximum is realised at some $z \in Z_s(S)$ such that $U_s(z) = m$. Pick now $\tilde \varphi:=s^{-1} \,\varphi$, $\tilde \lambda:=s\lambda$ and $\tilde \pi:=(U_s \times \mathrm{id})_{\#} \mu_s$. Then the above calculation shows that $\tilde \psi :=\tilde \varphi^{(u_{\tilde \lambda} \circ \ell)} = s^{-1}\,\psi \circ V_s - s^{-1}(1-s)\,u_\lambda \circ \ell \circ W$ on $\mathrm{supp}(\mu_s)$. It is now elementary to check that these data verify the $u$-separation for $(\mu_0,\mu_s)$, where now the compact set in question is $\tilde S:=\{(x,z) \in \mathrm{supp}(\mu_0 \times \mu_s) : \tilde \varphi(x) + \tilde \psi(z) = u_{\tilde \lambda}(\ell(x,z))\}$, see the analogous arguments in the proof of \cite[Prop.\ 5.5]{McCann:2020}. The only point which needs to be additionally verified when compared to the proof for $u = u_p$ is the integrability condition $\int \tilde \varphi \,d\mu_0 + \int \tilde \psi \, d\mu_s = u(1)$, but this follows from the optimality of $W_{\#} \mu_s$ for $(\mu_0,\mu_1)$ stated in Proposition \ref{Prop: Lipschitz continuous inverse map}, so that $\int u_\lambda \circ \ell \, d(W_{\#}\mu_s) = u(1)$, cf.\ Proposition \ref{prop: existence of optimal coupling}. With this in mind, it follows that
    \begin{align*}
        \int \tilde \varphi \, d\mu_0 + \int \tilde \psi \, d\mu_s &= s^{-1} \left(\int \varphi \, d\mu_0 + \int [\psi \circ P_2 \circ W - (1-s)\, u_\lambda \circ \ell \circ W]\, d\mu_s \right) \\
        &= s^{-1} \left(\int \varphi \, d\mu_0 + \int \psi \, d\mu_1 - (1-s) \int u_\lambda \circ \ell \, d(W_{\#}\mu_s)\right) \\
        &= s^{-1}\,(u(1) - (1-s) \,u(1)) = u(1).
    \end{align*}
    This concludes the proof.
\end{proof}

We remark that we get a slightly different representation of the intermediate potentials since our Lemma \ref{Lemma: Concavity along geodesics} differs from McCann's version \cite[Prop.\ 5.1]{McCann:2020} in the representation of the last term. Also, as remarked in \cite[Rem.\ 5.6]{McCann:2020}, the intermediate potentials are related to the Hopf--Lax semigroup of the Hamilton--Jacobi equation, where the Hamiltonian is $H(\cdot,\cdot,u_{s\lambda}^*)$ for suitable $s$.

The next lemma adapts {\cite[Lem.\ 5.7]{McCann:2020}} and {\cite[Lem.\ A.11]{kell2017interpolation}}.

\begin{Lemma}[Maps and their Jacobian derivatives]\label{lem: maps and their jacobian der}
    Let $u:(0,\infty)\to\R$ be an admissible function, and fix $X,Y\subset M$ compact. Let $\varphi:X\to\R$ be a semiconvex and Lipschitz function such that $\varphi\geq \left(\varphi^{(u\circ\ell)}\right)^{(u\circ\ell)}$ holds in a neighbourhood of $X$.
    \begin{enumerate}
        \item If $\varphi\oplus \varphi^{(u\circ\ell)}-u\circ\ell\geq0$ vanishes at $(\bar{x},\bar{y})\in X\times Y$, i.e.\ if $(\bar{x},\bar{y})\in\partial_{(u\circ\ell)}\varphi$, then $\bar{x}\in \operatorname{Dom}D\varphi$ implies $\bar{y}=F_1(\bar{x})$, where $F_s(x):=\exp_xs\,v(D\varphi(x),x;u)$, with
        \begin{equation}\label{eq: transport vector}
            v(D\varphi(x),x;u)=-\frac{(u')^{-1}(\vert\nabla\varphi\vert(x))}{\vert\nabla\varphi\vert(x)}\,  \nabla\varphi(x) = DH(D\varphi(x),x;u^*).
        \end{equation}  
        Moreover, $\bar{x}\in\operatorname{Dom}\tilde{D}^2\varphi$ implies $(\bar{x},\bar{y})\notin\operatorname{sing}(\ell)$. Analogously, $\bar y \in \mathrm{Dom} \, D\varphi^{(u \circ \ell)}$ gives $\bar x = \exp_{\bar y} -DH(-D\varphi^{(u \circ \ell)}(\bar{y}),\bar{y}; u^*)$.

        \item For $\operatorname{vol}_g$-a.e. $x\in X$, the approximate derivative $\tilde{D}F_s(x):T_xM\to T_{F_s(M)}M$ exists, depends smoothly on $s$, and $\tilde{D}F_s(x)\,w$ gives a Jacobi field along the geodesic $s\in[0,1]\mapsto F_s(x)$ for each $w\in T_xM$.

        \item Furthermore,
            \begin{equation}
                \frac{\partial}{\partial s}\bigg|_{s=0}\tilde{D}F_s=\tilde{D}\frac{\partial F_s}{\partial s}\bigg|_{s=0}=(Dv\circ D\varphi)\,\tilde{D}^2\varphi = (D^2 H \circ D \varphi) \,\tilde D^2 \varphi
            \end{equation}
            holds true $\operatorname{vol}_g$-a.e. on $X$. Here, the derivatives are computed with respect to the Lorentzian connection.
    \end{enumerate} 
\end{Lemma}
\begin{proof}
    \begin{enumerate}
        \item  Observe that, by definition of $\left(\varphi^{(u\circ\ell)}\right)^{(u\circ\ell)}$, the inequality $\left(\varphi^{(u\circ\ell)}\right)^{(u\circ\ell)}\oplus \varphi^{(u\circ\ell)}-u\circ\ell\geq0$ holds on $M\times Y$. Thus $\varphi\oplus \varphi^{(u\circ\ell)}-u\circ\ell\geq0$ on $U\times Y$, where $U$ is the hypothesised neighbourhood of $X$ on which $\varphi$ is Lipschitz and semiconvex.

        If the latter inequality is saturated at $(\bar{x},\bar{y})\in X\times Y$, and $\bar{x}\in\operatorname{Dom}D\varphi$, it follows that $u\circ\ell(\cdot,\bar{y})$ is superdifferentiable at $\bar{x}$ with supergradient $D\varphi(\bar{x})$. Indeed, for $v\in T_{\bar{x}}M$ in a small neighbourhood around 0:
        \begin{equation}\label{eq:maps and jacobian derivatives 1}
            \begin{aligned}
                u\left(\ell(\exp_{\bar{x}}v,\bar{y})\right)&\leq \varphi\oplus \varphi^{(u\circ\ell)}(\exp_{\bar{x}}v,\bar{y})\\&=\varphi\oplus \varphi^{(u\circ\ell)}(\bar{x},\bar{y})+D\varphi(\bar{x})v+O(|v|_{\tilde{g}})\\&=u\left(\ell(\bar{x},\bar{y})\right)+D\varphi(\bar{x})v+O(|v|_{\tilde{g}}),
            \end{aligned}
        \end{equation}
        since $\left(\varphi\oplus \varphi^{(u\circ\ell)}-u\circ\ell\right)(\bar{x},\bar{y})=0$. Therefore, use now Corollary \ref{cor: twist and non degeneracy}(ii) to conclude the first part of the statement.

        Moreover, if $\bar{x}\in\operatorname{Dom}\tilde{D}^2\varphi$, the Taylor expansion of $\varphi$ around $\bar{x}$ gives a quadratic bound for $u\circ\ell(x,\bar{y})-\varphi^{(u\circ\ell)}(\bar{y})$ at $\bar{x}$. According to Corollary \ref{cor: twist and non degeneracy}(iii), this implies $(\bar{x},\bar{y})\notin\operatorname{sing}(\ell)$.

        On the other hand, since $(\bar{x},\bar{y})\notin\operatorname{sing}(\ell)$, in particular, $\bar{y}$ is not in the timelike cut locus of $\bar{x}$. Therefore, we can apply the first variation formula to the following:
        \begin{align*}
            u\left(\ell(\exp_{\bar{x}}v,\bar{y})\right)-u\left(\ell(\bar{x},\bar{y})\right)&=-\frac{u'\left(\ell(\bar{x},\bar{y})\right)}{\ell(\bar{x},\bar{y})}\, g_{\bar{x}}(v,v(D\varphi(x),x;u))+O(|v|_{\tilde{g}}).
        \end{align*}
        Combining this with \eqref{eq:maps and jacobian derivatives 1},
        \begin{equation*}
            -\frac{u'\left(\ell(\bar{x},\bar{y})\right)}{\ell(\bar{x},\bar{y})}\, g_{\bar{x}}(v,v(D\varphi(\bar{x}),\bar{x};u))\leq D\varphi(\bar{x})v,
        \end{equation*}
        and this readily implies that $\nabla\varphi=-\frac{u'\left(\ell(\bar{x},\bar{y})\right)}{\ell(\bar{x},\bar{y})}\,v(D\varphi(x),x;u)$. In particular,
        \begin{equation*}
          v(D\varphi(x),x;u)=-\frac{\ell(\bar{x},\bar{y})}{u'\left(\ell(\bar{x},\bar{y})\right)}\,  \nabla\varphi.
        \end{equation*}
        Finally, noting that $| v(D\varphi(x),x;u)|_g=\ell(\bar{x},\bar{y})$, the result follows.

         \item Fix $\varepsilon>0$. Semiconvexity of $\varphi$ implies that outside of a set of volume $\varepsilon$ in $U\supset X$, $D\varphi$ agrees with a continuously differentiable covector field $V_{\varepsilon}$ on $M$. Moreover, its approximate second derivative agrees with $DV_{\varepsilon}$ outside of this small set.

         Consider, for each $x\in X$, the vector $v(V_{\varepsilon}(x),x;u) = DH(V_\varepsilon(x),x; u^*)$. If we define the maps $F_s^{\varepsilon}(x):=\exp_xs\,v(V_{\varepsilon}(x),x;u)$, these are $C^1$ in $x\in M$ and smooth in $s\in[0,1]$, and its mixed partial derivatives are continuous and equal: $\frac{\partial}{\partial s}DF_s^{\varepsilon}(x)=D\frac{\partial}{\partial s}F_s^{\varepsilon}(x)$, where $D$ is the derivative with respect to $x$.

            The idea is now to check that this approximate derivative will define Jacobi fields whenever it acts on a vector of the tangent bundle. The way of doing it will be by defining a geodesic variation, and applying {\cite[Lem.\ 8.3]{O'Neill:1983}}. Fix $(x,w)\in TM$, and let $r\in[-1,1]\mapsto x(r)$ be a $C^1$ curve such that $x(0)=x$ and $\dot{x}(0)=w$. Then $r\in[-1,1]\mapsto F_s^{\varepsilon}(x(r))$ is a $C^1$-geodesic variation, since $s\in[0,1]\mapsto F_s^{\varepsilon}(x(r))$ is a geodesic segment for each $r\in[-1,1]$. Thus $\frac{\partial}{\partial r}\Big|_{r=0}F_s^{\varepsilon}(x(r))=DF_s^{\varepsilon}(x(0))\,w$ is a Jacobi field.

            Since the approximate derivative $\tilde{D}F_s(x)$ agrees with $DF_s^{\varepsilon}(x)$ outside of a set of volume $\varepsilon$, and this was arbitrary, we have that $\tilde{D}F_s(x(0))\,w$ depends smoothly on $s$ and is a Jacobi field for $x(0)\in U$ in a subset of full volume.

            \item Fix $\varepsilon>0$ and consider, as in the last point, the vector field $V_{\varepsilon}$ and the associated map $F_s^{\varepsilon}$. Differentiating the vector field $\frac{\partial F_s^{\varepsilon}(x)}{\partial s}\Big|_{s=0}=v(V_{\varepsilon}(x),x;u)$ using the Lorentzian connection gives
            \begin{equation*}
                D_k\frac{\partial}{\partial s}\bigg|_{s=0}F_s^{\varepsilon}(x)^i=(Dv\circ V_{\varepsilon})\,DV_{\varepsilon}.
            \end{equation*}

            We may again interchange the order of the $x$ and $s$ derivatives. Since these derivatives of $F^{\varepsilon}$ and $\varphi^{\varepsilon}$ coincide with those of $F$ and $\varphi$ outside of a set of volume $\varepsilon$, the result follows.
    \end{enumerate}
\end{proof}

We will often want to use the previous result for $u = u_\lambda$, where $u$ is admissible and $\lambda > 0$.

\begin{Theorem}[Characterising optimal maps]\label{thm: characterising opt maps}
            Let $u$ be admissible and let $(\mu,\nu)\in\mathcal{P}_c(M)^2$ be $u$-separated with $\lambda:=\ell_u(\mu,\nu) \in (0,\infty)$, a Kantorovich potential $\varphi$ and an optimal coupling $\pi$ as in the definition of $u$-separation. Set $X:=\operatorname{supp}\mu$, $Y:=\operatorname{supp}\nu$, and assume that $\mu\in\mathcal{P}^{ac}_c(M)$. Then
            \begin{enumerate}
                \item there is a unique map $F(x):=\exp_x DH(D\bar{\varphi}(x),x;u_\lambda^*)$, where $\lambda := \ell_u(\mu,\nu) \in (0,\infty)$, with $\nu=F_{\#}\mu$, and $\bar{\varphi}$ is semiconvex and Lipschitz and satisfies
                \begin{equation}
                    \bar{\varphi}(x)=\max_{y\in Y} u_{\lambda}\left(\ell(x,y)\right)-\bar{\varphi}^{(u_{\lambda}\circ\ell)}(y)
                    \label{eq: characterisation of optimal maps}
                \end{equation}
            in a neighborhood of $X$; in this case, $\pi=(\operatorname{id}\times F)_{\#}\mu$ is the unique optimal coupling between $\mu$ and $\nu$, $\varphi$ is semiconvex and Lipschitz in a neighborhood of $X$, and both $D\varphi=D\bar{\varphi}$ and $(x,F(x))\notin\operatorname{sing}(\ell)$ hold $\mu$-a.e.; 
                \item if, in addition, $\nu\in\mathcal{P}^{ac}_c(M)$, then $F\circ G(y)=y$ holds $\nu$-a.e. and $G(F(x))=X$ holds $\mu$-a.e., where $G(x):=\exp_y\left(-DH\left(-D\varphi^{(u_\lambda\circ\ell)}(y),y;u_\lambda^*\right)\right)$.
            \end{enumerate}
        \end{Theorem}
        \begin{proof}
            \begin{enumerate}
                \item By Theorem \ref{Theorem: Duality by u-separation}, the semicontinuous functions $(\varphi,\psi)$ given in the definition of $u$-separation satisfy  $(\varphi,\psi) = (\psi^{(u_\lambda \circ \ell)}, \varphi^{(u_\lambda \circ \ell)})$ on $X \times Y$, and moreover their (nonrenamed) extensions $\varphi:= \psi^{(u_\lambda \circ \ell)}$ and $\psi:=\varphi^{(u_\lambda \circ \ell)}$ are semiconvex and Lipschitz on neighborhoods of $X$ and $Y$. Let $S:=\{(x,y)\in X\times Y:\ \varphi(x)+\psi(y)=u_{\lambda}(\ell(x,y))\}$, which by $u$-separation is compactly contained in $\{\ell > 0\}$ and $\pi$ is concentrated on $S$.

                Since $\varphi$ is a semiconvex Lipschitz function, by Alexandrov's Theorem, $\varphi$ has (approximate) second derivative $\operatorname{vol}_g$-a.e. in $M$, and hence $\mu$-a.e. (recall $\mu\ll\operatorname{vol}_g$ by assumption). In particular, for $(x,y)\in S$, $x\in\operatorname{Dom}\tilde{D}^2\varphi$ $\operatorname{vol}_g$-a.e. $x$. For these pairs of points, Lemma \ref{lem: maps and their jacobian der} shows that $y=F_1(x)$, where $F_s(x):=\exp_x s \,DH(D\varphi(x),x;(u_\lambda)^*)$, and $(x,F_1(x))\notin\operatorname{sing}(\ell)$. Therefore, we obtain that $S$ is $\mu$-a.e. the graph of a map, and {\cite[Lem.\ 3.1]{ahmad2011optimal}} gives $\pi=(\operatorname{id}\times F_1)_{\#}\mu$, where $\pi$ is the optimal coupling whose existence is guaranteed by $u$-separation.

                If $\pi'\in\Pi_{\leq}(\mu,\nu)$ is another optimal coupling, then, since $\int u_\lambda \circ \ell \, d\pi' = u(1)$, it is easily seen that $\pi'$ must vanish outside $S$. In the same way as before, one concludes $\pi'=(\operatorname{id}\times F_1)_{\#}\mu$, which proves the uniqueness of the optimal coupling.

                What is left to check is that $F_1$ given by Lemma \ref{lem: maps and their jacobian der} agrees $\mu$-a.e. with the map $F$ in the statement of the Theorem. This will be done by using the uniqueness of the optimal coupling: if we prove that $\bar{\pi}:=(\operatorname{id}\times F)_{\#}\mu$ is also a maximiser, then we are done.
                
                Consider a Lipschitz map $\bar{\varphi}$ satisfying \eqref{eq: characterisation of optimal maps}. Alexandrov's Theorem gives again that $\bar{\varphi}$ has second derivative $\operatorname{vol}_g$-a.e. (and hence $\mu$-a.e.), so, in particular, it has first derivative $\mu$-a.e.\ Consider $x\in X\cap\operatorname{Dom}\tilde{D}\bar{\varphi}$, and the corresponding point $y\in Y$ which maximises \eqref{eq: characterisation of optimal maps}. Lemma \ref{lem: maps and their jacobian der} asserts that $y=F(x)$. Thus, 
                \begin{equation*}
                    u_{\lambda}(\ell(x,F(x)))=\bar{\varphi}(x)+\bar{\varphi}^{(u_{\lambda}\circ\ell)}(F(x))
                \end{equation*}
                holds $\mu$-a.e.\ Integrating this against $\mu$ yields
                \begin{equation*}\label{eq: duality optimal coupling}
                    \int_{M^2}u_{\lambda}(\ell(x,y))\, d\bar{\pi}(x,y)=\int_M \bar{\varphi}(x)\, d\mu +\int_M \bar{\varphi}^{(u_{\lambda}\circ\ell)}(y)\, d\nu,
                \end{equation*}
                where we have used that $F_{\#}\mu=\nu$. 
                
                Combining the right hand side to an integral with respect to the optimal $\pi$, we see that $\int u_\lambda \circ \ell \, d\bar\pi = u(1)$, thus also $\bar \pi$ is optimal.

                \item Given $\nu\in\mathcal{P}^{ac}_c(M)$, a symmetric argument as in (i) shows that there exists a unique map $G(x):=\exp_y\left(-DH(-D\varphi^{(u_\lambda\circ\ell)}(y).y;(u_\lambda)^*\right)$ with $\mu=G_{\#}\nu$. Therefore, it also holds that $\pi=(G\times\operatorname{id})_{\#}\nu$. In particular, using this result together with (i), the set $(X\cap\operatorname{Dom}D\varphi)\times(Y\cap\operatorname{Dom}D\varphi^{(u\circ\ell)})$ has full $\pi$-measure, and for each point $(x,y)$ in this set, it holds that $y=F(x)$ and $x=G(y)$. Therefore, $G$ acts $\mu$-a.e. as an inverse of $F$, who in turn acts $\nu$-a.e. as an inverse to $G$.
            \end{enumerate}
        \end{proof}

\begin{Remark}[On the transport maps for $u=u_p$]
    Note that Theorem \ref{thm: characterising opt maps} works with the Kantorovich potentials. Hence, denoting, for $\mu,\nu\in\mathcal{P}_c(M)$, $\lambda:=\ell_u(\mu,\nu)$, these are $(u_\lambda\circ\ell)$-convex functions. Denote the Kantorovich potential by $\varphi_\lambda$. As a result of Theorem \ref{Theorem: Duality by u-separation},
    \begin{equation*}
        \varphi_{\lambda}(x)+\varphi_{\lambda}^{(u_{\lambda}\circ\ell)}(y)=u_{\lambda}(\ell(x,y))\quad \pi\text{-a.e. }(x,y),
    \end{equation*}
    where $\pi$ is the optimal coupling. Now, taking $u=u_p$, $p \in (0,1)$ (also $p < 1$ works), this reads as
    \begin{equation*}
        \varphi_{\lambda}(x)+\varphi_{\lambda}^{(u_{\lambda}\circ\ell)}(y)=\frac{\ell^p(x,y)}{\lambda^p\, p}\quad \pi\text{-a.e. }(x,y).
    \end{equation*}
    Taking $\varphi:=\lambda^p\,\varphi_{\lambda}$, we see that $\varphi$ is a $u\circ\ell$-convex function, and whenever $(x,y)\in\partial_{(u_{\lambda}\circ\ell)}\varphi_{\lambda}$, also $(x,y)\in\partial_{(u\circ\ell)}\varphi$. Now, using Lemma \ref{lem: maps and their jacobian der},
    \begin{equation*}
        y=\exp_x\left(-\frac{(u')^{-1}(\vert\nabla\varphi\vert(x))}{\vert\nabla\varphi\vert(x)}\,  \nabla\varphi(x)\right),\quad \text{and}\quad y=\exp_x\left(-\frac{(u_{\lambda}')^{-1}(\vert\nabla\varphi_{\lambda}\vert(x))}{\vert\nabla\varphi_{\lambda}\vert(x)}\,  \nabla\varphi_{\lambda}(x)\right).
    \end{equation*}
    Note that these two points are indeed the same. Using that the exponential map is a bijection, we just need to check that both vectors are the same:
    \begin{align*}
        \frac{(u_{\lambda}')^{-1}(\vert\nabla\varphi_{\lambda}\vert(x))}{\vert\nabla\varphi_{\lambda}\vert(x)}\,  \nabla\varphi_{\lambda}(x) &= 
        \frac{\lambda^{-p/(1-p)} \vert\nabla\varphi_{\lambda}\vert^{-1/(1-p)}(x)}{\vert\nabla\varphi_{\lambda}\vert(x)}\,  \nabla\varphi_{\lambda}(x) \\
        &=\lambda^{q} \vert\nabla\varphi_{\lambda}\vert^{q-2}(x)\,  \nabla\varphi_{\lambda}(x) \\
        &=\lambda^{q} \lambda^{-p\,(q-2)}\vert\nabla\varphi\vert^{q-2}(x)\, \lambda^{-p}\, \nabla\varphi(x)\\
        &=\vert\nabla\varphi\vert^{q-2}(x)\, \nabla\varphi(x),
    \end{align*}
    where $p^{-1}+q^{-1}=1$. This is the transport described in \cite{McCann:2020}.
\end{Remark}

\begin{Corollary}[Uniqueness of Kantorovich potentials]\label{cor: uniqueness of kantorovich potentials}
    Let $(\mu,\nu)\in\mathcal{P}_c(M)^2$ be $u$-separated. Set $X:=\operatorname{supp}\mu$ and $Y:=\operatorname{supp}\nu$, and assume $\mu\in\mathcal{P}^{ac}_c(M)$. Then there exists a unique (up to an additive constant) Kantorovich potential $\varphi:X\to\R$ for $(\mu,\nu)$ which saturates the inequality \eqref{eq: Kantorovich potential inequality}.
\end{Corollary}
\begin{proof}
    We already obtained existence in Theorem \ref{Theorem: Duality by u-separation}, while uniqueness follows from Theorem \ref{thm: characterising opt maps}.
\end{proof}

\begin{Corollary}[Lagrangian characterisation of $u$-geodesics]\label{cor: lag charact og geodesics}
            Let $(\mu_0,\mu_1)\in\mathcal{P}_c(M)^2$ be $u$-separated for some admissible function $u:(0,\infty)\to\R$, and let $\lambda:=\ell_u(\mu,\nu)\in(0,\infty)$. Let $\varphi:\operatorname{supp}\mu_0\to\R$ denote the unique (up to an additive constant) Kantorovich potential. If $\mu_0\in\mathcal{P}_c^{ac}(M)$, then $F_s(x):=\exp_x s\,DH(D\varphi(x),x;(u_\lambda)^*)$ defines the unique $u$-geodesic $s\in[0,1]\mapsto\mu_s:=F_{s\#}\mu_0$ in $\mathcal{P}(M)$ from $\mu_0$ to $\mu_1$. Moreover, $\mu_s\in\mathcal{P}^{ac}_c(M)$ for all $s\in[0,1)$. 
        \end{Corollary}
        \begin{proof}
            Under these hypotheses, Theorem \ref{thm: characterising opt maps} gives that the unique optimal coupling between $\mu_0$ and $\mu_1$ is given by $\pi:=(\operatorname{id}\times F_1)_{\#}\mu_0$, where $F_1$ is defined as in Lemma \ref{lem: maps and their jacobian der}. Furthermore, it asserts that $(x,F_1(x))\notin\operatorname{sing}(\ell)$ $\mu_0$-a.e.\ In particular, this last property implies $\pi(\operatorname{sing}(\ell))=0$.  Next, Theorem \ref{thm: existence of u geod} gives that there is a unique $u$-geodesic from $\mu_0$ to $\mu_1$ given by $\mu_s:=(\bar{z}_s)_{\#}\pi$ for $s\in[0,1]$. Now, $F_s:M\to M$ satisfies the following property: $z_s(x,F_1(x))=F_s(x)$, which immediately gives that $\mu_s=(\bar{z}_s)_{\#}\pi=(F_s)_{\#}\mu_0$, as claimed. 
            
            To prove the second part of the result, fix $s\in(0,1)$. Consider $V\subset F_s(\operatorname{supp}\mu_0)=\operatorname{supp}\mu_s$. We want to obtain $\mu_s(V)$. Proposition \ref{Prop: Lipschitz continuous inverse map} asserts that $F_s$ has a Lipschitz inverse map, given precisely by the first component of $W$ (where $W$ is defined in the referenced result). If $V$ is such that $\operatorname{vol}_g(V)=0$, then $F_s^{-1}(V)$ has also zero $\operatorname{vol}_g$-volume. By absolute continuity of $\mu_0$, $\mu_s(V)=\mu_0(F_s^{-1}(V))=0$, which establishes absolute continuity of $\mu_s$.
        \end{proof}

\begin{Corollary}[Star-shapedness of Kantorovich potentials]
   Let $\mu_0,\mu_1\in\mathcal{P}_c(M)$, and let $(\mu_s)_{s\in[0,1]}$ be a $u$-geodesic from $\mu_0$ to $\mu_1$ with $u$-separated endpoints. Let $\varphi:\operatorname{supp}\mu_0\to\R$ be the Kantorovich potential for $(\mu_0,\mu_1)$ as given by the definition of $u$-separation. Then $\varphi_t:=t^{-1}\varphi$ is the unique Kantorovich potential (up to an additive constant) for $(\mu_0,\mu_t)$, which are also $u$-separated with $\ell_u(\mu_0,\mu_t) = t\, \ell_u(\mu_0,\mu_1)$.
\end{Corollary}
\begin{proof}
    This was already established in Proposition \ref{Prop: star shaped u separation}.
\end{proof}

\begin{Corollary}[Monge-Ampère type equation]\label{cor: monge ampere}
            Under the hypotheses of Theorem \ref{thm: characterising opt maps}, $F$ is countably Lipschitz and the Jacobian equation
            \begin{equation*}
                \rho_0(x)=\rho_1(F(x))\,JF(x)
            \end{equation*}
            holds $\mu$-a.e., where $\rho_0=d\mu/d\mathrm{vol}_g$, $\rho_1=d\nu/d\mathrm{vol}_g$ and $JF(x)=|\det\tilde{D}F(x)|$, with $\tilde{D}F(x)$ denoting the approximate derivative of $F$.
            
            Moreover, if $(\mu_t)_{t\in[0,1]}$ is the unique $u$-geodesic from $\mu$ to $\nu$, the more general Monge-Ampère type equation holds for all $s\in[0,1]$:
            \begin{equation}\label{eq: monge ampere}
                \rho_0(x)=\rho_s(F_s(x))\,JF_s(x)\qquad \mu\text{-a.e.}
            \end{equation}
        \end{Corollary}
        \begin{proof}
            The potential $\bar{\varphi}$ from Theorem \ref{thm: characterising opt maps} is semiconvex by Theorem \ref{Theorem: Duality by u-separation}. As a consequence, $\bar{\varphi}$ agrees with a $C^2$ function outside of a set of arbitrarily small volume. Thus $F$ is countably Lipschitz, hence approximately differentiable $\operatorname{vol}_g$-a.e.\ It is also injective $\mu$-a.e., according to Theorem \ref{thm: characterising opt maps}(ii). The Jacobian equation follows now from {\cite[Thm.\ 5.10]{McCann:2020}}.

            The second statement is a consequence of the star-shapedness along $u$-geodesics, and the above paragraph.
        \end{proof}

\section{Ricci curvature bounds and entropic convexity}
\label{Section: Ricci and entropic convexity}

Having established the properties of the transport arising from the study of $\ell_u$, we now turn to the problem of relating the convexity of the relative entropy along $u$-geodesics to timelike Ricci curvature lower bounds in the current section.

\subsection{Entropic convexity from Ricci lower bounds for $u$-separated measures}\label{subsection: entropic convexity from Ricci}

\begin{Proposition}[Jacobian along $u$-geodesics]\label{prop: jacobian along geodesics}
        Let $u:(0,\infty)\to\R$ be an admissible function. Let $(\mu_0,\mu_1)\in\mathcal{P}_c^{ac}(M)^2$ be $u$-separated, let $\lambda:=\ell_u(\mu_0,\mu_1)\in(0,\infty)$ and $\varphi$ be the Kantorovich potential. Set $X\times Y=\operatorname{supp}(\mu_0\times\mu_1)$, and let $F_s(x):=\exp_xs\,DH(D\varphi(x),x;u^*_\lambda)$ be as obtained in Theorem \ref{thm: characterising opt maps}.
        
        Then, for $\operatorname{vol}_g$-a.e. $x\in X$, the approximate derivative $A_s(x):=\tilde{D}F_s(x):T_xM\to T_{F_s(x)}M$ exists. Moreover, it is invertible, depends smoothly  on $s\in[0,1]$, and $\phi(s):=-\log|\det A_s(x)|$ satisfies
        \begin{equation}
            \phi '(s)=-\operatorname{Tr}B_s(x),
            \label{eq:phi'}
        \end{equation}
        \begin{equation}
            \phi''(s)=\operatorname{Ric}_{F_s(x)}\left(F_s'(x),F_s'(x)\right)+\operatorname{Tr}\left(B_s^2(x)\right),
            \label{eq:phi''}
        \end{equation}
        \begin{equation}\label{eq: trace inequality}
            \operatorname{Tr}\left(B_s^2(x)\right)\geq\frac{1}{n}\left(\operatorname{Tr}B_s(x)\right)^2,
        \end{equation}
        where $B_s(x):=A_s'(x)\,A_s(x)^{-1}$ and $'$ denotes $\frac{\partial}{\partial s}$.
    \end{Proposition}
    \begin{proof}
        The first claim follows as a consequence of Lemma \ref{lem: maps and their jacobian der}(ii), which asserts that, in the same hypotheses as now, the approximate derivative $\tilde{D}F_s(x)=A_s(x)$ exists for $\operatorname{vol}_g$-a.e. $x\in X$. Furthermore, this result also gives that $s\in[0,1]\mapsto A_s(x)\,w$ is a smooth Jacobi field for each $w\in T_xM$, which will be useful to prove one of the identities for the derivatives in what follows.

        The invertibility of $A_s$ may be established as in the Riemannian case of \cite{CorderoMcCannSchmuckenschlager01}, using the (strict) convexity of $H$. We omit the details.

        Set $\phi(s):=-\log|\det A_s(x)|$. A standard matrix identity gives the first claimed equality \eqref{eq:phi'}:
        \begin{equation*}
            \phi '(s)=(-\log|\det A_s(x)|)'=-\operatorname{Tr}(A_s(x)'A_s(x)^{-1})=-\operatorname{Tr}B_s(x).
        \end{equation*}

        For the second equality \eqref{eq:phi''}, start by differentiating the previous one:
        \begin{equation*}
        \begin{aligned}
            \phi''(s)&=(-\operatorname{Tr}B_s(x))'=-\operatorname{Tr}(B_s(x)')\\
            & =-\operatorname{Tr}\left(A_s(x)''A_s(x)^{-1}\right)-\operatorname{Tr}\left(A_s(x)'(-A_s(x)^{-1}A_s(x)'A_s(x)^{-1})\right)\\
            &=-\operatorname{Tr}\left(A_s(x)''A_s(x)^{-1}\right)+\operatorname{Tr}(B_s(x)^2).
        \end{aligned}
        \end{equation*}

        Now, to rewrite the first term, we use the property from Lemma \ref{lem: maps and their jacobian der}(ii) that establishes that $s\in[0,1]\mapsto A_s(x)\,w\in T_{F_s(x)}M$ is a smooth Jacobi field for each $w\in T_xM$, and hence it satisfies the Jacobi equation, i.e.\ in coordinates,
        \begin{equation*}
            0=\nabla_{F'}(\nabla_{F'}A^i_{\bar{j}})+R_{jkl}\ ^i F'^j F'^l A^k_{\bar{j}}.
        \end{equation*}
        where $\bar{j}$ denotes a fixed index, over which there is no summation. Multiplying this identity by $(A^{-1})^{\bar{j}}_i$ gives the desired result to arrive at \eqref{eq:phi''}:
        \begin{equation*}
            0=\left(\nabla_{F'}(\nabla_{F'}A^i_{\bar{j}})+R_{jkl}\ ^i F'^j F'^l A^k_{\bar{j}}\right)(A^{-1})^{\bar{j}}_i=\operatorname{Tr}\left(A_s(x)''A_s(x)^{-1}\right)+\operatorname{Ric}_{F_s(x)}(F_s'(x),F_s'(x)).
        \end{equation*}

        To obtain \eqref{eq: trace inequality}, we will use the standard inequality for a symmetric $n\times n$ matrix $C$: $\operatorname{Tr}C^2\geq\frac{1}{n}(\operatorname{Tr}C)^2$. Therefore, we only need to show that $B_s(x)$ is symmetric for all $x\in X$ and $s\in[0,1]$. Denoting by $B_s^*(x)$ the adjoint, we have that
        \begin{equation*}
            B_s^*(x)-B_s(x)=(A_s^*(x))^{-1}\left(A_s^*(x)'\,A_s(x)-A_s^*(x)\,A_s(x)'\right)\,A_s(x)^{-1},
        \end{equation*}
        and
        \begin{equation}\label{eq: symmetric B}
            \left(A_s^*(x)'\,A_s(x)-A_s^*(x)\,A_s(x)'\right)'=A_s^*(x)''\,A_s(x)-A_s^*(x)\,A_s(x)''.
        \end{equation}

        On the other hand, the Jacobi equation we wrote above gives
        \begin{equation}\label{eq: jacobi equation}
            A_s(x)''=-R(A_s(x),F_s'(x))\,F_s'(x)=-\mathcal{R}(s,x)\,A_s(x),
        \end{equation}
        where $\mathcal{R}(s,x):T_{F_s(x)}M\to T_{F_s(x)}M$ is given by $\mathcal{R}(s,x)(v):=R(v,F_s'(x))F_s'(x)$. The latter is symmetric, by simple symmetry properties of the Riemann tensor. Plugging now \eqref{eq: jacobi equation} into \eqref{eq: symmetric B}, we obtain that $\left(A_s^*(x)'\,A_s(x)-A_s^*(x)\,A_s(x)'\right)$ is constant in $s$, since for $s=0$, $A_0(x)=\operatorname{id}_{T_xM}$, and $A_0'(x)=(D^2H\circ D\varphi)\,\tilde{D}^2\varphi$. The latter is in principle not symmetric, but can be made symmetric for the purposes of establishing \eqref{eq: trace inequality} by multiplying from the right and from the left with $\sqrt{D^2H}$ as in McCann's proof of \cite[Prop.\ 6.1]{McCann:2020}. The result follows.
    \end{proof}

    \begin{Remark}
        From the results of Proposition \ref{prop: jacobian along geodesics}, we can observe convexity of the functional $\phi$ if we establish a lower bound on the Ricci curvature on timelike vector. To be more precise, assume that $\Ric$ is bounded from below by $K\in\R$ in all timelike directions over the set $Z(\operatorname{supp}(\mu_0\times\mu_1))$. Then, noticing that, since the curves $F_s(x)$ describe geodesics for each $x\in\operatorname{supp}\mu_0$, and hence $|F_s'(x)|_g=\ell(x,F_1(x))$, we have
        \begin{equation*}
            \phi''(s)-\frac{1}{n}(\phi'(x))^2\geq K\,\ell^2(x,F_1(x)).
        \end{equation*}
    \end{Remark}

    \begin{Lemma}[Second finite-difference representation, {\cite[Lem.~6.3]{McCann:2020}}]\label{lem: second finite difference representation}
       If $\phi\in L^{\infty}([0,1])$ is semiconvex on $(0,1)$ and $g(s,t):=\min\{s,t\}-st$, then
       \begin{equation}
           (1-t)\,\phi(0)+t\,\phi(1)-\phi(t)=\int_{[0,1]}\phi''(s)\,g(s,t)\,ds
       \end{equation}
       for each $t\in[0,1]$, where $\phi''$ denotes the distributional second derivative of $\phi$.
    \end{Lemma}

\begin{Definition}[Relative/Boltzmann--Shannon entropy]
Given a function $V \in C^2(M)$, let $\mathfrak m:=e^{-V} \mathrm{vol}_g$. The \emph{Boltzmann--Shannon} or \emph{relative entropy} $E_V:\mathcal{P}(M) \to [-\infty,\infty]$ is defined by
\begin{equation*}
        E_V(\mu):= \int \frac{d\mu}{d\mathfrak m}\, \log \frac{d\mu}{d\mathfrak m} \, d\mathfrak m
    \end{equation*}
    if this expression is well-defined in $[-\infty,\infty]$, and set to $-\infty$ otherwise.
\end{Definition}

\begin{Remark}
    It is a well-known fact that in the case when $\mu\in\mathcal{P}(M)$ vanishes outside a set $U$ of finite volume $\operatorname{vol}_g(U)<\infty$, as, for example, in the case where $\mu$ is compactly supported, the following bound holds:
    \begin{equation}\label{eq:previous bound}
        E_V(\mu)\geq-\log\mathfrak{m}(U)>-\infty,
    \end{equation}
    as an application of Jensen's inequality.
\end{Remark}

    \begin{Theorem}[Displacement Hessian of the relative entropy]\label{thm: displacement hessian of the entropy}
        Let $u:(0,\infty)\to\R$ be an admissible function. Fix $V\in C^2(M)$, a $u$-separated pair $(\mu_0,\mu_1) \in \mathcal P^{ac}_c(M)^2$, and let $s\in[0,1]\mapsto\mu_s=(F_s)_{\#}\mu_0\in\mathcal{P}^{ac}_c(M)$ be the unique $u$-geodesic described by Corollary \ref{cor: lag charact og geodesics}. If $e(0)$ and $e(1)$ are finite, then:
        \begin{enumerate}
            \item the relative entropy $s \mapsto e(s):=E_V(\mu_s)$ is continuous and semiconvex on $[0,1]$, and continuously differentiable on $(0,1)$, with
            \begin{equation}\label{eq: e'}
                e'(s)=\int_M[DV_{F_s(x)}F_s'(x)-\operatorname{Tr}B_s(x)]\,d\mu_0(x),\quad \text{and}
            \end{equation}
            \begin{equation}\label{eq: e''}
                e''(s)=\int_M[\operatorname{Tr}(B_s^2(x))+(\operatorname{Ric}+D^2V)_{F_s(x)}(F_s'(x),F_s'(x))]\,d\mu_0(x)
            \end{equation}
            holding in the distributional sense. Here, $A_s(x):=\tilde{D}F_s(x):T_xM\to T_{F_s(x)}M$ denotes the approximate derivative of $F_s$, $B_s(x):=A_s'(x)\,A_s(x)^{-1}$, $':=\frac{\partial}{\partial s}$ and $\operatorname{Tr}(B_s(x)^2)\geq\frac{1}{n}\,(\operatorname{Tr}(B_s(x)))^2$.
            \item The integral expression for $e''(s)$ depends lower semicontinuously on $s\in[0,1]$, and the integrand is bounded below.
        \end{enumerate}
    \end{Theorem}
    \begin{proof}
        Let $F_s(x)=\exp_x s\,v(D\varphi(x),x;u_\lambda) = \exp_x D^2H(D\varphi(x),x;u_\lambda^*)$, where $\varphi$ is the Kantorovich potential for the $u$-separated endpoints $\mu_0$ and $\mu_1$, $\lambda:=\ell_u(\mu_0,\mu_1) \in (0,\infty)$, and consider the $u$-geodesic $\mu_s:=(F_s)_{\#}\mu_0$. We write $JF_s(x):=\vert\det\tilde{D}F_s(x)\vert$ as in Lemma \ref{lem: maps and their jacobian der}, which asserts that this map exists and depends smoothly on $s\in[0,1]$ for each $x\in X_0$, where $X_0$ is a set of full measure in $\operatorname{supp}(\mu_0)$.

        As in Corollary \ref{cor: monge ampere}, let $ \rho_s:=d\mu_s/d\mathrm{vol}_g$, and we have $\rho_s(F_s(x))\,JF_s(x)=\rho_0(x)>0$ on $X_s\subset X_0$ of full $\mu_0$ measure (which may differ with $s\in[0,1]$).

        Furthermore, let $Z:=Z_s(\operatorname{supp}(\mu_0\times\mu_1))$. By Proposition \ref{prop: interpolants inherit compactness}, $\operatorname{supp}(\mu_s)\subseteq Z$. Therefore, using the definition of the $u$-geodesic as $\mu_s=(F_s)_{\#}\mu_0$, we get
        \begin{equation}\label{eq: displacement hessian 1}
            \begin{aligned}
                -\infty&<-\log\int_Z\operatorname{e}^{-V}d\mathrm{vol}_g
            \\ &\leq e(s)=\int_M\rho_s(y)\, e^{V(y)} \,\log(\rho_s(y)\, e^{V(y)})\,\operatorname{e}^{-V(y)}\,d\mathrm{vol}_g(y)
            \\&=\int_M\left(\log\rho_s(y)+V(y)\right)\,d\mu_s(y)\\&=\int_M\left(\log\rho_s(F_s(x))+V(F_s(x))\right)\,d\mu_0(x)
            \\&=\int_M\left(\log\rho_0(x)-\log\vert JF_s(x)\vert+V(F_s(x))\right)\,d\mu_0(x),
            \end{aligned}
        \end{equation}
        where the equality in the third line follows from noticing that $\rho_s\,e^V=d\mu_s/d\mathrm{vol}_g$, and the last equality is a consequence of the Monge-Ampère equation \eqref{eq: monge ampere}. After substituting and operating, we get
        \begin{equation}\label{eq: second finite rep 2}
            (1-t)e(0)+te(1)-e(t)=\int_M\left((1-t)\,\phi_x(0)+t\,\phi_x(1)-\phi_x(t)\right)\,d\mu_0(x),
        \end{equation}
        where $\phi_x(s)=-\log\vert JF_s(x)\vert+V(F_s(x))$. Notice that the first summand coincides with the function $\phi$ from Proposition \ref{prop: jacobian along geodesics}. By this result, if we set $A_s(x):=\tilde{D}F_s(x)$ and $B_s(x):=A_s'(x)\,A_s(x)^{-1}$, the inequality $\operatorname{Tr}\left(B_s(x)^2\right)\geq\frac{1}{n}\left(\operatorname{Tr}B_s(x)\right)^2$ holds, which gives, together with the other results from Proposition \ref{prop: jacobian along geodesics} and the chain rule:
        \begin{equation*}
           \phi'_x(s)=-\operatorname{Tr}B_s(x)+DV(F_s(x))\,F_s'(x), 
        \end{equation*}
        and
        \begin{equation}\label{eq: bound on phi''}
            \begin{aligned}
                \phi''_x(s)&=\operatorname{Ric}_{F_s(x)}\left(F_s'(x),F_s'(x)\right)+\operatorname{Tr}\left(B_s^2(x)\right)+D^2V\left(F_s'(x),F_s'(x)\right)
            \\&\geq K_Z\,\ell(x,F_1(x))^2,
            \end{aligned}
        \end{equation}
        where $K_Z$ denotes a lower bound for $\operatorname{Ric}+D^2V\geq K_Z\,g$ on the compact set $Z$, having used that $(F_s(x))_{s\in[0,1]}$ is a geodesic between $x$ and $F_1(x)$, so $\vert F_s'(x)\vert_g^2=\ell(x,F_1(x))^2$. The fact that $F_s$ is a maximising geodesic segment has been used as well to derive $\phi''_x(s)$, by $\nabla_{F_s'(x)}F'_s(x)=0$.

        Notice now that we can expand the integrand on the right hand side of \eqref{eq: second finite rep 2} using Lemma \ref{lem: second finite difference representation}, and the known expression for $\phi''_x(s)$ given by Proposition \ref{prop: jacobian along geodesics} to obtain:
        \begin{equation}\label{eq: displacement hessian 2}
            \begin{aligned}
                (1-t)\,e&(1)+t\,e(0)-e(t)=\int_M\int_{[0,1]}\phi_x''(s)\,g(s,t)\,ds \, d\mu_0(x)
            \\&=\int_M\int_{[0,1]}\left(\operatorname{Tr}(B_s^2(x))+(\operatorname{Ric}+D^2V)_{F_s(x)}(F_s'(x),F_s'(x))\right)\,g(s,t)\,ds\,d\mu_0(x)
            \\ &\geq \frac{K_Z}{2}\,t\,(1-t)\,\int_M \ell(x,F(x))^2\,d\mu_0(x),
            \end{aligned}
        \end{equation}
        where in the last inequality we have used the bound on $\phi_x''(s)$ in \eqref{eq: bound on phi''}, together with the identity $\int_{[0,1]}g(s,t)\operatorname{d}s=t\,(1-t)/2$. Using this, and noting that every subsegment of a $u$-geodesic is also a $u$-geodesic, we can use the previous bound to obtain semiconvexity and upper boundedness of $e$ on $[0,1]$. Indeed, for $0\leq s< r\leq1$, consider the $u$-geodesic joining $\mu_s$ to $\mu_r$, $(\tilde{\mu}_t=\mu_{(r-s)t+s})_{t\in[0,1]}$. For this geodesic, $\tilde{e}(0)=e(s)$ and $\tilde{e}(1)=e(r)$. Choose $t=1/2$, and $\tilde{e}(t)=e\left(\frac{s+r}{2}\right)$. Therefore,
        \begin{align*}
            \frac{e(s)+e(r)}{2}-e\left(\frac{s+r}{2}\right)&\geq\frac{K_Z}{8}\int_M\ell(x,\tilde{F}_1(x))^2\,d\tilde{\mu}_0(x)
            \\&=\frac{K_Z}{8}\,(r-s)^2\int_M\ell(x,F_1(x))^2\,d\mu_0(x)
            \\&\geq\frac{1}{8}\min\{K_Z,0\}\sup_{x,y\in Z}\ell(x,y)^2
            \\&>-\infty.
        \end{align*}
        This proves semiconvexity and upper boundedness of $e$ on $[0,1]$. Furthermore, it gives continuity of $e$ on $(0,1)$, since $e(s)$ is bounded below by \eqref{eq: displacement hessian 1}, and we have finiteness of $e(0)$ and $e(1)$ by assumption.

        We can use again Lemma \ref{lem: second finite difference representation} to obtain, by comparison with \eqref{eq: displacement hessian 2}, a distributional expression for $e''$, which is precisely \eqref{eq: e''}:
        \begin{align*}
            e''(s)&=\int_M \left( \operatorname{Tr}(B_s^2(x))+(\operatorname{Ric}+D^2V)_{F_s(x)}(F_s'(x),F_s'(x))\right)d\mu_0(x)
            \\ &\geq K_Z\int_M\ell(x,F_1(x))^2d\mu_0(x).
        \end{align*}
        The above lower bound on $e''(s)$ implies continuity of $e$ at the endpoints $[0,1]$, since otherwise $e''$ would contain a derivative of a Dirac delta measure.

        Using the previous expressions and Fubini's Theorem, we can integrate $e''$ to obtain the corresponding versions of $e'$ and $e$:
        \begin{align*}
            e'(s)&=c_1+\int_M\left(DV_{F_s(x)}F_s'(x)-\operatorname{Tr}B_s(x)\right)\,d\mu_0(x),\\
            e(s)&=c_0+c_1s+\int_M\left(V(F_s(x))-\log\vert JF_s(x)\vert\right)\,d\mu_0(x).
        \end{align*}
        Since for $s=0$ the map $F_0$ is constant, and using the expression of $e(s)$ derived in \eqref{eq: displacement hessian 1}, it is easily obtained that $c_0=E_0(\mu_0)$ and $c_1=0$, which establishes \eqref{eq: e'}.

        Finally, to prove the claimed lower semicontinuity of $e''$, notice that on a set of full measure $X_0$, $\phi_x''(s)$ depends smoothly on $s\in[0,1]$, as a consequence of Proposition \ref{prop: jacobian along geodesics} and \eqref{eq: bound on phi''}, and can be  bounded below by independently of $x\in Z$ by \eqref{eq: bound on phi''}. With this, lower semicontinuity of $e''(s)$ follows immediately from Fatou's Lemma. Indeed, take a sequence $f_n(x):=\phi_x''(s_n)$ for $s_n\to s$. Therefore, by the smoothness of $\phi_x''$, $(f_n)_{n\in\N}$ is a sequence of measurable functions, and it holds that $f_n\to\phi''_\cdot(s)$. A direct application of Fatou's Lemma shows that $\int_M\phi''_xd\mu_0(x)\leq\liminf_{n\to\infty}\int_Mf_n(x)d\mu_0(x)$, which is precisely the statement of $e''(s)$ being lower semicontinuous.

        To conclude that $e(s)$ is continuously differentiable, note that there exists a constant $k:=\vert K_Z\vert\,\ell(x,F_1(x))^2$ such that the sum $\phi_x'(s)+ks$ increases continuously with $s\in[0,1]$. Hence, continuity of $e'(s)$ on $(0,1)$ follows from the representation \eqref{eq: e'} by Lebesgue's dominated convergence theorem. Indeed, note that for any sequence $(s_n)_{n\in\N}$ converging to $s\in(0,1)$, without loss of generality we can assume the sequence to be bounded away from 0 and 1, and we can consider the associated $f_n(x):=\phi'_x(s_n)$. This is a sequence of integrable functions converging pointwise to $f(x):=\phi'_x(s)$, and it is dominated by $\phi'_x(\sup_{n\in\N}s_n)+k$, which is integrable. Then, the dominated convergence theorem gives $\int_{M} f_n\, d\mu_0\to \int_Mf\, d\mu_0$ as $s_n\to s$, as desired.
    \end{proof}

     \begin{Definition}[$(K,N,u)$-convexity]\label{def: (k,n,u)-convexity}
        Fix $K\in\R$ and $N \in (0,\infty]$. 
        \begin{enumerate}
            \item  A function $e:[0,1]\to[-\infty,\infty]$ is said to be \emph{$(K,N)$-convex} if $e$ is upper semicontinuous, $\operatorname{Dom}e:=\{s\in[0,1]:\ e(s)<\infty\}$ is connected, and either $e^{-1}(-\infty)$ contains the interior $I$ of $\operatorname{Dom}e$ or is empty, and in the latter case: $e$ is semiconvex throughout $I$ and satisfies
        \begin{equation*}
            e''(s)-\frac{1}{N}(e'(s))^2\geq K
        \end{equation*}
        in the distributional sense on $I$, where $(1/\infty) (e')^2:= 0$.
        
        \item Given a globally hyperbolic spacetime $(M^n,g)$ and an admissible function $u$, a functional $E:\mathcal{P}(M)\to\R\cup\{\pm\infty\}$ is said to be \emph{weakly $(K,N,u)$-convex for $Q\subset\mathcal{P}(M)^2$} if for each $(\mu_0,\mu_1)\in Q$ there is a $u$-geodesic in $\mathcal{P}(M)$ joining $\mu_0$ to $\mu_1$ (in particular, $\ell_u(\mu_0,\mu_1) \in (0,\infty)$) on which $E(\mu_s)$ is $(K\,\ell_u(\mu_0,\mu_1)^2,N)$-convex. $E$ is said to be \emph{$(K,N,u)$-convex for $Q$} if, in addition, $E(\mu_s)$ is $(K\,\ell_u(\mu_0,\mu_1)^2,N)$-convex for all $u$-geodesics $s\in[0,1]\mapsto\mu_s\in\mathcal{P}(M)$ with endpoints in $Q$.
        \end{enumerate}
    \end{Definition}

\begin{Lemma}\label{lem: simple lemma}
        For $a,b\in\R$ and $\varepsilon>0$, the following inequality holds: $(a-b)^2\leq(1+\varepsilon^{-1})\,a^2+(1+\varepsilon)\,b^2$.
    \end{Lemma}
    \begin{proof}
        The inequality claimed can be expanded to obtain that it is equivalent to $\varepsilon^{-1}\,a^2+\varepsilon\, b^2+2ab\geq0$. If $a$ or $b$ are equal to 0, then the inequality is trivial, so assume from now on that they are both not 0.
    
        Consider the function $f:(0,\infty)\to\R$ defined by $f(x)=x^{-1}a^2+x b^2+2ab$. The function $f$ is clearly smooth, hence we can write its derivative $f'(x)=b^2-a^2x^{-2}$. We obtain that $f'(x)=0$ if and only if $x=\vert\frac{a}{b}\vert$. Furthermore, $f''(\vert\frac{a}{b}\vert)>0$, which makes the point $\vert\frac{a}{b}\vert$ a minimum. Hence, since $f(\vert\frac{a}{b}\vert)=2\vert a\vert\vert b\vert+2ab\geq0$, the result follows.
    \end{proof}

    We are now in position to prove our main result of this subsection, namely the fact that timelike lower Ricci curvature bounds imply convexity of the relative entropy $E_V$ along $u$-geodesics between $u$-separated measures.

    \begin{Theorem}[Entropic convexity from timelike lower Ricci curvature bounds]\label{cor: entropic convexity from ricci bounds}
        Let $(M^n,g)$ be a globally hyperbolic spacetime. Fix $V\in C^2(M)$ and $N \in [n,\infty]$ (where $N = n$ by convention means $V = 0$ and $\Ric^{(n,0)}:=\Ric$). For any admissible function $u:(0,\infty)\to\R$, denote by $Q_u\subset\mathcal{P}_c^{ac}(M)^2$ the set of pairs of absolutely continuous probability measures which are $u$-separated. If $K \geq 0$ and $\operatorname{Ric}^{(N,V)}(v,v)\geq K|v|^2_g\geq0$ holds in every timelike direction $(x,v)\in TM$, then for any $(\mu_0,\mu_1)\in Q_u$, the relative entropy $s\mapsto E_V(\mu_s)$ along the unique $u$-geodesic from $\mu_0$ to $\mu_1$ is $(K,N,u)$-convex. For a general $K \in \R$, $e(s):=E_V(\mu_s)$ satisfies the distributional inequality
        \begin{equation*}
            e'' - \frac{1}{N} (e')^2 \geq K\,\|\ell\|_{L^2(\pi)}^2,
        \end{equation*}
        where $\pi$ is the unique optimal coupling of $(\mu_0,\mu_1)$.
    \end{Theorem}
    \begin{proof}
       We assume $n < N < \infty$, the edge cases can be argued via limits. Fix an admissible function $u:(0,\infty)\to\R$, and let $(\mu_0,\mu_1)\in Q_u$. Corollary \ref{cor: lag charact og geodesics} gives the unique $u$-geodesic from $\mu_0$ to $\mu_1$, given by $(\mu_s:=(F_s)_{\#}\mu_0)_{s\in[0,1]}$, which satisfies that $\mu_s\in\mathcal{P}^{ac}_c(M)$ for all $s\in[0,1)$.

       Consider the entropy along the $u$-geodesic, and denote it by $e(s):=E_V(\mu_s)$. By \eqref{eq:previous bound}, $e(s)>-\infty$. 

       Assume first that $e(0)$ and $e(1)$ are both finite. Then, we can use Theorem \ref{thm: displacement hessian of the entropy} to establish bounds on $\frac{1}{N}e'(s)^2$:
        \begin{align*}
            \frac{1}{N}e'(s)^2&=\frac{1}{N}\left(\int_M[DV_{F_s(x)}F_s'(x)-\operatorname{Tr}B_s(x)]\, d\mu_0(x)\right)^2
            \\&\leq \frac{1}{N}\int_M [DV_{F_s(x)}F_s'(x)-\operatorname{Tr}B_s(x)]^2d\mu_0(x) \\&\leq\int_M\left(\frac{1+\varepsilon^{-1}}{N}\vert DV_{F_s(x)} F_s'(x)\vert^2+(1+\varepsilon)\frac{n}{N}\operatorname{Tr}(B_s^2(x))\right)d\mu_0(x)
            \\&=\int_M\left(\frac{1}{N-n}\vert DV_{F_s(x)} F_s'(x)\vert^2+\operatorname{Tr}(B_s^2(x)\right)d\mu_0(x),
        \end{align*}
         where in the first inequality we have used Jensen's inequality, for the second inequality we have used \eqref{eq: trace inequality} and Lemma \ref{lem: simple lemma}, and for the last equality we have chosen $\varepsilon=\frac{N-n}{n}>0$. Using this bound on $\frac{1}{N}e'(s)^2$ together with the expression of $e''(s)$ given by \eqref{eq: e''}, we obtain
        \begin{align*}
            e''(s)-\frac{1}{N}e'(s)^2&\geq\int_M\operatorname{Ric}^{(N,V)}(F_s'(x),F_s'(x))d\mu_0(x)
            \\ &\geq K\int_M\ell(x,F_1(x))^2d\mu_0(x) = K \|\ell\|_{L^2(\pi)}^2 \\
            (\text{only if } K \geq 0) \quad &\geq K\,\left(\int_M\ell(x,y)d\pi(x,y)\right)^2\\
             &\geq K\,u_{\lambda}^{-1}\left(\int_M u_{\lambda}(\ell(x,y))d\pi(x,y)\right)^2=K\,u_{\lambda}^{-1}(u(1))^2
            \\&= K\,\ell_u(\mu_0,\mu_1)^2,
        \end{align*}
        where in the inequality in the second line we have used the hypothesised bound on the Bakry-Émery Ricci tensor, for the inequality in the third line we have used Jensen's inequality and the form of the optimal coupling $\pi=(\operatorname{id}\times F_1)_{\#}\mu_0$ given by Theorem \ref{thm: characterising opt maps}, and for the inequality in the fourth line we have used again Jensen's inequality together with Proposition \ref{prop: existence of optimal coupling}.
        In the case that either $e(0)$ or $e(1)$ are not finite, we can apply the argument above on any subinterval of $[0,1]$ with finite entropy at its endpoints, to reach the conclusion.
    \end{proof}

    \subsection{Relaxation of $u$-separation}
    \label{Subsection: relaxation}

    In this subsection, we show that entropic convexity extends to $u$-geodesics whose endpoints are not necessarily $u$-separated. We follow McCann's approach \cite[Sec.\ 7]{McCann:2020}, but have to deal with the additional challenge that, a priori, restrictions of optimal couplings may no longer be optimal due to the nonlinearity of the optimal transport problem defining $\ell_u$.

    \begin{Theorem}[Maps characterising interpolants without duality]\label{thm: relaxing separation}
        Let $(M,g)$ be a globally hyperbolic spacetime. Fix $V\in C^2(M)$ and an admissible function $u:(0,\infty)\to\R$. Let $\mu\in\mathcal{P}^{ac}(M)$ and $\nu\in\mathcal{P}(M)$ be such that $\lambda:=\ell_u(\mu,\nu) \in (0,\infty)$, and suppose there are lower semicontinuous functions $a,b: M \to \R$ such that $a \in L^1(\mu), \, b \in L^1(\nu)$, and we have $u_{\lambda}(\ell(x,y)) \leq a\oplus b(x,y):=a(x) + b(y)$ on $\mathrm{supp}(\mu) \times \mathrm{supp}(\nu)$. Then the following hold: 
        \begin{enumerate}
            \item There is at most one $\ell_u$-optimal coupling $\pi\in\Pi_{\leq}^{u_{\lambda}}(\mu,\nu)$ such that $\ell>0$ holds $\pi$-a.e.\ and $\int u_\lambda \circ \ell \, d \pi = u(1)$.
            \item If such a coupling exists, then $\pi=(\operatorname{id}\times F)_{\#}\mu$ for some map $F:\operatorname{supp}\mu\to\operatorname{supp}\nu$, and $\pi(\operatorname{sing}(\ell))=0$.
            \item In the hypotheses of (ii), $\pi$ decomposes into countably many nonnegative, mutually singular measures, i.e.\ $\pi=\sum_{i\in\N}\pi^i$, such that $\cup_{i\in\N}\operatorname{supp}\pi^i$ is $u_{\ell_u(\mu,\nu)}\circ\ell$-cyclically monotone and the marginals $(\hat{\mu}^i,\hat{\nu}^i)$ of $\hat{\pi}^i:=\pi^i/\pi^i(M^2)$ satisfy that $\operatorname{supp}(\hat{\mu}^i\times\hat{\nu}^i)$ is compact and disjoint from $\{\ell\leq0\}$. For each $i\in\N$, the map $F$ agrees $\mu^i$-a.e. with the unique $\ell_u$-optimal map $F^i$ pushing $\hat{\mu}^i$ forward to $\hat{\nu}^i$ from Theorem \ref{thm: characterising opt maps}. Moreover, $\operatorname{graph}(F^i)\subset\operatorname{supp}\pi^i$.
            \item The $u$-geodesic $(\mu_s)_{s\in[0,1]}$ defined by $\mu_s:=(\bar z_s)_{\#}\pi$ satisfies $\mu_s\in\mathcal{P}^{ac}(M)$ for $s\in[0,1)$.
            \item The measures $\mu^i_s:=(z_s)_{\#}\pi^i$ decompose $\mu_s$ into mutually singular pieces for $s\in[0,1)$.
            \item The sum $\pi=\sum_i\pi^i$ is finite if and only if $\operatorname{supp}\pi$ is compact and disjoint from $\{\ell\leq0\}$.
        \end{enumerate}
    \end{Theorem}
    \begin{proof}
       The idea of the proof is to find a suitable covering of the support of $\mu\times\nu$. In particular, we will need that the closure of each element of the covering is compact. By restricting the product measure to each of these suitable elements of the covering, we will be able to construct a decomposition of the optimal coupling $\pi$, provided it exists, and apply the results developed in Section \ref{section: OT} regarding the existence and uniqueness of optimal couplings for compactly supported measures. These will be glued together to obtain the final desired decomposition for $\pi$ and the corresponding $u$-geodesic.

       Assume first that (i) and (ii) hold, and suppose that there exists a unique $\ell_u$-optimal coupling $\pi\in\Pi_{\leq}^{u_{\lambda}}(\mu,\nu)$ which satisfies $\int u_\lambda \circ \ell \, d \pi = u(1)$ and for which $\ell>0$ holds $\pi$-a.e.\ By Theorem \ref{thm: smoothness of l}, $\{\ell>0\}$ is open, and since $M$ is a manifold, in particular it is locally compact and Hausdorff, therefore $\{\ell>0\}\cap\operatorname{supp}\pi$ can be covered by open rectangles $U\times W$ whose compact closures are contained in $\{\ell>0\}$. Furthermore, since $M$ is second countable, only countably many of these rectangles are necessary. Write the covering collection as $\{U_i\times W_i\}_{i\in\N}$.
       With this, the decomposition of $\pi$ is built inductively as follows: for $i=0$, fix $\pi^0=0$, and for $i\in\N$, $\pi^i$ is defined as the restriction of $\pi - \sum_{j < i}\pi^j$ to $U_i\times W_i$. Therefore, $\pi$ can be written as $\pi=\sum_{i\in\N}\pi^i$, and it is clear by construction that the summands are mutually singular, and each $\pi^i$ is nonnegative and vanishes outside the $i$-th rectangle, which makes them compactly supported. Moreover, by construction, $\pi$ is supported on $\cup_{i\in\N}\operatorname{supp}\pi^i$, and since $\pi$ is assumed to be $\ell_u$-optimal and to satisfy $\int u_\lambda \circ \ell \, d\pi = u(1)$, this set is $u_{\lambda}\circ\ell$-cyclically monotone, with $\lambda:=\ell_u(\mu,\nu)$, by Proposition \ref{Prop: cyclmonotoneoptimal}.

       It is immediate from the construction that the case in which $\operatorname{supp}\pi$ is compact and contained in $\{\ell>0\}$ gives a finite decomposition, proving the ``if" implication of (vi). Now, consider the finite sum $\sum_{k=1}^i\pi^k$. The support of this measure is compact, as it is the finite union of compact sets, and since every one of these compact sets is disjoint from $\{\ell\leq0\}$, a closed set, then so is the union. This establishes the ``only if" implication of (vi).

       Normalise $\hat{\pi}^i=\frac{1}{\pi^i(M^2)}\pi^i$ whenever $\pi^i$ is non-vanishing.
       Consider now, for each $i\in\N$, the marginals of $\pi^i$ and $\hat{\pi}^i$, denoted by $(\mu^i,\nu^i)$ and $(\hat{\mu}^i,\hat{\nu}^i)$, respectively. Note that, since $\pi^i$ is by construction compactly supported, so are the two pairs of marginals, and the supports $\operatorname{supp}(\hat{\mu}^i\times\hat{\nu}^i)$ are by construction disjoint from $\{\ell\leq0\}$. Note that $\hat{\mu}^i$ is the normalisation of $\mu^i$, and same with $\hat{\nu}^i$ and $\nu^i$. This establishes the first claim in (iii).
       
       To prove the second claim in (iii), notice that since $\hat{\pi}^i$ is a probability measure, so are its compactly supported marginals $(\hat{\mu}^i,\hat{\nu}^i)$. By the disjointness of $\operatorname{supp}(\hat{\mu}^i\times\hat{\nu}^i)$ from $\{\ell\leq0\}$ obtained in the previous paragraph, this pair is $u$-separated (cf.\ Remark \ref{Remark: simple u separation}), and we can apply Theorem \ref{thm: characterising opt maps}. This result gives the existence of a unique $u$-optimal coupling $\tilde{\pi}^i$ (which does not necessarily agree \textit{a priori} with $\hat{\pi}^i$ because of the non-linearity in the definition of the Lorentz-Orlicz-Wasserstein time separation on the couplings) of the form $(\operatorname{id}\times \tilde{F}^i_1)_{\#}\hat{\mu}^i$, where $\lambda_i:=\ell_u(\hat{\mu}^i,\hat{\nu}^i)$, $\tilde{F}^i_s(x):=\exp_xs\,DH(D\tilde{\varphi}^i,x;u_{\lambda_i}^*)$ and $\tilde{\varphi}^i$ is the Kantorovich potential between $(\hat{\mu}^i,\hat{\nu}^i)$ given by $u$-separation. From this, one obtains a coupling $\bar{\pi}^i\in\Pi_{\leq}(\mu^i,\nu^i)$ given by $\bar{\pi}^i=\pi^i(M^2)\,\tilde{\pi}^i=(\operatorname{id}\times\bar{F}^i_1)_{\#}\mu^i$, where $\mu^i$ has absorbed the normalisation constant, and $\bar{F}^i_s:=\tilde{F}^i_s$. Another consequence of Theorem \ref{thm: characterising opt maps} is that $\bar{\pi}^i(\operatorname{sing}(\ell))=0$. 
       
       Note that the functions $\bar{F}^i_1$, $i\in\N$, satisfy that $\operatorname{Dom}\bar{F}^i_1\subset\operatorname{supp}\mu^i$, which is compact, since $\operatorname{supp}\pi^i$ is too. Therefore, each $\bar{F}_1^i$ can be extended to $\operatorname{supp}\mu^i$ in such a way that $\operatorname{graph}(\bar{F}^i_1)\subset\operatorname{supp}\pi^i$. This gives the last claim in (iii).
       
       To conclude the proof of (iii), we need to show that the map $F$ given by (ii) agrees $\mu^i$-a.e. with the map $\bar{F}_1^i\equiv F^i$, $i\in\N$, which we have just obtained above. We will do it in the following way: we define $F_s(x):=\bar z_s(x,F(x))$, and show that this must agree with $F^i_s$ $\mu^i$-a.e., for all $s\in[0,1]$. Note that $F_s$ is well defined and depends smoothly on $s$, by Lemma \ref{lemma: midpoint continuity away from cut locus}, and the assumption that $\pi(\operatorname{sing}(\ell))=0$.

       Assume for contradiction that for some $i\in\N$ there exists a set $S\subset M$ of positive $\mu^i$-measure on which $F^i\neq F$. Furthermore, since $\mu^i$ is absolutely continuous with respect to $\operatorname{vol}_g$, it is given by a density with respect to the volume measure. It is possible to make $S$ smaller so that it still has positive $\mu^i$-measure but it is compact, hence this density would be bounded in $S$. By Proposition \ref{prop: jacobian along geodesics}, the map $F_s^i$ has approximate derivative $\tilde{D}F^i_s(x)$ for $\mu^i$-a.e.\ $x$, which depends smoothly on $s\in[0,1]$. Furthermore, by compactness of $S$ and smoothness on the parameter $s$, this approximate derivative (when $x$ is fixed) is bounded above and below throughout $S$.
       
       Since $S$ is compact, we can choose $r>0$ such that if we consider the $r$-neighbourhood of $S$, denoted by $S_r$, the volume does not grow much, i.e.\ $\operatorname{vol}_g(S_r)<\frac{3}{2}\operatorname{vol}_g(S)$. On the other hand, the properties of $F^i_s$ and $\tilde{D}F^i_s$ give that the map $F_s^i$ admits a bi-Lipschitz restriction to $S$ for $s\in[0,\frac{1}{2}]$. On the other hand, recall that the family $F_s$ depends smoothly on $s\in[0,\frac{1}{2}]$. Furthermore, both $F_s^i$ and $F_i$ stay away from the cut locus, and coincide with the identity when $s=0$. Therefore, there exists $s>0$ sufficiently small such that $F^i_s(S)$ and $F_s(S)$ are both contained in $S_r$ and both have a sufficiently large volume, i.e.\ $\operatorname{vol}_g(F^i_s(S)),\operatorname{vol}_g(F_s(S))>\frac{3}{4}\operatorname{vol}_g(S)$. This forces the intersection of $F^i_s(S)$ and $F_s(S)$ to have positive volume, so, in particular, there exist $x,y\in S\subset M$ for which $F^i_s(x)=F_s(y)$. In particular, the Lagrangian trajectories starting from $x$ and $y$ would cross. Now, since the graphs of $F^i$ and $F$ lie on a $u_{\lambda}$-cyclically monotone set, namely $\operatorname{supp}(\pi^i+\pi^j)\subset\operatorname{supp\pi}$,\ cf. Proposition \ref{Prop: cyclmonotoneoptimal}, we are in position to apply Proposition \ref{prop: lagrangian trajectories dont cross}, which implies that $x=y$. In particular, by the same result, $F^i_1(x)=F_1(x)$. This argument can be made for almost every point in a positive measure subset of $S$, namely the preimage of $F^i_s(S)\cap F_s(S)$, therefore it concludes the argument, since it raises a contradiction with the definition of $S$, which concludes the proof of (iii).

       Note that the conclusion of (iii) has one particular consequence, that does not hold \textit{a priori} for the (Lorentz-)Orlicz-Wasserstein distance, namely that the restriction of an optimal coupling given by transport rays that do not cross is still an optimal coupling between its normalised marginals. Uniqueness will give that $\tilde{\pi}^i$ agrees with $\hat{\pi}^i$, and $\bar{\pi}^i$ agrees with $\pi^i$. 
       
       Next, we prove (i) and (ii). To prove (ii), assume that there is an $\ell_u$-optimal coupling $\pi\in\Pi_{\leq}^{u_{\lambda}}(\mu,\nu)$ such that $\ell>0$ holds $\pi$-a.e.\ and $\int u_\lambda \circ \ell \, d\pi' = u(1)$. Applying the decomposition just proved in (iii), the map given by $F(x):=F^i(x)$ if $x\in\operatorname{supp}\mu^i$ is well defined. We will see that this is precisely the map given in the statement of (ii). Indeed, since we have said in (iii) that for each $i\in\N$ $\pi^i$ is concentrated on the graph of $F^i$ (and therefore of $F$) and far from the singular points, by {\cite[Lem.\ 3.1]{ahmad2011optimal}} $\pi^i=(\operatorname{id}\times F^i)_{\#}\mu^i$ and $\pi=(\operatorname{id}\times F)_{\#}\mu$ by the decomposition. This gives (ii).
       
       To prove (i), we need to show uniqueness under the assumption that such a coupling exists. If there were another $\ell_u$-optimal coupling $\pi'\in\Pi_{\leq}^{u_{\lambda}}(\mu,\nu)$ with $\ell>0$ holding $\pi'$-a.e., we could apply the argument in the paragraph above to $\pi''=\frac{1}{2}(\pi+\pi')$ to deduce the existence of a map $F''$ such that $\pi''=(\operatorname{id}\times F'')_{\#}\mu$. Note that, indeed, $\pi''$ would also be an optimal coupling, by the observation made in the proof of Proposition \ref{Prop: cyclmonotoneoptimal}, and the fact that the standard optimal transport problem with cost $u_{\lambda}\circ\ell$ is linear on the couplings. But then again, by construction, $\pi$ and $\pi'$ vanish outside the graph of $F''$, so by the same result, $\pi=(\operatorname{id}\times F'')_{\#}\mu=\pi'$, which concludes the proof of (i).

       At this point, Corollary \ref{cor: lag charact og geodesics} asserts that the unique $u$-geodesic from $\hat{\mu}^i$ to $\hat{\nu}^i$ is given by $\hat{\mu}_s^i:=(\bar z_s)_{\#}\hat{\pi}^i$, and for every $s\in[0,1)$, $\hat{\mu}_s^i\in\mathcal{P}^{ac}_c(M)$. Hence, if we consider the curves given by $\mu^i_s:=(\bar z_s)_{\#}\pi^i$, these would also satisfy that for every $s\in[0,1)$ $\mu_s^i$ is an absolute continuous measure (but no longer a probability measure). On the other hand, by Theorem \ref{thm: existence of u geod}, the curve of measures defined by $\mu_s:=(\bar z_s)_{\#}\pi$ is a $u$-geodesic joining $\mu$ to $\nu$.
       To bring these two observations together, consider the previous decomposition of $\pi$, and write
       \begin{equation*}
          \mu_s:=(\bar z_s)_{\#}\pi=\sum_{i\in\N}(\bar z_s)_{\#}\pi^i=\sum_{i\in\N}\mu_s^i.
       \end{equation*}
       This is a decomposition into countably many nonnegative (not necessarily mutually singular yet) curves of absolutely continuous measures. This establishes (iv).

        Finally, we need to show (v). We have already written a decomposition of $\mu_s$, $s\in[0,1]$, however, we had not obtained yet the mutual singularity, which is our goal now. Fix $i,j\in\N$ such that $i\neq j$. It follows that $\mu^i$ and $\mu^j$ inherit mutual singularity from $\pi^i$ and $\pi^j$, because $(\operatorname{id}\times F)_{\#}(\mu^i\wedge\mu^j)$ must vanish, since it is common to $\pi^i$ and $\pi^j$. 
        
        Because $\mu$ is inner regular (it is absolutely continuous with respect to the volume measure), we obtain a collection of disjoint $\sigma$-compact sets $U^i\subset\operatorname{supp}\mu^i$ such that
        \begin{equation*}
        \mu^i(U^j):= \left\{ \begin{array}{lcc} \mu^i(M), & \text{if $i=j$}, \\ \\ 0, & \text{otherwise.} \end{array} \right.
    \end{equation*}
    We claim that $\{\mu^i_s\}_{i\in\N}$ are mutually singular for each $s\in(0,1)$. Indeed, by construction, $\mu^i_s$ vanishes outside the $\sigma$-compact set $F^i_s(U^i)$, which is disjoint from $F^j_s(U^j)$ unless $i=j$. This last claim is true because if $z\in F^i_s(U^i)\cap F^j_s(U^j)$, then $U^i\cap U^j\neq\emptyset$ by Proposition \ref{prop: lagrangian trajectories dont cross} and the $u_{\lambda}\circ\ell$-cyclical monotonicity of $\operatorname{supp}\pi$. But this forces $i=j$, which concludes the proof.
    \end{proof}

\begin{Lemma}[{\cite[Lem.\ 7.2]{McCann:2020}}]\label{lem: 7.2}
    Let $\mathfrak{m}$ and $\mu$ be Borel measures on a metric space $(X,d)$, with $\mu\ll\mathfrak{m}$, and $\mu(X)<\infty$. Set
    \begin{equation*}
        E_{\pm}(\mu\vert\mathfrak{m}):=\int_X\left(\frac{d\mu}{d\mathfrak{m}}\,\log\frac{d\mu}{d\mathfrak{m}}\right)_{\pm}\,d\mathfrak{m},
    \end{equation*}
    where $(a)_{\pm}:=\max\{\pm a,0\}$. The following hold.
    \begin{enumerate}
        \item If $0\leq\nu\leq\mu$, and $E_{+}(\mu\vert\mathfrak{m})$ (resp. $E_{-}(\mu\vert\mathfrak{m})$) is finite, then $E_+(\nu\vert\mathfrak{m})$ (resp. $E_-(\nu\vert\mathfrak{m})$) is finite. If neither is finite, then set $E(\nu\vert\mathfrak{m}):=-\infty$, otherwise, $E(\nu\vert\mathfrak{m}):=E_+(\nu\vert\mathfrak{m})-E_-(\nu\vert\mathfrak{m})$ satisfies
        \begin{equation*}
            -\mu(X)-E_-(\mu\vert\mathfrak{m})\leq E(\nu\vert\mathfrak{m})\leq E_+(\mu\vert\mathfrak{m}).
        \end{equation*}
        \item If $\mu=\sum_{i=1}^{\infty}\mu^i$, and $\mu^i$ and $\mu^j$ are mutually singular for all $i,j\in\N$, then either $E(\mu\vert\mathfrak{m})=-\infty$ or $E(\mu\vert\mathfrak{m})=\lim_{k\to\infty}E(\sum_{i=1}^k \mu^i\vert\mathfrak{m})$.
    \end{enumerate}
\end{Lemma}

 \begin{Lemma}[{\cite[Lem.\ 7.3]{McCann:2020}}]\label{lem: 7.3}
        Let $c\in\R$ and $f_k:[0,1]\to[-\infty,c]$, $k\in\N$, be a sequence of convex functions such that a subsequence $(f_{k(i)})_{i\in\N}$ converges pointwise $\mathcal{L}^1$-a.e. to a convex limit $f:[0,1]\to[-\infty,c]$ satisfying either
        \begin{equation*}
            \inf_{s\in[0,1]}f(s)>-\infty\quad\text{(proper),}\quad\text{ or }\sup_{s\in(0,1)}f(s)=-\infty\quad\text{(improper).}
        \end{equation*}
        In the proper case, the derivatives $f'=\lim_{i\to\infty}f'_{k(i)}$ converge pointwise $\mathcal{L}^1$-a.e., and the second derivatives $f''=\lim_{i\to\infty}f''_{k(i)}$ converge distributionally on $(0,1)$.
    \end{Lemma}

    \begin{Theorem}[Displacement Hessian of the relative entropy II]\label{thm: displacement hessian of the entropy ii}
        Let $(M,g)$ be a globally hyperbolic spacetime. Fix $V\in C^2(M)$ and $N \in [n,\infty]$, and assume $V=0$ if $N=n$. Let $u:(0,\infty)\to\R$ be an admissible function. Fix $\mu,\nu\in\mathcal{P}^{ac}(M)$ such that
        \begin{enumerate}
            \item $\lambda:=\ell_u(\mu,\nu)\in(0,\infty)$, and there exist lower semicontinuous functions $a,b:M \to \R$ with $a \in L^1(\mu)$, $b \in L^1(\nu)$ such that  $u_{\lambda} \circ \ell \leq a \oplus b$ on $\mathrm{supp}(\mu \times \nu)$ and there exists an optimal coupling $\pi\in\Pi^{u_{\lambda}}_{\leq}(\mu,\nu)$ with $\ell>0$ holding $\pi$-a.e.\ which satisfies $\int u_{\lambda} \circ \ell \, d\pi = u(1)$, 
            \item the relative entropy $e(s):=E_V(\mu_s)$, with $\mu_s:=(\bar z_s)_{\#}\pi$, and the map $F_s(x):=\bar z_s(x,F(x))$ from Theorem \ref{thm: relaxing separation} satisfy $\max\{e(0),e(1)\}<\infty$ and $\sup_{0<s<1}e(s)>-\infty$,
            \item \begin{equation}\label{eq: constant c}
            C:=\Bigg\vert\Bigg\vert  \int_M \min\bigg\{\operatorname{Ric}^{(N,V)}_{F_s(x)}\left(\frac{\partial F}{\partial s},\frac{\partial F}{\partial s}\right),0\bigg\}\operatorname{d}\mu_s \Bigg\vert\Bigg\vert_{L^{\infty}([0,1])}<\infty.
        \end{equation}
        \end{enumerate} 
        Then the conclusions of Theorem \ref{thm: displacement hessian of the entropy}(i) remain true, except that $e(\cdot)$ may be upper semicontinuous rather than continuous at the endpoints of the interval $s\in[0,1]$.
    \end{Theorem}

            \begin{proof}
            The idea of the proof is to use the decomposition provided by Theorem \ref{thm: relaxing separation} and apply Theorem \ref{thm: displacement hessian of the entropy} to each element.

            Let $F:\operatorname{supp}\mu\to\operatorname{supp}\nu$ be the transport map and $(\mu_s)_{s\in[0,1]}$ the $u$-geodesic given by Theorem \ref{thm: relaxing separation}. Consider the mutually singular decompositions given by the same result for $\pi=\sum_{i\in\N}\pi^i$ and $\mu_s=\sum_{i\in\N}\mu^i_s$, with $\mu^i_s=(z_s)_{\#}\pi^i$. Write $\mu^i:=\mu^i_0$ and $\nu^i:=\mu^i_1$.

            Normalise $\hat{\pi}^i:=\frac{1}{\pi^i(M)}\pi^i$, $\hat{\mu}^i:=\frac{1}{\mu^i(M)}\mu^i$, and $\hat{\nu}^i:=\frac{1}{\nu^i(M)}\nu^i$. Note that, by construction, $(\hat{\mu}^i,\hat{\nu}^i)$ are the marginals of $\hat{\pi}^i$. Since we are working with compactly supported probability measures, and by Theorem \ref{thm: relaxing separation} we have $\operatorname{supp}\pi^i\subset\{\ell>0\}$, Theorem \ref{thm: characterising opt maps} gives the unique optimal map from $\hat{\mu}^i$ to $\hat{\nu}^i$, denoted by $F^i$, which coincides $\hat{\mu}^i$-a.e.\ with $F$, by Theorem \ref{thm: relaxing separation}(iii). Moreover, as seen in the proof of Theorem \ref{thm: relaxing separation}(iii), $\hat{\pi}^i$ is $\ell_u$-optimal for this pair.

            Furthermore, $\mu^i$ and $\nu^i$ inherit an upper bound on their entropy from $\max\{e(0),e(1)\}<\infty$ by Lemma \ref{lem: 7.2}(i). Their normalised versions are also bounded, which follows from a simple calculation:
            \begin{equation}\label{eq: scaling law entropy}
                E_V(a\,\mu)=a\,E_V(\mu)+\mu(M)\,a\log a,\quad a\in(0,\infty).
            \end{equation}
            Therefore, setting $\hat{e}_i(s):=E_V(\hat{\mu}^i_s)$, Theorem \ref{thm: displacement hessian of the entropy} applies to give
            \begin{equation}\label{eq: ei'}
                \hat{e}_i '(s)=\int_M[DV_{F_s(x)}F_s'(x)-\operatorname{Tr}B_s(x)]d\hat{\mu}^i_0(x),\quad \text{and}
            \end{equation}
            \begin{equation}\label{eq: ei''}
                \hat{e}_i''(s)=\int_M[\operatorname{Tr}(B_s^2(x))+(\operatorname{Ric}+D^2V)_{F_s(x)}(F_s'(x),F_s'(x))]d\hat{\mu}^i_0(x),
            \end{equation}
            for $s\in(0,1)$. Taking into account \eqref{eq: scaling law entropy}, the same result for $e'_i(s)$ and $e''_i(s)$ holds after replacing $\hat{\mu}^i_0$ by $\mu^i_0$.

            Since by Theorem \ref{thm: relaxing separation} the measures $\{\mu_s^i\}_{i\in\N}$ are mutually singular for $s\in[0,1)$, provided that we can conclude the same for $s=1$, the final result for $e(s)$ can be obtained by summing over all elements in $\N$. First, since $\nu\in\mathcal{P}^{ac}(M)$, we could do the same reasoning in every argument up until now and obtain an optimal transport from $\nu$ to $\mu$ which is past-directed. By the uniqueness of geodesics, this transport would coincide (as curves) with that obtained from $\mu$ to $\nu$. In particular, a consequence would be the mutual singularity of the family of measures $\{\mu_1^i\}_{i\in\N}$, which is the only result necessary now.

            Therefore, we can obtain the results for $e'(s)$ and $e''(s)$ by summing $e'_i(s)$ and $e_i''(s)$ over $i\in\N$, respectively, provided these sums do not diverge. To be precise, define $f_k(s):=\sum_{i=1}^k e_i(s)$. The assumptions on $C$, combined with the fact that $\operatorname{Tr}(B_s^2(x))\geq0$, taken to the already obtained expression of $e''_i(s)$ show that the function $s\mapsto\frac{C}{2}s^2+f_k(s)$ is convex on $[0,1]$. Furthermore, since $\max\{e(0),e(1)\}<\infty$, Lemma \ref{lem: 7.2}(i) gives that $f_k(0)$ and $f_k(1)$ are bounded above in terms of $C$ and the endpoints $\mu$ and $\nu$ (therefore there is a uniform bound for $k\in\N$). On the other hand, Lemma \ref{lem: 7.2}(ii) asserts that for each $s\in[0,1]$ either $e(s):=E_V(\mu_s)=\lim_{k\to\infty}f_k(s)$ or $e(s)=-\infty$. The following step is to check that if the second case holds for some $s\in(0,1)$, it will be true for all $s\in(0,1)$, which would be a contradiction with hypothesis (ii). We state this as a Claim for the moment, and postpone the proof until we have obtained the proof of the Theorem.

            \textit{Claim}: If $e(s)=-\infty$ for some $s\in(0,1)$, then $\sup_{t\in(0,1)}e(t)=-\infty$.

            By the Claim, if there exists $s\in(0,1)$ such that $e(s)\neq-\infty$, then $e(t)\neq-\infty$ for all $t\in(0,1)$. This, combined with the continuity of the $e_i$ for $i\in\N$ given by Theorem \ref{thm: displacement hessian of the entropy}, gives pointwise convergence of the sequence of functions $s\mapsto\frac{C}{2}s^2+f_k(s)$ to $s\mapsto\frac{C}{2}s^2+e(s)$, where the limit is is convex and real-valued. Note now that we are in the hypotheses of the proper case in Lemma \ref{lem: 7.3}. This result gives that $e'(s)=\lim_{k\to\infty}f_k'(s)$ pointwise $\mathcal{L}^1$-a.e. and $e''(s)=\lim_{k\to\infty}f_k''(s)$ distributionally on $(0,1)$. Hence, in this case, \eqref{eq: e''} follows by summing $e''_i$ using Lebesgue's monotone convergence theorem, and the pointwise lower bound establishes above shows that the integrand is greater than $-C$. Integrating the expression for $e''$ yields \eqref{eq: e'} exactly as it did in Theorem \ref{thm: displacement hessian of the entropy}, which concludes the proof subject to proving the Claim.

            \textit{Proof of the Claim.} Assume that there is some $s\in(0,1)$ such that $e(s)=-\infty$. Denote by $\rho_s:=d\mu_s/d\mathfrak{m}$ the Radon-Nikodym derivative of $\mu_s$ with respect to the weighted Lorentzian volume $\mathfrak{m}:=e^{-V}\operatorname{vol}_g$. Set $N^i:=(\operatorname{supp}\pi^i)\setminus\operatorname{sing}(\ell)$. The sets $N^i$ describe the supports of $\pi^i$ without singular points, and $\pi^i$ describes a coupling between $\mu_0^i$ and $\mu_1^i$, which are part of the decompositions of $\mu_0$ and $\mu_1$, respectively. Consider $N^{\infty}=\cup_{i\in\N}N^i$. By construction, this set is disjoint from $\operatorname{sing}(\ell)$, and by Theorem \ref{thm: relaxing separation}(iii), $\ell>0$ holds $\pi$-a.e.\ Hence, the map $\bar z_s:N^{\infty}\to M$ is well-defined and smooth by Lemma \ref{lemma: selecting midpoints away from cut locus}, and it has an inverse map on $\bar z_s(N^{\infty})$ by Proposition \ref{Prop: Lipschitz continuous inverse map}.

            By Theorem \ref{thm: relaxing separation}, it holds that $\pi(\operatorname{sing}(\ell))=0$, therefore $N^{\infty}$ carries the full measure of $\pi$. This will be important to analyse $\mu_s$, since $\mu_s$ is given by $(\bar z_s)_{\#}\pi$. On the other hand, since we have that $\mu_s$ is absolutely continuous with respect to $\operatorname{vol}_g$ (and therefore $\mathfrak{m}$) by Theorem \ref{thm: relaxing separation}, it is also inner regular, and therefore there exists a $\sigma$-compact subset $U$ of $\{z\in \bar z_s(N^{\infty}):\ \rho_s(z)\leq 1\}$, which differs from the latter by a $\mu_s$-negligible set. Note that this is the case because $\mu_s$ is absolutely continuous and therefore its density is a measurable map, which makes the latter a measurable set. Take $S:=\bar z_s^{-1}(U)$, which is also $\sigma$-compact, since $\bar z_s^{-1}$ is smooth in $U$.

            The idea of the proof is, for a given $t\in(0,1)$, to try to bring the set $S$ into $\operatorname{supp}\mu_t$, and study its influence over the entropy functional at $\mu_t$. The way of doing it will be via the maps $\bar z_r$, $r\in[0,1]$, which, since $\mu_r=(\bar z_r)_{\#}\pi$, allow to map subsets of the supports of the measures to a different support. Ultimately, we will see that in $\bar z_t(S)$, the entropy is also $-\infty$, so ultimately this will lead to $e(t)=-\infty$, as desired.

            Let $\bar{\pi}^i$ denote the restriction of $\pi^i$ to $S$ with the convention that $\bar{\pi}^{\infty}:=\pi$. Let $\bar{\mu}^i_r:=(z_r)_{\#}\bar{\pi}^i$ and $\bar{\nu}^i_r:=\sum_{j=1}^i\bar{\mu}^j_r$ for each $i\in\N\cup\{\infty\}$ and $r\in[0,1]$. Denote their entropies by $\bar{e}_i(r):=E_V(\bar{\mu}^i_r)$ and $\bar{f}_i(r):=E_V(\bar{\nu}^i_r)$. It is clear that $0\leq\bar{\mu}^i_r\leq\mu^i_r$, and they inherit absolute continuity and mutual singularity from $\{\mu^i_r\}_{i\in\N}$, so $\bar{f}_k(r)=\sum_{i=1}^k \bar{e}_i(r)$.

            Set $\bar{\rho}^t_i:=d\bar{\mu}^i_t/d\mathfrak{m}$ and $\rho_t^{\infty}:=\rho_t$. We want to establish a relationship between the densities $\bar{\rho}^i_t$ and $\rho^i_t$ in terms of $S$, since the entropy functional ultimately relies on these.

            Since $\mu_s$ is absolutely continuous with respect to the volume measure, by Lebesgue's density theorem, $U=\bar z_s(S)$ has either full or zero Lebesgue density with respect to $\mu_s$, $\mu_s$-a.e. Furthermore, since by Proposition \ref{Prop: Lipschitz continuous inverse map}, $\bar z_t$ is countably bi-Lipschitz on $S$, and $\mu_s=(\bar z_s)_{\#}\pi$, it follows that $S$ has either zero or full $n$-dimensional density $\pi$-a.e. in $N^{\infty}$.  
            Similarly, it follows that $\bar z_t(S)\subseteq\operatorname{supp}\mu_t$ has either full or zero Lebesgue density $\mu_t:=(\bar z_t)_{\#}\pi$-a.e. for each $t\in(0,1)$. This establishes
            \begin{equation}\label{eq: bar rho}
                \bar{\rho}^i_t:=\frac{d\bar{\mu}^i_t}{d\mathfrak{m}}=1_{z_t(S)}\rho^i_t,\quad\forall t\in(0,1)\quad\text{and}\quad i\in\N\cup\{\infty\}.
            \end{equation}
            On the other hand, we can estimate
            \begin{align*}
                -\infty=e(s)&=\int_M \rho_s\log\rho_s\, d\mathfrak{m} \\
                &\geq\int_{\{\rho_s\leq1\}}\rho_s\log\rho_s\, d\mathfrak{m} \\
                &=\sum_{i\in\N}\int_{z_s(S)}\rho_s^i\log\rho^i_s\, d\mathfrak{m} \\
                &=\lim_{k\to\infty}\bar{f}_k(s),
            \end{align*}
            where the equality in the third line is a consequence of our choice of $U$ and the mutual singularity of $\bar{\mu}^i_t\leq\mu^i_t$, and the last equality is due to the absolute continuity of the measures $\mu^i_t$ and \eqref{eq: bar rho}. In particular, this yields the immediate claim: if $e(s)=-\infty$, then $\lim_{k\to\infty}\bar{f}_k(s)=-\infty$.

            The goal now is to extend this conclusion to the whole interval $(0,1)$. The idea is to again employ Lemma \ref{lem: 7.3} in a similar fashion to the main part of this proof. Let $\hat{\bar{\pi}}^i:=\frac{1}{\bar{\pi}^i(M)}\bar{\pi}^i$ and $\hat{\bar{\mu}}_t^i:=\frac{1}{\bar{\mu}_t^i(M)}\bar{\mu}_t^i$ be the normalisations of $\bar{\pi}^i$ and $\bar{\mu}_t^i$, respectively. Recall that by Theorem \ref{thm: relaxing separation}, each $\pi^i$ is $\ell_u$-optimal for $(\mu^i_0,\mu^i_1)$, $i\in\N$, and therefore, since $\bar{\pi}^i\leq\pi^i$, and using a similar argument as in the proof of Theorem \ref{thm: relaxing separation}(iii), $\hat{\bar{\pi}}^i$ inherits $\ell_u$-optimality from $\hat{\pi}^i$ for its marginals $(\hat{\bar{\mu}}_0^i,\hat{\bar{\mu}}_1^i)$. Build the $u$-geodesic joining these two marginals, denoted by $(\hat{\pi}^i_t)_{t\in[0,1]}$. By Theorem \ref{thm: displacement hessian of the entropy}, the functional $t\mapsto\frac{C}{2}t^2+\bar{f}_i(t)$ is convex in $[0,1]$, after using \eqref{eq: scaling law entropy}. By Lemma \ref{lem: 7.2}(i), $\bar{f}_i(0)$ and $\bar{f}_i(1)$ are bounded above in terms of $\mu$, $\nu$ and $C$. Establishing $\bar{f}(t):=\limsup_{i\to\infty}\bar{f}_i(t)$, since $e(s)=-\infty$, we have $f(s)=-\infty$, and Lemma \ref{lem: 7.3} yields $\sup_{t\in(0,1)}\bar{f}(t)=-\infty$, as desired.

            To conclude,
            \begin{align*}
                e(t)&=\int_M\rho_t\log\rho_t\, d\mathfrak{m}\\
                &=\int_{z_t(S)}\bar{\rho}_t\log\bar{\rho}_t\, d\mathfrak{m} +\int_{M\setminus z_t(S)}\rho_t\log\rho_t\, d\mathfrak{m}.
            \end{align*}
            The first summand coincides with $\bar{f}_{\infty}(t):=E_V(\bar{\mu}_t^{\infty})$, which diverges to $-\infty$ by the assertion in the last paragraph and Lemma \ref{lem: 7.2}(ii). Thus $e(t)=-\infty$ by our infinity convention, which concludes the proof.
        \end{proof}

         \begin{Theorem}[Weak convexity from timelike lower Ricci curvature bounds]
         \label{Theorem: Weak convexity from Ricci}
        Let $(M,g)$ be a globally hyperbolic spacetime. Fix $V\in C^2(M)$ bounded, $N \in [n,\infty]$, and $u:(0,\infty)\to\R$ admissible, by convention $V = 0$ if $N = n$. If $\operatorname{Ric}^{(N,V)}(v,v)\geq K|v|^2_g\geq0$ holds in every timelike direction $(x,v)\in TM$, then the relative entropy $E_V(\mu)$ is weakly $(K,N,u)$-convex for the set $Q\subseteq\mathcal{P}^{ac}(M)^2$ of measures $(\mu,\nu)$ such that $\lambda:=\ell_u(\mu,\nu)\in(0,\infty)$, and there exist lower semicontinuous functions $a,b:M \to \R$ with $a \in L^1(\mu)$, $b \in L^1(\nu)$ such that $u_{\lambda} \circ \ell \leq a \oplus b$ on $\mathrm{supp}(\mu \times \nu)$ and there exists an optimal coupling $\pi\in\Pi^{u_{\lambda}}_{\leq}(\mu,\nu)$ with $\ell>0$ holding $\pi$-a.e.\ which satisfies $\int u_{\lambda} \circ \ell \, d\pi = u(1)$.
    \end{Theorem}
    \begin{proof}
        The strategy to prove this result from Theorem \ref{thm: displacement hessian of the entropy ii} is analogous to the one used to prove Corollary \ref{cor: entropic convexity from ricci bounds} from Theorem \ref{thm: displacement hessian of the entropy}. Note that the required hypothesis \eqref{eq: constant c} of Theorem \ref{thm: displacement hessian of the entropy ii} holds, since our assumptions on the Ricci curvature guarantee that $C=0$ works.

        Note, however, that in this case we do not have compactly supported measures, so we do not have a priori \eqref{eq:previous bound}. In particular, for the $u$-geodesic $(\mu_s)_{s\in[0,1]}$ obtained in Theorem \ref{thm: displacement hessian of the entropy ii}, we cannot conclude $e(s):=E_V(\mu_s)>-\infty$, which was a key ingredient in the proof. Nonetheless, note that in the case where $\max\{e(0),e(1)\}<\infty$, Theorem \ref{thm: displacement hessian of the entropy ii} asserts that $e(\cdot)$ is semiconvex, and therefore, $e(s)$ is real-valued unless $\sup_{t\in(0,1)}e(t)=-\infty$. In the first case, the proof developed in Corollary \ref{cor: entropic convexity from ricci bounds} can be used in this context, whereas the second case is already contained in Definition \ref{def: (k,n,u)-convexity}. If $\max\{e(0),e(1)\}=\infty$, we can apply the preceding argument on any subinterval $[t_0,t_1]\subset[0,1]$ satisfying $\max\{e(t_0),e(t_1)\}<\infty$ to conclude that $e(s)$ is real-valued, convex, and satisfies the desired estimates on $[t_0,t_1]$ unless $(t_0,t_1)\subset e^{-1}(-\infty)$.
    \end{proof}

    \begin{Remark}[Comparison with Section $7$ of \cite{McCann:2020}]
    \label{Remark: comparing relaxation with McCann}
    In the results on relaxation of $p$-separation in \cite[Sec.\ 7]{McCann:2020}, McCann assumes finiteness and attainment in the Kantorovich dual problem for $\ell_p$, which produces the dual potentials $(a,b)$ such that $u_p \circ \ell \leq a\oplus b$ on $\mathrm{supp}(\mu \times \nu)$. By homogeneity, for any $\lambda > 0$, multiplying this equation with $\lambda^{-p}$ gives $(u_p)_{\lambda} \circ \ell \leq \lambda^{-p}a \oplus b =: a' \oplus b'$, so that McCann's assumptions in \cite[Thm.\ 7.4, Cor.\ 7.5]{McCann:2020} imply our assumptions in Theorems \ref{thm: displacement hessian of the entropy ii} and \ref{Theorem: Weak convexity from Ricci} in the case $u = u_p$, not just for $p \in (0,1)$ but also for $p < 0$ (by the same homogeneity argument given above).
    \end{Remark}

    \subsection{Ricci curvature lower bounds from entropic convexity}\label{Subsection: Ricci from entropic convexity}

    We conclude our efforts concerning the relationship between timelike Ricci bounds and entropic convexity in this subsection, showing that the convexity of the relative entropy along $u$-geodesics implies a bound below on the timelike Ricci curvature.

\begin{Lemma}\label{lem: prescribing}
    Let $u:(0,\infty)\to\R$ be an admissible function, and fix a compact set $X\times Y\subset M^2\setminus\operatorname{sing}(\ell)$ with $(\bar{x},\bar{y})$ in its interior. Let $\varphi:X\to\R$ satisfying the first- and second-order conditions
    \begin{align*}
        D\varphi(\bar{x})=D_x(u\circ\ell)(\bar{x},\bar{y}),\quad\text{and}\quad D^2\varphi(\bar{x})>D^2_{xx}(u\circ\ell)(\bar{x},\bar{y}).
    \end{align*}
    Then there exists a $u\circ\ell$-convex function $\tilde{\varphi}$ on $X$ which agrees with $\varphi$ in some neighbourhood of $\bar{x}$.
\end{Lemma}

\begin{proof}
    The result follows using the same arguments as McCann in {\cite[Lem.\ 8.3]{McCann:2020}}, observing that $b:=u\circ\ell$ has the same properties as $u_p \circ \ell$, $p\in(0,1)$, by Corollary \ref{cor: twist and non degeneracy}(iv).
\end{proof}

 \begin{Theorem}[Entropic convexity implies a timelike Ricci bound]
        Let $(M^n,g)$ be a globally hyperbolic spacetime. Fix $V\in C^2(M)$ and $K\in\R$ and $N \in [n,\infty]$. Fix $V=0$ if $N=n$. If $\operatorname{Ric}^{(N,V)}(v,v)\geq K|v|^2_g$ fails at some timelike vector $(x,v)\in TM$, then the relative entropy $E_V(\mu)$ fails to be weakly $(K,N,u)$-convex for any admissible function $u:(0,\infty)\to\R$.

        Moreover, the $u$-geodesic $(\mu_s)_{s\in[0,1]}$ along which $e(s):=E_V(\mu_s)$ is not $(\ell_u(\mu_0,\mu_1)^2K,N)$-convex may be constructed so that $e\in C^2([0,1])$, and $\operatorname{supp}(\mu_0\times\mu_1)$ is disjoint from $\{\ell\leq0\}$, but contained in an arbitrarily small neighbourhood of $(x,x)$. 
    \end{Theorem}

    \begin{proof}

    Assume there exist a point $\bar{x}\in M$ and a timelike tangent vector at $\bar{x}$, $\hat{v}\in T_{\bar{x}}M$ with $\vert\hat{v}\vert_g=1$, for which $\operatorname{Ric}^{(N,V)}(\hat{v},\hat{v})<K$. The argument follows similarly as in {\cite[Thm.\ 8.5]{McCann:2020}}, with suitable adaptations corresponding to our transport maps.

    The idea is to construct a $u$-geodesic which transports some conveniently chosen initial measure $\mu_0$ in the direction of $\hat{v}$. By building the geodesic appropriately, which will be possible by fixing first the Kantorovich potentials, we will be able to see that the associated entropy functional fails to be convex. Since we will also argue that this $u$-geodesic that we will build is unique, we can conclude that the relative entropy is not even weakly convex.

    For $r>0$, consider the Riemannian ball of radius $r$ centred around $\bar{x}$, 
    $B_r(\bar{x})$, and take $\mu_0^{(r)}$ to be the uniform distribution on this ball with respect to $\operatorname{vol}_g$. We have that $\mu_0^{(r)}\to\delta_{\bar{x}}$ against test functions as $r\to0$. Moreover, fix $t\in(0,1]$ sufficiently small so that $t$ does not exceed the injectivity radius of the exponential map. Define $y_t:=\exp_{\bar{x}}t\,\hat{v}$. By construction, $y_t$ lies in the future of $\bar{x}$ away from the cut locus, so there is a compact neighbourhood $X\times Y$ of $(\bar{x},y_t)$ disjoint from $\operatorname{sing}(\ell)$.

    Let us build a suitable optimal transport problem starting from $\mu_0^{(r)}$ making use of the timelike vector $v_t:=t\,\hat{v}$. This will be possible by prescribing a conveniently chosen Kantorovich potential $\varphi=\left(\varphi^{(u_t\circ\ell)}\right)^{(u_t\circ\ell)}$ defined on $B_r(\bar{x})$. Assuming that we have such a $\varphi$, consider $F_s(x):=\exp_{x} v\,(D\varphi(x),x;u_t)$ as described in Lemma \ref{lem: maps and their jacobian der}, where $v(D\varphi(x),x;u_t) \equiv DH(D\varphi, x;(u_t)^*)$. This map describes the unique $u$-geodesic from $\mu_0^{(r)}$ to $\mu_1^{(r)}:=F_{1\#}\mu_0^{(r)}$ as follows. Note that for $r>0$ sufficiently small, $B_r(\bar{x})\times F_1(B_r(\bar{x}))\subset X\times Y$, and therefore disjoint from $\operatorname{sing}(\ell)$.\ Using Remark \ref{Remark: simple u separation} and Theorem \ref{thm: characterising opt maps}, we obtain that $\mu_s^{(r)}:=F_{s\#}\mu_0^{(r)}\in\mathcal{P}^{ac}(M)$ defines the unique $u$-geodesic on $s\in[0,1]$ connecting its endpoints.
    
    Given that we expect to obtain the failure of the entropic convexity from the failure of the bound on the Ricci curvature, we need to study the first and second derivatives of the entropy functional. Since the contruction is focused around $\bar{x}$ and $\hat{v}$, it is reasonable to worry only about this convexity at $s=0$, as well as at the limit $r\to0$. Hence, by Theorem \ref{thm: displacement hessian of the entropy}, 
    \begin{equation*}     
        \begin{aligned}
             e'(0;r)&=\int_M [DV_{F_s(x)}F_s'(x)-\operatorname{Tr}B_s(x)]\Big\vert_{s=0}d\mu_0^{(r)}(x) \\ 
             & =\int_M [DV_{F_0(x)}F_0'(x)-\operatorname{Tr}B_0(x)]\, d\mu_0^{(r)}(x) \\
             &=\int_M [DV_x\, v(D\varphi(x),x;u_t)-\operatorname{Tr}(D_x\,v(D\varphi(x),x;u_t)\,\tilde{D}^2\varphi(x))]\, d\mu_0^{(r)}(x) \\
             & \qquad\to DV(\bar{x})\,v(D\varphi(\bar{x}),\bar{x};u_t)-\operatorname{Tr}\left(Dv(D\varphi(\bar{x}))\tilde{D}^2\varphi(\bar{x})\right) \quad \text{as}\ r\to0,
        \end{aligned}
    \end{equation*}
    and
    \begin{equation*}
        \begin{aligned}
            e''(0:r)&=\int_M [\operatorname{Tr}B_s^2(x)+(\operatorname{Ric}+D^2V)_{F_s(x)}(F_s'(x),F_s'(x))]\Big\vert_{s=0}d\mu_0^{(r)}(x) \\
            & =\int_M \Big[\operatorname{Tr}\left((Dv(D\varphi(x),x;u_t)\,\tilde{D}^2\varphi(x))^2 \right)\\
            &\qquad\qquad+(\operatorname{Ric}+D^2V)_x(v(D\varphi(x),x;u_t),v(D\varphi(x),x;u_t))\Big]\, d\mu_0^{(r)}(x) \\
            &\qquad\to \operatorname{Tr}\left((Dv(D\varphi(\bar{x}),\bar{x};u_t)\,\tilde{D}^2\varphi(\bar{x}))^2 \right)\\
            &\qquad\qquad+(\operatorname{Ric}+D^2V)_{\bar{x}}(v(D\varphi(\bar{x}),\bar{x};u_t),v(D\varphi(\bar{x}),\bar{x};u_t)) \qquad \text{as}\ r\to0.
        \end{aligned}
    \end{equation*}

        Using Lemma \ref{lem: prescribing}, prescribe
        \begin{equation*}
        \begin{split}
                  v\left(D\varphi(\bar{x}),\bar{x};u_t\right)=v_t, \qquad \text{and} \\
           \tilde{D}^2\varphi(\bar{x})=-\frac{1}{N-n}\left(DV(\bar{x})v_t\right)(Dv(D\varphi(\bar{x})))^{-1}.
        \end{split}
        \end{equation*}

        With these choices,
         \begin{equation*}
            \begin{aligned}
                e'(0;r)& \to DV(\bar{x})v_t -\operatorname{Tr}\left(-\frac{1}{N-n}DV(\bar{x})v_t\right)\\
                & =DV(\bar{x})v_t+n\frac{1}{N-n}DV(\bar{x})v_t,
            \end{aligned}
        \end{equation*}
        and
        \begin{equation*}
            \begin{aligned}
                e''(0;r)&\to \operatorname{Tr}\left(\left(-\frac{1}{N-n}DV(\bar{x})v_t\right)^2\right)+(\operatorname{Ric}+D^2V)_{\bar{x}}(v_t,v_t)\\
                &=\frac{n}{(N-n)^2}(DV(\bar{x})v_t)^2+\operatorname{Ric}(v_t,v_t)+D^2V(\bar{x})(v_t,v_t)\\
                &=\frac{n}{(N-n)^2}(DV(\bar{x})v_t)^2+\operatorname{Ric}^{(N,V)}(v_t,v_t)+\frac{1}{N-n}(DV(\bar{x})v_t)^2.
            \end{aligned}
        \end{equation*}
        Finally,
        \begin{equation*}
            \lim_{r\to0} e''(0;r)-\frac{1}{N}e'(0;r)^2=\operatorname{Ric}^{(N,V)}(v_t,v_t)<K\vert v_t\vert^2_{g}=K\lim_{r\to0}\ell_u\left(\mu_0^{(r)},\mu_1^{(r)}\right)^2,
        \end{equation*}
        which gives that there exists $r>0$ sufficiently small for which $e''(0;r)-\frac{1}{N}e'(0;r)^2<K\ell_u\left(\mu_0^{(r)},\mu_1^{(r)}\right)^2$, contradicting the $\left(K\ell_u\left(\mu_0^{(r)},\mu_1^{(r)}\right)^2,N\right)$-convexity of $e(s;r)$ on $[0,1)$, which yields the proof.
    \end{proof}

\section{Outlook}

Our work opens up the possibility of investigations in several different directions. First, one could study the connection between timelike Ricci curvature \emph{upper} bounds as considered by Mondino--Suhr \cite{MS:22} in the general Orlicz setting introduced here, as well as broader ranges of synthetic dimensions $N$ and the convexity of Rényi entropies in the Lorentz--Finsler context (cf.\ Braun--Ohta \cite{BO:23}). Moreover, natural questions of stability of concepts we studied in this work arise, such as optimal couplings, Kantorovich potentials, or the $u$-separation condition, under suitable perturbations of the function $u$.

Second, just like the results of McCann \cite{McCann:2020} and Mondino--Suhr \cite{MS:22}, our work motivates a study of nonsmooth Lorentzian spaces satisfying synthetic Ricci curvature bounds below, which we may term $\mathrm{TCD}_u$ (timelike curvature-dimension) or $\mathrm{ TMCP}_u$ (timelike measure contraction property) spaces. One may ask whether these conditions are stable, not just under suitable convergence of Lorentzian spaces, but also with respect to certain convergences $u_n \to u$ of admissible functions. Results which consider the stability under perturbations of $u$ could help generalise the characterisation of timelike Ricci lower bounds via entropic convexity, as presented in this work, to hold for a more general class of functions $u$, such as the interesting edge case $u(x) = u_1(x) = x$. Further topics worth looking into include a study of the degenerate elliptic $u$-d'Alembertian and relevant comparison results in analogy with \cite{Octet, braun2024dAlembertian}, as well as hyperbolic Orlicz spaces $L^u$ which should fit into Gigli's \cite{gigli2025hyperbolic} setting of hyperbolic Banach spaces.

\appendix 
\section{Global hyperbolicity and order-theoretic completeness}
\label{Appendix}

In this appendix, we clarify the relationship between certain notions of order-theoretic completeness recently studied for the chronological \cite{braun2026spacetime} and causal \cite{Octet, gigli2025hyperbolic} relations, and global hyperbolicity. Assuming closedness of the causal relation, these turn out to be equivalent notions for (finite-dimensional) smooth spacetimes, mirroring the well-known equivalence between completeness and properness for a Riemannian manifold as a consequence of the Hopf--Rinow theorem. Recall that $\tilde g$ is a background complete Riemannian metric on $M$.

\begin{Definition}[Causal forward/backward completeness]
A spacetime $(M,g)$ is called \emph{causally forward complete} if every sequence $(x_n) \subseteq M$ with $x_n \leq x_{n+1} \leq z$ for all $n \in \N$ and some $z \in M$, is convergent in $M$. Similarly, we call $(M,g)$ \emph{causally backward complete} if it is forward complete when considered with the opposite time orientation. Analogously, we say $(M,g)$ is \emph{chronologically forward complete} if every sequence $x_n \ll x_{n+1} \ll z$ converges, and \emph{chronologically backward complete} if its time-reversal is chronologically forward complete.
\end{Definition}

\begin{Lemma}[Chronological completeness = causal completeness]
\label{Lemma: chron comp = caus comp}
A spacetime $(M,g)$ is chronologically forward (resp.\ backward) complete if and only if it is causally forward (resp.\ backward) complete.
\end{Lemma}
\begin{proof}
    We only show the claim for the forward completeness notions. It is trivial that every causally forward complete spacetime is chronologically forward complete. For the converse direction, suppose $(M,g)$ is chronologically forward complete and let $p_n \leq p_{n+1} \leq z$ be a causally bounded ordered sequence. We claim that we can construct a sequence $(x_n)$ in such a way that $x_n \ll x_{n+1} \leq z$, $x_n \ll p_n$, $d_{\tilde g}(x_n,p_n) < n^{-1}$. Once we have done this, picking any $\tilde z \in I^+(z)$ shows the necessary chronological boundedness of $x_n$, from which convergence of $(x_n)$ and thus of $(p_n)$ follows, establishing causal forward completeness.

    To construct $x_n$, we proceed by induction. Pick $x_1 \in I^-(p_1)$ such that $d_{\tilde g} (x_1,p_1) < 1$. Suppose $x_1,\dots,x_n$ have been constructed, then by assumption $x_n \ll p_n \leq p_{n+1} \leq z$. By push-up, $x_n \ll p_{n+1}$, so we may pick $x_{n+1} \in I^+(x_n) \cap I^-(p_{n+1})$ in such a way that $d_{\tilde g}(x_{n+1},p_{n+1}) < (n+1)^{-1}$. Thus we have constructed the desired sequence $(x_n)$.
\end{proof}

\begin{Lemma}[Complete $\Rightarrow$ causal]
\label{Lemma: Ordered complete -> causal}
    Let $(M,g)$ be either causally forward or causally backward complete. Then it is causal, i.e.\ there are no closed causal curves.
\end{Lemma}
\begin{proof}
    We argue via causal forward completeness (the argument using causal backward completeness is analogous). Arguing by contradiction, suppose $\gamma:[a,b] \to M$ is a nontrivial closed causal curve, so that $\gamma(a) = \gamma(b)$. So there exists $t_0 \in (a,b)$ such that $\gamma(t_0) \neq \gamma(a)$. Define a sequence $(x_n)$ via
    \begin{equation*}
        x_{2n}:=\gamma(a), \quad x_{2n+1}:=\gamma(t_0).
    \end{equation*}
    Then $x_n \leq x_{n+1} \leq \gamma(b)$, and since $\gamma(a) \neq \gamma(t_0)$, $(x_n)$ is not convergent, a contradiction.
\end{proof}

\begin{Theorem}[Equivalence of order-theoretic completeness and global hyperbolicity]
\label{Theorem: equivalence of order completeness and glob hyp}
Let $(M,g)$ be a smooth spacetime such that the causal relation is closed in $M^2$. Then the following are equivalent:
\begin{enumerate}
    \item $(M,g)$ is globally hyperbolic.
    \item $(M,g)$ is chronologically forward complete.
    \item $(M,g)$ is chronologically backward complete.
    \item $(M,g)$ is causally forward complete.
    \item $(M,g)$ is causally backward complete.
    
\end{enumerate}
\end{Theorem}
\begin{proof}
We have shown $(ii) \Leftrightarrow (iv)$ and $(iii) \Leftrightarrow (v)$ in Lemma \ref{Lemma: chron comp = caus comp}. We show now that $(i) \Leftrightarrow (iv)$, the equivalence $(i) \Leftrightarrow (v)$ is analogous.

$(i) \Rightarrow (iv):$ Suppose $(M,g)$ is globally hyperbolic, and let $p_n \leq p_{n+1} \leq y$. Letting $x:=p_1$, we have that $p_n \in J(x,y)$ for every $n$. By compactness of the latter, every subsequence of $(p_n)$ has a further subsequence that converges to some point in $J(x,y)$. All of these sublimits have to agree: Indeed, if $p_{n_k}$ and $p_{m_l}$ are two subsequences of $(p_n)$ with limits $p,q$, then for fixed $k$, for large enough $l$ we have that $p_{n_k} \leq p_{m_l}$, so taking $l \to \infty$ gives $p_{n_k} \leq q$, hence taking $k \to \infty$ we get $p \leq q$, which yields $p = q$ by symmetry. From this, it is easy to conclude that in fact the sequence $(p_n)$ itself has to converge, showing $(iv)$.

$(iv) \Rightarrow (i):$ By Lemma \ref{Lemma: Ordered complete -> causal}, $(M,g)$ is causal, and by assumption the causal relation is closed, so it suffices to show that causal diamonds $J(x,y)$ are bounded for every $x,y \in M$. Indeed, suppose not and take $p_n \in J(x,y)$ such that $d_{\tilde g}(x,p_n) > n$. Connect $x$ to $p_n$ by $\tilde g$-arclength parametrised causal curves $\gamma_n:[0,b_n] \to M$. We may apply the limit curve theorem (see e.g.\ Minguzzi \cite[Thm.\ 3.1]{MinguzziLimitCurve}) to conclude the subsequential existence of an inextendible causal curve $\gamma:[0,\infty) \to M$ starting at $x$. Since $\gamma$ is the locally $\tilde g$-uniform limit of the $\gamma_n$, all of which lie in $J(x,y)$, by closedness of the latter also $\gamma([0,\infty)) \subseteq J(x,y)$. Defining $z_n:=\gamma(n)$ gives a sequence $z_n \leq z_{n+1} \leq y$ which does not converge, contradicting $(iv)$.
\end{proof}

It would be interesting to examine whether the closedness assumption on the causal order can be removed from the above result, and if not, to construct an example demostrating this.

\begin{Remark}[On the generality of the equivalence]
The only tool we used to prove the equivalence of chronological/causal forward/backward completeness and global hyperbolicity in this section was the limit curve theorem for inextendible causal curves, as well as approximability by points in their timelike futures/pasts in conjunction with the openness of the timelike relation. In particular, Theorem \ref{Theorem: equivalence of order completeness and glob hyp} holds verbatim for Lorentz-Finsler spacetimes. Moreover, in the generality of Lorentzian pre-length spaces as studied by Kunzinger-Sämann \cite{KS:18}, its proof yields the equivalence of compactness of causal diamonds with chronological/causal forward/backward completeness provided the Lorentzian pre-length space is causally closed, localisable, and its background positive definite metric is proper (cf.\ \cite[Thm.\ 3.14]{KS:18}). It is unclear in this generality whether causality and compactness of causal diamonds are enough to recover global hyperbolicity in the sense of \cite[Def.\ 2.35(v)]{KS:18}. 
\end{Remark}

\section*{Data availability statement}

There is no data associated with this manuscript.

\section*{Conflict of interest statement}

The authors confirm that no conflict of interest exists for this work.

\section*{Acknowledgments}
The authors would like to thank Mathias Braun and Mauricio Che for guiding them to the literature on Orlicz--Wasserstein spaces. They are also grateful to Stefan Suhr for fruitful discussions and his ideas regarding the content of Appendix \ref{Appendix}, and to Matteo Calisti and Alec Metsch for helpful discussions concerning the majority of the paper and for their comments on preliminary versions of this manuscript. Furthermore, they are thankful for the hospitality of the Bernoulli Centre at EPFL during the workshop ``Optimal Transport and Metric Geometry across Structures" in October 2025. 

Marta Sálamo Candal acknowledges the support of the Vienna School of Mathematics. 

This research was funded in part by the Austrian Science Fund (FWF) [Grant DOI\\10.55776/EFP6 and 10.55776/J4913]. For open access purposes, the authors have applied a CC BY public copyright license to any author accepted manuscript version arising from this submission.

\addcontentsline{toc}{section}{References}
\bibliography{Bibliography} 
\bibliographystyle{acm}

\end{document}